\documentclass[11pt]{amsart}

\usepackage{amsmath, amsthm, amssymb}
\input xy
\xyoption{all}
\usepackage{mathrsfs}
\usepackage{enumerate}
\usepackage{color}
\usepackage{tikz-cd}
\usetikzlibrary{matrix}
\usepackage{caption}
\usepackage{subcaption}
\usepackage{tasks}
\usepackage{relsize}
\usepackage[shortlabels]{enumitem}
\usepackage[numbers, sort & compress]{natbib}
\usepackage{appendix}

\newtheorem{dummy}{}[section]
\newtheorem{theorem}[dummy]{Theorem}
\newtheorem{proposition}[dummy]{Proposition}
\newtheorem{lemma}[dummy]{Lemma}
\newtheorem{corollary}[dummy]{Corollary}

\theoremstyle{definition}
\newtheorem{definition}[dummy]{Definition}

\newtheorem{remark}[dummy]{Remark}

\newtheorem{condition}[dummy]{Condition}

\newcommand{\Z}{\ensuremath{\mathbb{Z}}}
\newcommand{\Q}{\ensuremath{\mathbb{Q}}}
\newcommand{\C}{\ensuremath{\mathbb{C}}}

\newcommand{\vir}{\mathrm{vir}}

\newcommand{\ovir}{\ca{O}^{\vir}}
\renewcommand{\tilde}{\widetilde}

\usepackage[linktoc = all]{hyperref}
\hypersetup{
  colorlinks,
  citecolor = black,
  filecolor = black,
  linkcolor = black,
  urlcolor = black
}

\usepackage{geometry}
\usepackage{xcolor}
\newenvironment{Comment}[2]{\noindent \color{#1}{\texttt #2: }}{\par \noindent}
\newcommand{\Ming}[1]{\begin{Comment}{red}{Ming}#1 \end{Comment}}
\newcommand{\Yang}[1]{\begin{Comment}{olive}{Yang}#1 \end{Comment}}
\newcommand{\Mingg}[1]{\begin{Comment}{blue}{Ming}#1 \end{Comment}}

\newcommand{\git}{/ \! \!/}

\newcommand{\ChYang}[1]{{\color{olive} #1}}

\newcommand{\mb}{\mathbf}
\newcommand{\bb}{\mathbb}
\newcommand{\ca}{\mathcal}
\newcommand{\fr}{\mathfrak}

\newcommand{\vb}{\vec{\beta}}
\newcommand{\ub}{\underline{\beta}}
\newcommand{\fbi}{f_{\vec{\beta}_{(i)}}}

\newcommand{\ev}{\underline{\mathrm{ev}}}

\title{$K$-theoretic quasimap wall-crossing}

\author[M.~Zhang]{Ming Zhang}
\address{Department of Mathematics, the University of British Columbia}
\email{zhangming@math.ubc.ca}
\author[Y.~Zhou]{Yang Zhou}
\thanks{The second author is supported by the Simons Collaboration Grant for Mathematicians and the Center
  of Mathematical Sciences and Applications, Harvard University.}
\address{Center of Mathematical Sciences and Applications, Harvard University}
\email{yangzhou@cmsa.fas.harvard.edu}

\begin{document}

%%% Version 1%%%%
\maketitle
\begin{abstract}
  In this paper, we prove a $K$-theoretic wall-crossing formula for
  $\epsilon$-stable quasimaps for all GIT targets in all genera. It recovers the
  genus-0 $K$-theoretic toric mirror theorem by Givental--Tonita
  \cite{Givental-Tonita} and Givental \cite{givental14, givental15}, and the
  genus-0 mirror theorem for quantum $K$-theory with level structure by
  Ruan-Zhang
  \cite{RZ1}. The proofs are based on $K$-theoretic virtual localization on the
  master space introduced by the second author in \cite{Zhou2}.
\end{abstract}

\tableofcontents
\addtocontents{toc}{\protect \setcounter{tocdepth}{1}}

\section{Introduction}
\subsection{Overview}
Quantum $K$-theory was introduced by Givental \cite{WDVV} and Y.P. Lee
\cite{Lee} as the $K$-theoretic version of Gromov-Witten theory. Let $X$ be a
smooth Deligne--Mumford stack with projective coarse moduli over $\bb{C}$. The
classical Gromov-Witten invariants of $X$ are defined as integrals of cohomology
classes over the moduli space $\overline{M}_{g, n}(X, d)$ of stable maps. In
quantum $K$-theory, the basic objects are vector bundles and coherent sheaves.
Quantum $K$-invariants of $X$ are defined as holomorphic Euler characteristics
of natural $K$-theory classes over $\overline{M}_{g, n}(X, d)$. Recently there
has been increased interest in studying quantum $K$-invariants due to their
connections to 3d gauge theories \cite{Jockers-Mayr1, Jockers-Mayr2,
  Jockers-Mayr3, Ueda-Yoshida} and representation theory \cite{Okounkov,
  Aganagic-Okounkov-1, Aganagic-Okounkov-2, PSZA, KPSZ, RZ2}. For example, the
level structure defined in \cite{RZ1} is related to Chern-Simon levels in 3d
$\ca{N} = 2$ gauge theory (c.f. \cite{Jockers-Mayr1, Jockers-Mayr3,
  Ueda-Yoshida}) and the space of conformal blocks (c.f. \cite{RZ2}).

As shown in \cite{WDVV, Lee}, the WDVV equation and most of Kontsevich-Manin
axioms hold in quantum $K$-theory. In this sense, quantum $K$-theory has similar
structures to Gromov-Witten theory. Moreover, it is shown in
\cite{Givental-Tonita}, \cite{Tonita1} and \cite{givental19} that quantum
$K$-invariants and their permutation-equivariant versions are determined by
cohomological Gromov-Witten invariants of the same manifold or orbifold.
However, the precise relationship between $K$-theoretic and cohomological
invariants is sophisticated. In general, Gromov-Witten invariants are rational
numbers, while their $K$-theoretic counterparts are always integers. To relate
them, one needs to apply the (virtual) Kawasaki-Riemann-Roch formula
\cite{Tonita3} to express holomorphic Euler characteristics of coherent sheaves
on $\overline{M}_{g, n}(X, d)$ in terms of intersection numbers over the
corresponding \emph{inertia stack} $I \, \overline{M}_{g, n}(X, d)$. The
complexity comes from the combinatorics of the strata of the inertia stack.

Unlike cohomological
Gromov-Witten theory, the computations in quantum $K$-theory have been rather
rare. One important way to compute genus-0 quantum $K$-invariants is by using
$K$-theoretic mirror theorems, see e.g., \cite{Tonita2}, \cite{Jockers-Mayr2},
and \cite[Section 6.3]{Jockers-Mayr1}. To established the general mirror
theorems in the $K$-theory setting, Givental introduced and studied a variant of
quantum $K$-invariants in a series of papers \cite{givental11, givental12,
  givental13, givental14, givental15, givental16, givental17, givental18,
  givental19, givental21, givental22}. These invariants are called
\emph{permutation-equivariant} quantum $K$-invariants. In the toric case, the
mirror theorem \cite{Givental-Tonita, givental14, givental15} states that up to
a change of variable (mirror map), a specific generating series ($J$-function)
of genus-0 permutation-equivariant quantum $K$-invariants with descendants can
be identified with an explicit $q$-hypergeometric series ($I$-function). Their
proofs are based on the so-called adelelic characterization of the range of the
$J$-function. This strategy has been generalized to quantum $K$-theory with
level structure in \cite{RZ1}.

In cohomological Gromov-Witten theory, the most general mirror theorem is
obtained by using the wall-crossing strategy. Note that the moduli space
$\overline{M}_{g, n}(X, d)$ of stable maps is a compactification of the space of
algebraic maps from smooth curves to $X$. In many cases, there are other natural
and simpler compactifications. For a large class of GIT quotients of affine
varieties, a family of compactifications has been constructed in \cite{kim1,
  Kim4}, unifying and generalizing many previous constructions \cite{mustata,
  kim5, MOP, toda}. These new theories are called $\epsilon$-stable quasimap
theories, where $\epsilon \in \bb{Q}_{>0}$ is the stability parameter. Examples
of those GIT targets include complete intersections in toric Deligne--Mumford
stacks, flag varieties of classical types, zero loci of
sections of homogeneous bundles, and Nakajima quiver varieties.

The space of stability parameters have wall-and-chamber structure. Once the
degree of quasimaps is fixed, there are only finitely many walls. There are two
extreme chambers corresponding to $\epsilon$ being sufficiently large or
sufficiently close to 0. We denote them by $\epsilon = \infty$ and $\epsilon =
0 + $,
respectively. For $\epsilon = \infty$, the notion of $\epsilon$-stable
quasimaps
coincides with that of stable maps and therefore one obtains the
Gromov–Witten
theory. Roughly speaking, as $\epsilon$ decreases, domain curves of
$\epsilon$-stable quasimaps have fewer and
  fewer rational tails. Hence the moduli
space $Q^\epsilon_{g, n}(X, \beta)$ of $\epsilon$-stable quasimaps becomes
simpler
as $\epsilon$ approaches 0. Similar to $\overline{M}_{g, n}(X, d)$, all
$\epsilon$-stable quasimap spaces have canonical perfect obstruction theories
and hence are equipped with virtual fundamental cycles and virtual structure
sheaves. This allows us to define cohomological and $K$-theoretic quasimap
invariants, which depend on the stability parameter $\epsilon$. In some cases,
the $(\epsilon = 0 +) $-stable quasimap invariants are easier to compute. A
\emph{wall-crossing} formula is a relation between invariants from different
stability chambers. It enables us to recover Gromov-Witten invariants from
$(\epsilon = 0 +) $-stable quasimap invariants. This idea has been applied
successfully in cohomological Gromov-Witten theory, see, for example,
\cite{Lho-Pandharipande, Kim-Lho, Guo-Janda-Ruan}. In particular, the genus-0
mirror theorem can be deduced from such wall-crossing formulas.

The cohomological quasimap wall-crossing formulas have been proved in various
generalities in \cite{kim2, Kim4, Kim6, Clader-Janda-Ruan} and the most general
version is proved by the second author in \cite{Zhou2}.
The goal of this paper is to generalize the analysis in \cite{Zhou2} to the
$K$-theory setting and establish wall-crossing formulas in $K$-theoretic
quasimap theory for all GIT targets in all genera. These formulas also work in
the presence of twisting \cite{Tonita2, givental22} and level structure
\cite{RZ1}. In genus zero, our formulas implies the $K$-theoretic toric mirror
theorems proved in \cite{Givental-Tonita, givental14, givental15, RZ1} and the
wall-crossing formula in \cite{Tseng-You}

\subsection{Wall-crossing formulas}
Let $W$ be an affine variety with a right action of a reductive group $G$. We
assume that $W$ has at worst local complete intersection singularities. Let
$\theta$ be a character of $G$ such that the $\theta$-stable locus
$W^{s}(\theta)$ is smooth, nonempty, and coincides with the
$\theta$-semistable locus $W^{ss}(\theta)$. In this paper, we consider the
``stacky'' GIT quotient
\[
  X = [W^{ss}(\theta)/G].
\]
Let $\underline{X}$ denote the coarse moduli of $X$ and let $W \git_{0} G =
\operatorname{Spec} H^0(W, \mathcal O_W)^G$ denote the affine quotient. The GIT set-up gives (see \cite[\textsection 1.2]{Zhou2} and
\cite[\textsection{3.1}]{kim2} for details) a morphism
\begin{equation}\label{eq:equation-GIT-embedding}
[W/G]\rightarrow[\C^{N+1}/\C^*]
\end{equation}
for some $N\in\Z_{>0}$, inducing a closed embedding
$\underline{X}\rightarrow \bb{P}^N\times W\git_0 G$.
Hence $X$ is a smooth proper Deligne--Mumford stack over the affine
quotient
$W \git_{0} G$.

For $\epsilon \in \bb{Q}_{>0}$, let $Q^\epsilon_{g, n}(X, \beta)$ denote the
moduli
space of $\epsilon$-stable quasimaps of class $\beta$.
In this paper, the source
curve of a quasimap is a twisted curve with balanced nodes and
\emph{trivialized} gerbe markings. This agrees with the convention in~\cite{Zhou2} and slightly
from that in~\cite{Kim4}. See~\cite[\textsection{6.1.3}]{AGV} for the
comparison. According to
\cite{Kim4},
the moduli space $Q^\epsilon_{g, n}(X, \beta)$ is a Deligne-Mumford stack
proper
over the affine quotient $W \git_{0} G$.

Let $\Lambda$ denote the ground $\lambda$-algebra which contains the Novikov
ring $\bb{Q}[[Q]]$ (see Section \ref{K-theoretic-quasimap-invariants} for the
precise definition). Let $
IX = \coprod_r I_rX$
be the cyclotomic inertia stack of $X$ and let $\bar{I}X=\coprod_r \bar{I}_rX$
be the \emph{rigidified} cyclotomic inertia stack. There is a
natural projection 
\[
\varpi:IX\rightarrow \bar{I}X,
\]
which exhibits $IX$ as the universal gerbe over $\bar{I}X$ (cf.~\cite{AGV}). Let
  $\mathrm{ev}_i:Q^\epsilon_{g, n}(X, \beta) \rightarrow  IX$ be the evaluation
  map
  at the $i$-th marked point. We consider
  $\ev_i=\varpi\circ\mathrm{ev}_i:Q^\epsilon_{g, n}(X, \beta) \rightarrow
  \bar{I}X$ and refer to it as the \emph{rigidified evaluation map} at the $i$-th
  marked point.
  We denote by $K(\bar{I} X)$ the
Grothendieck group of topological complex vector bundles on $\bar{I} X$ with
rational coefficients. Let $q$ be a formal variable and let
$\mb{t}(q) = \sum_{j \in \bb{Z}} \mb{t}_j q^j$ be a general Laurent polynomial
in $q$
with coefficients $t_j \in K(\bar{I}X)\otimes \Lambda$.
%\marginpar{\footnotesize{
%    \Yang{
 %     Which convention are we using: the marking are trivialized gerbes or
%      just gerbes? In \cite{Kim4}, they use the latter and thus the evaluation
 %     maps land in the rigidified (cyclotomic) inertia stacks.
 %   }
 %   \Ming{Use the same convention as your paper. The markings are trivialized gerbes.}
%  }}
Let $L_i$ denote the $i$-th cotangent line bundle
at the $i$-th marked point of \emph{coarse} curves. When $W \git_0G$ is a
point,
the GIT quotient $X$ is proper. Consider the $S_n$-action on the quasimap moduli space
$Q^\epsilon_{g, n}(X, \beta)$ defined by permuting the $n$ markings. Then there is a natural (virtual) $S_n$-module
\[
  \big[\mb{t}(L), \dots, \mb{t}(L) \big]^\epsilon_{g, n, \beta}: =
  \sum_{m \geq 0}(-1)^mH^m \big(Q^\epsilon_{g, n}(X, \beta),
  \ca{O}_{Q^\epsilon_{g, n}(X, \beta)}^{\text{vir}} \cdot
 \prod_{i = 1}^n \big( \sum_{j} \ev_i^*( \mb{t}_j)L_i^j \big)
  \big),
\]
where $\ca{O}_{Q^\epsilon_{g, n}(X, \beta)}^{\text{vir}}$ is the virtual
structure
sheaf and the operation $\cdot$ denotes the tensor product.
For simplicity, we
extend the definition of $\ev_i^*$ by linearity and denote
$
\ev_i^*(\mb{t}(L_i))
:=  \sum_{j} \ev_i^*(
\mb{t}_j)L_i^j$.

The \emph{permutation-equivariant}
$K$-theoretic $\epsilon$-stable quasimap invariant is defined as the complex dimension
of its $S_n$-invariant submodule. Equivalently, we define
\[
  \big \langle
  \mb{t}(L), \dots, \mb{t}(L) \big \rangle_{g, n, \beta}^{S_n, \epsilon}: =
  p_*
  \Big(  \ca{O}_{Q^\epsilon_{g, n}(X, \beta)}^{ \text{vir}} \cdot
  \prod_{i = 1}^n
\ev_i^*(\mb{t}(L_i))
  \Big),
\]
where $p_*$ is the proper pushforward along the projection
\[
  p: \big[Q^ \epsilon_{g, n}(X, \beta)/S_n \big] \rightarrow \text{Spec} \,
  \bb{C}
\]
in $K$-theory. If $W \git G$ is only quasiprojective but has a torus-action with good
  properties, we can define
  \emph{torus-equivariant} $K$-theoretic quasimap invariants via the
  $K$-theoretic
  localization formula.
%\marginpar{
 % \footnotesize{
  %  \Yang{
   %   This is imprecise.
    %  We need the torus to act on $[W/G]$.
   % }
   % \Ming{Yes, I think it is okay to be vague in the introduction. I will
    %  mention the definitions in the non-compact case is analogous to those in
     % Section 6.3 of \cite{kim1} }
 % }
%}
For simplicity of notation, we assume that $W \git_0G$
is a point from now on. The theorems hold true in general and the proofs are
verbatim.

In many applications, it is necessary to modify the virtual structure sheaf
$\ca{O}_{Q^\epsilon_{g, n}(X, \beta)}^{\text{vir}}$ by tensoring it with a
  determinant line bundle
%\marginpar{\footnotesize{
 %   \Yang{
  %    What is ``a determinant line bundle?''
   % }
 % }}
or a $K$-theory class coming from the universal family
over $Q^\epsilon_{g, n}(X, \beta)$. We discuss all such twistings in Section
\ref{twisted-theory}. Using the modified virtual structure sheaves, we can
define twisted quasimap invariants with level structure. In this introduction,
we will use the same notation for the corresponding invariants and generating
series in the twisted theory.

The wall-crossing formula involves an important generating series of
$K$-theoretic residues
%\marginpar{\footnotesize{
 %   \Yang{
  %    What is ``$K$-theoretic residue''?
   % }}}
over $g = 0$, ($\epsilon = 0 + $)-stable quasimap graph space.
We denote it by $I(Q, q)$ and call it the small $I$-function. We refer the
reader
to Definition \ref{K-theoretic-I-function} for the precise definition of the
$K$-theoretic small $I$-function. These functions are rational functions in $q$
modulo any power of Novikov’s variables, and explicitly computable in general
(see Remark \ref{remark-small-I-function}). Given a ($K$-group valued) rational
function $f(q)$, we denote by $[f(q)]_ + $ the Laurent polynomial part in the
partial fraction decomposition of $f(q)$. For a curve class $\beta$,
we
define the
$K$-group valued Laurent polynomial
\[
  \mu_{\beta}(q) \in K(\bar{I}X) \otimes \Lambda[q, q^{-1}]
\]
to be the coefficient of $Q^\beta$ in $[(1-q)I(Q, q)-(1-q)]_ + $.

For a fixed curve class $\beta$, the space $\mathbb Q_{> 0} \cup \{0 + , \infty
\}$
is divided into stability chambers by finitely many walls
  $\{1/d \mid d \in \mathbb Q_{>0}, d \leq \text{deg}(\beta) \}$.
%\marginpar{\footnotesize{
 %   \Yang{
  %    The walls are in general not integers unless you raise $\theta$ to some
   %   multiple of itself.
    %}
 % }}
Let $\epsilon_0 = 1/d_0$ be a wall,
where $d_0 \leq \text{deg}(\beta)$. Let
$\epsilon_-< \epsilon_ + $ be the stability conditions in the two
adjacent chambers separated by $\epsilon_0$. We also fix the genus $g$ and the
number of markings $n$ such that $2g -2 + n + \epsilon_0 \,
\text{deg}(\beta)>0$. To compare the two virtual structure sheaves $\ca{O}_{Q^{\epsilon_+}_{g,
    n}(X, \beta)}^{ \text{vir}}$ and $\ca{O}_{Q^{\epsilon_-}_{g, n}(X, \beta)}^{
  \text{vir}}$, we recall several natural morphisms between various moduli
spaces. First, we have
\[
  \begin{tikzcd}
Q^{\epsilon_+}_{g,n}(X, \beta)\arrow[d,"\iota"]
&
Q^{\epsilon_-}_{g,n}(X, \beta))\arrow[d,"\iota"]
\\
Q^{\epsilon_+}_{g,n}(\bb{P}^N, d)
\arrow[r,"c"]
&
Q^{\epsilon_-}_{g,n}(\bb{P}^N, d)
\end{tikzcd}.
\]
Here $d:=\mathrm{deg}(\beta)$, the $\iota$'s are defined by composing the quasimap
with~\eqref{eq:equation-GIT-embedding} and forgetting the orbifold structure of
the domain curves, and $c$ is defined by contracting all degree-$d_0$ rational
tails to length-$d_0$ base points (\cite[\textsection{3.2.2}]{kim2}). It is
obvious that the morphisms $\iota$ and $c$ commute with the $S_n$-actions on the
quasimap moduli spaces. Consider the $S_k$-action on
$Q^{\epsilon_-}_{g,n+k}(\bb{P}^N, d-kd_0)$ defined by permuting the last $k$
markings. Let
\begin{equation*}
b_k:[Q^{\epsilon_-}_{g,n+k}(\bb{P}^N, d-kd_0)/S_k]\rightarrow Q^{\epsilon_-}_{g,n}(\bb{P}^N, d)
\end{equation*}
be the map that replaces the last $k$ markings by base points of length $d_0$
(\cite[\textsection{3.2.3}]{kim2}). For any $\epsilon$, let
\[
  \ev:Q^{\epsilon}_{g,n+k}(X, \beta)\rightarrow (\bar{I}X)^{n}
\]
be the product of the rigidified evaluation
maps at the first $n$ markings.

Let $K_\circ([Q^{\epsilon_-}_{g,n}(\bb{P}^N, d)/S_n])_{\Q}$ denote the Grothendieck
group of coherence sheaves on $[Q^{\epsilon_-}_{g,n}(\bb{P}^N, d)/S_n]$ with $\Q$
coefficients. It can be identified with the rational Grothendieck group
$K_\circ^{S_n}(Q^{\epsilon_-}_{g,n}(\bb{P}^N, d))_{\Q}$ of
$S_n$-equivariant coherent sheaves on $Q^{\epsilon_-}_{g,n}(\bb{P}^N, d)$.

\begin{theorem}[Theorem \ref{all-genus-single-wall-crossing}]
  \label{all-genus-ovir-single-wall-crossing-intro}
  Assuming that $2g-2 + n + \epsilon_0 \, \emph{deg}(\beta)>0$, we have
  \begin{equation}\label{eq:equation-all-genus-ovir-single-wall-crossing-intro}
 \begin{aligned}
   &(\iota\times \ev)_*\ca{O}^{\emph{vir}}_{Q^{\epsilon_-}_{g,n}(X, \beta)}
   -((c\circ \iota)\times\ev )_*\ca{O}^{\emph{vir}}_{Q^{\epsilon_+}_{g,n}(X, \beta)}\\
   =&\sum_{k\geq 1}\bigg(\sum_{\vec{\beta}}((b_{k}\circ c\circ \iota )\times \ev)_*
   \bigg(
     \prod_{a=1}^k\ev_{n+a}^*\,\mu_{\beta_a}(L_{n+a})
     \cdot
     \ca{O}^{\emph{vir}}_{Q^{\epsilon_+}_{g,n+k}(X, \beta')}
   \bigg)\bigg)
 \end{aligned}
  \end{equation}
 in $K_0\big([(Q^{\epsilon_-}_{g,n}(\bb{P}^N, d)\times (\bar{I}X)^n\big)/S_n]\big)_{\Q}$, where $\vec{\beta}$ runs
 through all ordered tuples
  \[
    \vec \beta = (\beta^\prime, \beta_1, \ldots, \beta_k)
  \]
  such that $\beta = \beta^\prime + \beta_1 + \cdots + \beta_k$ and
  $\deg(\beta_i) = d_0$ for all $i = 1 , \ldots, k$. The same formula also holds
  for the
  twisted virtual structure sheaves with level structure define in Section \ref{twisted-theory}.

\end{theorem}
The above theorem implies a numerical wall-crossing formula. Before stating the result, let us introduce some notation so that the numerical wall-crossing terms are indexed
by \emph{unordered} tuples. Consider the permutation action of $S_k$ on the set of all ordered tuples $\vec \beta = (\beta^\prime, \beta_1, \ldots, \beta_k)
$ in Theorem
\ref{all-genus-ovir-single-wall-crossing-intro}:
\begin{equation} \label{eq:S_k-permutating-index}
  \sigma(\vb): = (\beta', \beta_{\sigma(1)}, \dots,
  \beta_{\sigma(k)}),\quad \sigma\in S_k.
\end{equation}
We denote the orbit of $\vec \beta$ under the permutation action by
  \[
    \underline{\beta} = (\beta^\prime,\{ \beta_1, \ldots, \beta_k\}),
  \]
  where $\{\beta_1,\dots,\beta\}$ is unordered and hence forms a multiset.
Let $S_{\vb}\subset S_k$ denote the stabilizer subgroup that fixes
$\vb$. We define the following ($S_n\times S_{\vb}$)-permutation equivariant
invariant
\begin{align*}
 & \big \langle
  \mb{t}(L), \dots,
        \mb{t}(L),
        \mu_{ \beta_1}(L) , \dots,  \mu_{ \beta_k}(L) \big \rangle_{g, n + k,
        \beta'}^{S_n \times S_{\vb},
        \epsilon_ + }\\
      : =&
  p^{\vb}_*
  \Big(  \ca{O}_{Q^{\epsilon_+}_{g, n+k}(X, \beta')}^{ \text{vir}} \cdot
  \prod_{i = 1}^n \ev_i^*({\bf t}(L_i))
  \cdot\prod_{j=1}^k \ev_{n+j}^*(\mu_{\beta_j}(L_{n+j}))
  \Big),
\end{align*}
where $p_*^{\vb}$ is the proper pushforward along the projection
\[
  p^{\vb}: \big[Q^ {\epsilon_+}_{g, n+k}(X, \beta')/S_n\times S_{\vb} \big] \rightarrow \text{Spec} \,
  \bb{C}
\]
in $K$-theory. Note that the above invariant only depends on the orbit $\ub$ of
$\vb$.

All the maps involved in Theorem \ref{all-genus-ovir-single-wall-crossing-intro}
induce isomorphisms of the relative cotangent spaces at the first $n$ markings. By tensoring both sides of~\eqref{eq:equation-all-genus-ovir-single-wall-crossing-intro} with
$\prod_{i=1}^n{\bf{t}}(L_i)$ and taking the proper
pushforward along $[(Q^{\epsilon_-}_{g,n}(\bb{P}^N, d)\times (IX)^n\big)/S_n]\rightarrow\operatorname{Spec}\C$, we
obtain the following result:
\begin{theorem}
  \label{all-genus-single-wall-crossing-intro}
  Assuming that $2g-2 + n + \epsilon_0 \, \emph{deg}(\beta)>0$, we have
  \begin{align*}
    & \big \langle \mb{t}(L), \dots, \mb{t}(L) \big \rangle_{g, n, \beta}^{S_n,
      \epsilon_-}- \big \langle \mb{t}(L), \dots, \mb{t}(L) \big
      \rangle_{g, n,
      \beta}^{S_n, \epsilon_ + } \\
    = & \sum_{k = 1}^m \sum_{ \underline{\beta}} \big \langle \mb{t}(L), \dots,
        \mb{t}(L),
        \mu_{ \beta_1}(L) , \dots,  \mu_{ \beta_k}(L) \big \rangle_{g, n + k,
        \beta'}^{S_n \times S_{\underline{\beta}},
        \epsilon_ + },
  \end{align*}
  %\marginpar{\footnotesize{
   %   \Yang{
    %    We also have a version involving
     %   $[\mathbf t(L) , \ldots, \mathbf t(L)]_{g, n, \beta}^{\epsilon}$.
      %  We get a equation of $S_n$-modules. Do you think that is important?
      %}
   % }}
  where $\underline{\beta}$ runs through all
  \[
    \underline{\beta} = (\beta^\prime,\{ \beta_1, \ldots, \beta_k\}),
  \]
 with $\{ \beta_1, \ldots, \beta_k\}$ unordered, $\beta = \beta^\prime + \beta_1 + \cdots + \beta_k$ and
  $\deg(\beta_i) = d_0$ for all $i = 1 , \ldots, k$. The same formula also holds
  for
  twisted permutation-equivariant $\epsilon$-quasimap invariants with level structure.
\end{theorem}
%\Yang{We have not introduced the notation
 % \[
  %  \langle \cdots \rangle^{S_n \times S_k}_{g, n + k, \beta^\prime}
  %\]
 % yet. It is important than summing over all $\beta$ makes the last
  %$k$-insertions carry a natural $S_k$-action.
  %Maybe we should at least comment on this here.
%}
%\Ming{I stated a wall-crossing formula in K-theory and defined the above
 % notation.}

By definition,~\eqref{eq:equation-all-genus-ovir-single-wall-crossing-intro} is
also an identity of $S_n$-equivariant $K_\circ$-classes on
$Q^{\epsilon_-}_{g,n}(\bb{P}^N, d)\times (\bar{I}X)^n$. If we tensor~\eqref{eq:equation-all-genus-ovir-single-wall-crossing-intro} with
$\prod_{i=1}^n{\bf{t}}(L_i)$ and take the proper pushforward along
$Q^{\epsilon_-}_{g,n}(\bb{P}^N, d)\times (\bar{I}X)^n\rightarrow \operatorname{Spec}\C$,
we obtain
\begin{theorem}
  \label{all-genus-single-wall-crossing-intro-module-version}
  Assuming that $2g-2 + n + \epsilon_0 \, \emph{deg}(\beta)>0$, we have
  \begin{align*}
    & \big[ \mb{t}(L), \dots, \mb{t}(L) \big]_{g, n, \beta}^{
      \epsilon_-}-
      \big [ \mb{t}(L), \dots, \mb{t}(L)
      \big]_{g, n,
      \beta}^{\epsilon_ + } \\
    = & \sum_{k = 1}^m \sum_{ \underline{\beta}} \big[ \mb{t}(L), \dots,
        \mb{t}(L),
        \mu_{ \beta_1}(L) , \dots,  \mu_{ \beta_k}(L) \big ]_{g, n + k,
        \beta'}^{S_{\ub},
        \epsilon_ + }
  \end{align*}
  as virtual $S_n$-modules, where $\underline{\beta}$ satisfies the same conditions as in
  Theorem~\ref{all-genus-single-wall-crossing-intro} and
  \[
\big[ \mb{t}(L), \dots,
        \mb{t}(L),
        \mu_{ \beta_1}(L) , \dots,  \mu_{ \beta_k}(L) \big ]_{g, n + k,
        \beta'}^{S_{\ub},
        \epsilon_ + }:= \tilde{p}^{\vb}_*
  \Big(  \ca{O}_{Q^{\epsilon_+}_{g, n+k}(X, \beta')}^{ \mathrm{vir}} \cdot
  \prod_{i = 1}^n \ev_i^*({\bf t}(L_i))
  \cdot\prod_{j=1}^k \ev_{n+j}^*(\mu_{\beta_j}(L_{n+j}))
  \Big)
\]
with $\tilde{p}^{\vb}_*:\big[Q^ {\epsilon_+}_{g, n+k}(X, \beta')/S_{\vb} \big] \rightarrow \operatorname{Spec}\bb{C}$. The same formula also holds
 in the twisted permutation-equivariant $\epsilon$-stable quasimap $K$-theory with level structure.
\end{theorem}

Theorem~\ref{all-genus-single-wall-crossing-intro-module-version} and
Theorem~\ref{all-genus-single-wall-crossing-intro} are in fact equivalent. On
the one hand, it is clear that Theorem~\ref{all-genus-single-wall-crossing-intro-module-version}
implies Theorem~\ref{all-genus-single-wall-crossing-intro}. On the other hand,
Theorem~\ref{all-genus-single-wall-crossing-intro} produces an identity in an
arbitrary $\lambda$-ring $\Lambda$. In particular, we can choose $\Lambda$ to contain the abstract algebra of symmetric functions $Q[[N_1,N_2,\dots]]$ with
the Adams operations $\Psi^r(N_m)=N_{rm}$. As explained in Example 2 and Example
3 in~\cite{givental11}, due to Schur–Weyl’s reciprocity, the
permutation-equivariant invariant captures the entire information about the
$S_n$-module $[\mb{t}(L), \dots,\mb{t}(L)]^\epsilon_{g, n, \beta}$. Hence
Theorem~\ref{all-genus-single-wall-crossing-intro} also implies Theorem~\ref{all-genus-single-wall-crossing-intro-module-version}.

We introduce the permutation-equivariant \emph{genus-$g$ descendant potential}
of $X$:
\[
  F^{\epsilon}_{g}(\mathbf t(q)) = \sum_{n = 0}^\infty \sum_{\beta \geq
    0}Q^\beta
  \langle \mathbf t(L) , \ldots, \mathbf t(L) \rangle^{S_n, \epsilon}_{g, n,
    \beta}
\]
and define the following truncation of $(1-q)I(Q, q)-(1-q)$:
\[
  \mu^{\geq \epsilon}(Q, q) = \sum_{\epsilon \leq 1/ \deg(\beta)< \infty}
  \mu_\beta(q)Q^{\beta}.
\]
By repeatedly applying Theorem~ \ref{all-genus-single-wall-crossing-intro} to
cross the walls in $[\epsilon, \infty)$, we obtain the following corollary,
which
holds for both the untwisted and twisted cases.
\begin{corollary} \label{potential-wall-crossing}
  For $g \geq 1$ and any $\epsilon$, we have
  \[
    F^{\epsilon}_{g} \big(\mathbf t(q) \big) = F^{ + \infty}_{g} \big(\mathbf
    t(q) + \mu^{\geq \epsilon}(Q, q) \big).
  \]
  For $g = 0$, the same equality holds true modulo the constant and linear terms
  in $\mb{t}$.
\end{corollary}

In genus zero, Givental introduced an important generating series called
the
$J$-function. Its generalization to the $K$-theoretic $\epsilon$-stable quasimap theory is
straightforward and we refer the reader to Definition \ref{K-theoretic-I-function}
for the precise formula. The following theorem is
proved in Section \ref{The-genus-0-case}, which again holds for both the
untwisted and twisted cases.
\begin{theorem}[Theorem~\ref{thm:genus-0-case}] \label{genus-0-mirror-theorem}
  For any $\epsilon$, we have
  \begin{equation*}
    J^{\infty}_{S_\infty}(\mathbf t(q) + \mu^{\geq \epsilon}(Q, q), Q) =
    J_{S_\infty}^{\epsilon}(\mathbf t(q), Q).
  \end{equation*}
\end{theorem}
The small $I$-function is a specialization of the $J$-function in the chamber
$\epsilon = 0 + $. To be more precise, we have
\[
  (1-q)I(Q, q) = J_{S_\infty}^{0 + }(\mathbf t(q), Q)|_{\mathbf t(q) = 0}.
\]
Hence Theorem \ref{genus-0-mirror-theorem} generalizes the genus-zero toric
mirror theorems in quantum $K$-theory
\cite{Givental-Tonita, givental14, givental15} and quantum $K$-theory with
level
structure \cite{RZ1}, and the $K$-theoretic wall-crossing formula in
\cite{Tseng-You}.

\subsection{Comparison with the proof of the cohomological wall-crossing
  formulas in~\cite{Zhou2}}
The main idea of this paper is the same as that of~\cite{Zhou2}. Namely, we
deduce the wall-crossing formulas by applying the torus-localization formula to
the master space constructed in~\cite{Zhou2}. However, $K$-theory is more sensitive
to the stacky structure of the moduli spaces, comparing to cohomology theory. In
particular, we mention some new features here. The first one is the emergence of the
permutation-equivariant structure. When simplifying the localization contributions,
we need to split a certain number of nodes simultaneously and the newly-created
markings are \emph{unordered}. It `forces' us to define invariants over the quotient
stack of a usual
quasimap moduli space by some permutation group. In the cohomological setting,
such modification will only change a invariant by multiplying a constant.
However, in the $K$-theory setting, the effect of such modification is more
complicated. In fact, this was first observed by Givental in \cite{givental12}
when analyzing
the $K$-theoretic localization formula over the moduli space of stable maps.
This phenomenon motivated him to introduce permutation-equivariant $K$-theoretic
invariants.

To simplify the localization contributions, we need to refine a few key lemmas
in~\cite{Zhou2}
about the structures of certain morphisms. And the combinatorics used in the
simplification is also more involved than that in~\cite[\textsection{7}]{Zhou2}.

\subsection{Plan of the paper}
This paper is organized as follows. In Section
\ref{section-K-theoretic-quasimap}, we recall the basic notation in $K$-theory
and define permutation-equivariant $K$-theoretic quasimap invariants and their
generating functions. Twisted theories with level structure are discussed at the
end of this section. In Section \ref{section-entangled-tails}, we first recall
the construction of moduli spaces of quasimaps with entangled tails in
\cite{Zhou2}. Then we define and study the virtual structure sheaves over these
moduli spaces. In Section \ref{section-master-space}, we recall the master space
introduced in \cite{Zhou2}. It has a $\bb{C}^*$-action. We compute the
$\bb{C}^*$-equivariant $K$-theoretic Euler classes of the virtual normal bundles
of all fixed-point components in the master space. In Section
\ref{section-wall-crossing}, we first discuss some basic properties of the
residue operation. Then we apply the $K$-theoretic localization formula to the
master space and obtain the wall-crossing formula. Note that the combinatorics
in simplifying the localization contributions is much more sophisticated than
that in the cohomological setting \cite{Zhou2}.

\subsection{Acknowledgments}
The first author would like to thank David Anderson, Dan Edidin, Daniel
Halpern-Leistner, Amalendu Krishna, Jeongseok Oh and Bhamidi Sreedhar for
helpful discussions. The second author would like to thank
Huai-liang Chang,
Michail Savvas,
Arnav Tripathy and
Kai Xu
for helpful discussions.

This project started when the second author was invited by Yongbin Ruan to visit
the University of Michigan. Both authors would like to thank Yongbin Ruan for
his encouragement and support.

%%% Local Variables:
%%% mode: latex
%%% TeX-master: "main"
%%% End:

%\input{quasimap-invariants}

\section{$K$-theoretic quasimap invariants}
\label{section-K-theoretic-quasimap}

In this section, we first introduce some basic notation in $K$-theory. Then we
review the basics of (orbifold) quasimap theory and define
permutation-equivariant $K$-theoretic invariants. Twisted theories will be 
in the end.

\subsection{Basic notation in $K$-theory}
For a Deligne--Mumford stack $X$, we denote by $K_\circ(X)$ the Grothendieck
group of coherent sheaves on $X$ and by $K^\circ(X)$ the Grothendieck group of
locally free sheaves on $X$. Suppose $X$ has a $\bb{C}^*$-action. Let
$K_\circ^{\C^*}(X)$ and $K^\circ_{\C^*}(X)$ denote the equivariant $K$-groups.
We have canonical isomorphisms
\[
  K_\circ^{\C^*}(X) \cong K_\circ([X/ \C^*]),
  \quad
  K^\circ_{\C^*}(X) \cong K^\circ([X/ \C^*]).
\]
Let $K(X)$ denote the Grothendieck group of topological complex vector bundles
on $X$. Given a coherent sheaf (or vector bundle) $F$, we denote by $[F]$ or
simply $F$ its associated $K$-theory class. Tensor product makes $K_\circ(X)$ a
$K^\circ(X)$-module:
\[
K^\circ(X)\otimes K_\circ(X)\rightarrow K_\circ X,\quad \big([E],[F]\big)\mapsto [E]\cdot [F]:=[E\otimes_{\ca{O}_X}F].
\]
The unit in $K^\circ(X)$ is the class of the structure sheaf and we simply denote it by 1.
Throughout the paper, we consider Grothendieck groups with rational
coefficients $K_\circ(X)_{\Q}: = K_\circ(X) \otimes \bb{Q}$,
$K^\circ(X)_{\Q}: = K^\circ(X) \otimes \bb{Q}$ and
$K(X)_{\Q}=K(X)\otimes\Q$
%\Ming{Originally, I wrote complex coefficients. That is
%needed when considering orbifold Riemann-Roch formula. In our approach, we don't
%go to cohomology theory and the localization formula holds over $\Q$. I can't
%think of any place where we need $\C$}.

For a flat morphism $f:X \rightarrow Y$, we have the flat pullback
$f^*: K_\circ(Y) \rightarrow K_\circ(X)$.
For a proper morphism $g:X \rightarrow Y$, we
have the proper pushforward $f_*:K_\circ(X) \rightarrow K_\circ(Y)$ defined by
\[
  [F] \mapsto \sum_n(-1)^n[R^nf_*F].
\]
In the case when $X$ is proper and $Y$ is a point, the proper pushforward can
be identified with the holomorphic Euler characteristic
\[
  \chi(F): = \sum_{i \geq0} \text{dim}_{\bb{C}} \, H^i(X, F).
\]

A key tool in computing holomorphic Euler characteristics and proper
pushforwards is the localization formula in equivariant $K$-theory. Such
localization formula was first introduced by Atiyah-Segal \cite{Atiyah-Segal}
and developed by Thomason \cite{Thomason1, Thomason2} in the algebraic setting.
A
detailed discussion of the classical $K$-theoretic torus-localization formula
can be found in \cite[Chapter 5]{Chriss-Ginzburg}. Virtual localization
formulas
in $\bb{C}^*$-equivariant $K$-theory for stacks are proved in
\cite{DHL, Kiem-Savvas}. We will use the version in \cite{Kiem-Savvas} and
recall
the precise statement in Section \ref{section-wall-crossing}.

We will use $q$ to denote the equivariant parameter (or weight), which
corresponds to the standard representation $\bb{C}_{\text{std}}$ of $\bb{C}^*$.
We have $K^\circ_{\bb{C}^*}(\text{pt}) \cong \bb{Z}[q, q^{-1}]$.

Let $E$ be a vector bundle on $X$. We define the \emph{$K$-theoretic Euler
  class} of $E$ by
\[
  \lambda_{-1}(E^\vee): = \sum_i(-1)^i \wedge^iE^\vee \in K^\circ(X),
\]
where $\wedge^iE^\vee$ denote the $i$-th exterior power of $E^\vee$.
Suppose $X$ has a $\bb{C}^*$-action. For a $\bb{C}^*$-equivariant vector bundle
$E$, the same formula defines its $\bb{C}^*$-equivariant $K$-theoretic Euler
class $\lambda_{-1}^{\C^*}(E^\vee) \in K^\circ_{\bb{C}^*}(X)$.

\subsection{Stable quasimaps to GIT quotients}\label{stable-quasimaps}
Let $W$ be an affine variety with a right action of a reductive group $G$. We
assume that $W$ has at worst local complete intersection singularities. Let
$\theta$ be a character of $G$. Denote by $W^{s}(\theta)$ the $\theta$-stable
locus, and $W^{ss}(\theta)$ the $\theta$-semistable locus with respect to the
linearization $L_\theta: = W \times \bb{C}_\theta$. We assume that
$W^{s}(\theta)$
is smooth, nonempty, and coincides with $W^{ss}(\theta)$. In this paper, we
consider the ``stacky'' GIT quotient
\[
  X = [W^{ss}(\theta)/G].
\]
Note that $X$ is a smooth proper Deligne--Mumford stack over he affine quotient
$W \git_{0} G = \operatorname{Spec} H^0(W, \mathcal O_W)^G$.

Let $(C, x_1, \dots, x_n)$ be a $n$-pointed, genus $g$ twisted curve with
balanced
nodes and trivialized gerbe markings (c.f. \cite[\textsection{4}]{AGV}). A map
$[u]:C \rightarrow X$ corresponds to a pair $(P, u)$ with
\[
  P \rightarrow C
\]
a principal $G$-bundle on $C$ and
\[
  u:C \rightarrow P \times_G W
\]
a section of the fiber bundle $P \times_GW \rightarrow C$. We call $[u]$ a
quasimap to $X$ if [$u]$ is representable and $[u]^{-1}([W^{us}/G])$ is
zero-dimensional. Here $W^{us}$ denotes the unstable locus. The locus
$[u]^{-1}([W^{us}/G])$ is called the \emph{base locus} of $[u]$ and points in
the base locus are called \emph{base points}.

The class $\beta$ of a quasimap is defined to be the
group homomorphism
\[
  \beta: \text{Pic}([W/G]) \rightarrow \bb{Q}, \quad L \mapsto
  \text{deg}([u]^*(L)).
\]
We refer to the rational number $\mathrm{deg}(\beta) = \text{deg}([u]^*(L_\theta))$
as
the \emph{degree} of the quasimap $[u]$. A group homomorphism
$\beta: \text{Pic}([W/G]) \rightarrow \bb{Q}$ is called an \emph{effective}
curve
class if it is the class of some quasimap $[u]$. We denote by
$\text{Eff}(W, G, \theta)$ the semigroup of $L_\theta$-effective curve classes
on
$X$. For convenience, we write $\beta \geq0$ if $\beta \in \text{Eff}(W, G,
\theta)$
and $\beta>0$ if the effective curve class is nonzero.

Fix a positive rational number $\epsilon$. A quasimap is called
$\epsilon$-\emph{stable} if the following three conditions hold:
\begin{enumerate}
\item
  The base points are disjoint from the trivialized gerbe markings and nodes
  %\Ming{I think this is the
   % terminology of orbifold quasimap theory paper. I used the same terminology
    %in two paragraphs above where I introduced $(C,x_1,\dots,x_n)$}
 % \marginpar{\footnotesize{
  %    \Yang{
   %     The wording ``gerbe markings'' looks weird. Also, it makes the reader
    %    wonder: what about non-orbifold markings? Are those also required to be
     %   away from the base points?
     % }
   % }}
  of $(C, x_1, \dots, x_n)$.
\item
  For every $y \in C$, we have $l(y) \leq 1/ \epsilon$, where $l(y)$ is the
  length at $y$ of the subscheme $[u]^{-1}([W^{us}(\theta)/G])$.
\item
  The $\bb{Q}$-line bundle $(u^*L_\theta)^{\otimes \epsilon} \otimes
  \omega_{C, \log}$ is positive, where $ \omega_{C, \log}: = \omega_C(\sum_i x_i)$
  is the log
  dualizing sheaf.
\end{enumerate}
A quasimap is called $(0 +)$-stable if it is $\epsilon$-stable for every
sufficiently small positive rational number $\epsilon$, and $\infty$-stable if
it is $\epsilon$-stable for every sufficiently large $\epsilon$.

It is straightforward to define families of quasimaps over any base scheme $S$. 
We denote by $Q^{\epsilon}_{g, n}(X, \beta)$ the moduli of genus-$g$
$\epsilon$-stable quasimaps to $X$ of curve class $\beta$ with $n$ markings.
This is slightly different from the convention in~\cite{Kim4}, where the gerbe
markings are not trivialized. Let $\underline{Q}^{\epsilon}_{g, n}(X, \beta)$
be the moduli space defined in~\cite[\textsection{2.3}]{Kim4}. Let $\ca{G}_i$ be
the $i$-th marking in the universal curve, which is a gerbe over
$\underline{Q}^{\epsilon}_{g, n}(X, \beta)$. Then we have
\[
Q^{\epsilon}_{g, n}(X, \beta)=\ca{G}_1\underset{\underline{Q}^{\epsilon}_{g,
    n}(X, \beta)}{\times}\cdots
\underset{\underline{Q}^{\epsilon}_{g, n}(X, \beta)}{\times}
\ca{G}_n.
\]
Hence by \cite[Theorem 2.7]{Kim4},
$Q^{\epsilon}_{g, n}(X, \beta)$ is a Deligne--Mumford stack, proper over the affine quotient $W \git_0 G$, with
a perfect obstruction
theory.

 Let $
IX = \coprod_r I_rX$
be the cyclotomic inertia stack of $X$ and let $\bar{I}X=\coprod_r \bar{I}_rX$
be the rigidified cyclotomic inertia stack (c.f. \cite[\textsection{3.1}]{AGV}). There is a
natural projection 
\[
\varpi:IX\rightarrow \bar{I}X,
\]
which exhibits $IX$ as the universal gerbe over $\bar{I}X$.
For $1\leq i\leq n$, there is an evaluation map
\[
\mathrm{ev}_i:Q^\epsilon_{g, n}(X, \beta) \rightarrow I X
\]
at the $i$-th marking. Similarly, there are evaluation maps
$\overline{\mathrm{ev}}_i:\underline{Q}^\epsilon_{g, n}(X, \beta) \rightarrow
\bar{I} X$. We have a non-cartesian commutative diagram 
\begin{equation*} 
  \begin{tikzcd}
    Q^{\epsilon}_{g, n}(X, \beta)\arrow[r,"\mathrm{ev}_i"]\arrow[d] &
    IX \arrow[d,"\varpi"]\\
    \underline{Q}^{\epsilon}_{g, n}(X, \beta)\arrow[r,"\overline{\mathrm{ev}}_i"]
    \arrow[r]
      &
      \bar{I}X
  \end{tikzcd}.
  \end{equation*}
  For later applications, we define the \emph{rigidified evaluation map}
  \begin{equation}\label{eq:rigidified-evaluation-map}
  \ev_i=\varpi\circ\mathrm{ev}_i:Q^\epsilon_{g, n}(X, \beta) \rightarrow \bar{I} X.
  \end{equation}

  According to the general constructions in
\cite{Lee, Qu}, the perfect obstruction theory induces a virtual structure
sheaf
\[
  \ca{O}_{Q^\epsilon_{g, n}(X, \beta)}^{\text{vir}} \in K_\circ(Q^\epsilon_{g,
    n}(X, \beta)).
\]
See Section \ref{virtual-structure-sheaf} for the details of the construction
of the virtual structure sheaf. The moduli stack of $(0 +)$-stable (resp.
$\infty$-stable) quasimaps is denoted by $Q^{0 + }_{g, n}(X, \beta)$ (resp.
$Q^{\infty}_{g, n}(X, \beta)$).

%\Ming{To express all terms in the J-function as residues, I need to introduce
 % $QG_{0, 1+k}(X, \beta)$ instead of just $QG_{0, 1}(X, \beta)$}
To define the $J$-function and the $I$-function for the genus-0 theory, we will
also need quasimap \emph{graph spaces}. Here we recall a special case of the
definition in \cite{Kim4}. Given an effective curve class $\beta$, choose
$\epsilon\in\Q_{>0}$ and $A\in\bb{Z}_{>0}$
such that
$1/A<\epsilon<1/\operatorname{deg}(\beta)$. We view $\bb{P}^1$ as the GIT quotient $\bb{C}^2 \git
\bb{C}^*$ with the polarization $\ca{O}_{\bb{P}^1}(A)$. Then the $(0+)$-stable quasimap graph space is defined by
\[
  QG^{0+}_{0, 1}(X, \beta): = Q^{\epsilon}_{0, 1}(X \times \bb{P}^1, \beta
  \times[\bb{P}^1]).
\]
The definition is independent of the choice of $A$ and $\epsilon$. This moduli space parametrizes quasimaps to $X\times\bb{P}^1$ with a unique rational component
whose coarse moduli is mapped isomorphically onto $\bb{P}^1$. 
% \Ming{double check if the definition is correct.}

We denote the first marking
by
$x_{\star}$. Consider the $\bb{C}^*$-action on $\bb{P}^1$ given by
\begin{equation} \label{eq:C*action}
  t[\zeta_0, \zeta_1] = [t \zeta_0, \zeta_1], \quad t \in \bb{C}^*.
\end{equation}
Set $0: = [1:0]$ and $\infty: = [0:1]$. Then the tangent space of $\bb{P}^1$ at
$\infty$ (resp. 0) is isomorphic to the standard representation (resp. the dual
of the standard representation) of $\bb{C}^*$. The $\bb{C}^*$-action
(\ref{eq:C*action}) induces an action on $QG_{0, 1}^{0+}(X, \beta)$. Let $F^{0,
  \beta}_{\star, 0}$ be the distinguished fixed-point component consisting of
$\mathbb C^*$-fixed quasimaps such that only the marking $x_{\star}$ is over
$\infty$, while the other $k$ markings and the entire class
$\beta$ are over $0 \in \mathbb P^1$.

The restriction of the absolute perfect obstruction theory of $QG^{0+}_{0,
  1}(X, \beta)$ to $F^{0, \beta}_{\star, 0}$ decomposes into moving and fixed
parts. By \cite{G-P}, the fixed part of the obstruction theory defines a
perfect obstruction theory on $F^{0, \beta}_{\star, 0}$ and hence induces a
virtual structure sheaf $\ca{O}^\vir_{F^{0, \beta}_{\star, 0}} \in
K_\circ(F^{0, \beta}_{\star, 0})$. The moving part of the obstruction theory
gives rise to a virtual normal
bundle $N^{\mathrm{vir}}_{F^{0, \beta}_{\star, 0}/ QG^{0+}_{0, 1}(X, \beta)}
\in K_{\bb{C}^*}^\circ(F^{0, \beta}_{\star, 0})$.

\subsection{$K$-theoretic quasimap invariants}
\label{K-theoretic-quasimap-invariants}

We recall Givental--Tonita's \emph{$K$-theoretic symplectic loop space
  formalism} in the orbifold setting \cite{Tonita1, givental17}. To introduce
permutation-equivariant $K$-theoretic invariants, we choose the ground
coefficient ring to be a $\lambda$-algebra $\Lambda$, i.e. an algebra over
$\bb{Q}$ equipped
with abstract Adams operations $\Psi^k, k = 1, 2, \dots$. Here
$\Psi^k: \Lambda \rightarrow \Lambda$ are ring homomorphisms satisfying
$\Psi^r \Psi^s = \Psi^{rs}$ and $\Psi^1 = \text{id}$. In this paper, we assume
that $\Lambda$ is over $\bb{C}$ and includes the Novikov variables $Q^\beta,
\beta \in \text{Eff}(W, G, \theta)$ and the torus-equivariant $K$-ring of a
point if we consider torus-actions on the target. We also assume that $\Lambda$
is equipped with a maximal ideal $\Lambda_ + $ such that $\Psi^i(\Lambda_ +)
\subset(\Lambda_ +)^2$ for $i>1$. For example, one can choose
\[
  \Lambda = \bb{Q}[[N_1, N_2, \dots]][[Q]][\lambda_0^\pm, \dots,
  \lambda_N^\pm],
\]
where $N_i$ are the Newton polynomials (in infinitely or finitely many
variables), $Q$ denotes the Novikov variable(s), and $\lambda_i$ denote the
torus-equivariant parameters. The Adams operations $\Psi^r$ act on $N_m$ and
$Q$
by $\Psi^r(N_m) = N_{rm}$ and $\Psi^r(Q^\beta) = Q^{r \beta}$, respectively,
and they act trivially on the torus-equivariant parameters. One can take
$\Lambda_ + $ to be the maximal idea generated by $N_i$, $\lambda_i$ and
Novikov variables of positive degrees.

As mentioned in the introduction, for simplicity, we assume the affine quotient
$W \git_0G$ is a point and hence $X$ is proper. Let
$K(\bar{I}X)$ be the Grothendieck group of topological complex vector
bundles
on $\bar{I}X$ with rational coefficients
%\marginpar{\footnotesize{
 %   \Yang{Not complex coefficients?}
  %}
%\Ming{Why does it need to be complex?}}
(see e.g. \cite{Adem}). The \emph{Mukai
  pairing} on $K(\bar{I}X)$ is defined by
\begin{equation*} \label{eq:mukai-pairing}
  (\alpha, \beta) = \chi(\bar{I}X, \alpha \cdot \iota^* \beta),
\end{equation*}
where $\iota$ is the involution on $\bar{I}X$ reversing the banding.

%We make
%the following
%\begin{assumption}
 % The Mukai pairing $(-, -)$ is non-degenerate.
%\end{assumption}
%\Yang{When does this hold?}
%
%\Ming{Initially, I needed this condition to write down the formula of the
 % J-function. We don't really need it if we define the J-function as the
  %pushforward of the K-theoretic residue. However, I wonder if we need to use
  %this assumption when splitting along the nodes later. I need to double-check it.}

We define the \emph{$K$-theoretic loop space} by
\[
  \ca{K}: = [K(\bar{I}X) \otimes \bb{C}(q)] \widehat{\otimes} \Lambda,
\]
where $\bb{C}(q)$ is the field of complex rational functions in $q$ and
``$\widehat{\phantom{\otimes}}$''  means the completion in the $\Lambda_ +
$-adic topology. In other words, \emph{modulo any power of $\Lambda_ + $}, the
elements of $\ca{K}$ are rational functions of $q$ with vector coefficients
from $K(\bar{I}X) \otimes \Lambda$. From now on, we simplify refer to them as
rational functions.

By viewing elements in $\bb{C}(q) \widehat{\otimes} \Lambda$ as coefficients,
we extend the Mukai pairing to $\ca{K}$ via linearity. There is a natural
$\widehat{\Lambda}$-valued symplectic form $\Omega$ on $\ca{K}$ defined by
\[
  \Omega(f, g): = [\text{Res}_{q = 0} + \text{Res}_{q = \infty}](f(q), g(q^{-1}))
  \frac{dq}{q}, \quad \text{where} \ f, q \in \ca{K}.
\]

With respect to $\Omega$, there is a \emph{Lagrangian polarization} $\ca{K} =
\ca{K}_ + \oplus \ca{K}_-$, where
\[
  \ca{K}_ + = (K(\bar{I}X) \otimes \bb{C}[q, q^{-1}]) \widehat{\otimes} \Lambda
  \quad \text{and} \quad \ca{K}_- = \{f \in \ca{K}|f(0) \neq \infty, f(\infty) =
  0 \}.
\]
In other words, $\ca{K}_ + $ is the space of $K(\bar{I}X) \otimes
\Lambda$-valued Laurent polynomials in $q$ (in the $\Lambda_ + $-adic sense)
and $\ca{K}_-$ consists of \emph{proper} rational functions of $q$ regular at 0.
For $f(q) \in \ca{K}$, we write $f(q)=[f(q)]_++[f(q)]_-$, where
$[f(q)]_+\in\ca{K}_+$ and $[f(q)]_-\in\ca{K}_-$.

In a series of work \cite{givental11, givental12, givental13, givental14,
  givental15, givental16, givental17, givental18, givental19, givental21,
  givental22}, Givental introduced and studied the \emph{permutation-equivariant}
quantum $K$-theory,
which takes into account the $S_n$-action on the moduli spaces of stable maps
by
permuting the markings. We recall its natural generalization to the quasimap
setting \cite{Tseng-You}.

% We fix a finite basis $\{\phi_a \}$ of $K(I_\mu X)$ and denote by $\{\phi^a
% \}$ the dual basis with respect to the Mukai pairing $(-, -)$.

Let $\ev_i:Q^\epsilon_{g, n}(X, \beta) \rightarrow \bar{I} X$ be the rigidified
evaluation map~\eqref{eq:rigidified-evaluation-map} at the $i$-th marked point. We denote by $L_i$ the $i$-th
cotangent line bundle at the $i$-th marked point of \emph{coarse} curves.
Consider the natural $S_n$-action on the quasimap moduli space
$Q^\epsilon_{g, n}(X, \beta)$ by permuting the $n$ markings. Then for an
arbitrary Laurent polynomial $\mb{t}(q) = \sum_{j \in \bb{Z}} \mb{t}_j q^j \in
\ca{K}_ + $, there is a natural (virtual) $S_n$-module
\[
  \big[\mb{t}(L), \dots, \mb{t}(L) \big]^\epsilon_{g, n, \beta}: =
  \sum_{m \geq 0}(-1)^mH^m \big(Q^\epsilon_{g, n}(X, \beta),
  \ca{O}_{Q^\epsilon_{g, n}(X, \beta)}^{\text{vir}} \cdot \prod_{i = 1}^n
 \ev_i^*({\bf t}(L_i)) \big),
\]
where $\ev_i^*({\bf t}(L_i)):=\sum_{j} \ev_i^*(\mb{t}_j)L_i^j$.
The permutation-equivariant invariant is defined as the dimension of its
$S_n$-invariant submodule. Equivalently, we have the following
\begin{definition} \label{definition-invariant}
  The correlator of the permutation-equivariant $K$-theoretic $\epsilon$-stable
  quasimap invariants is defined by
  \[
    \big \langle \mb{t}(L), \dots, \mb{t}(L) \big \rangle_{g, n, \beta}^{S_n,
      \epsilon}: = p_* \big(\ca{O}_{Q^\epsilon_{g, n}(X, \beta)}^{\text{vir}} \cdot
    \prod_{i = 1}^n \ev_i^*({\bf t}(L_i)) \big),
  \]
  where $p_*$ is the proper pushforward along the projection
  \[
    p: \big[Q^\epsilon_{g, n}(X, \beta)/S_n \big] \rightarrow \text{Spec} \,
    \bb{C}.
  \]
\end{definition}
By definition, $K$-theoretic invariants are integers (if $\mb{t}(q)$ has
integral coefficients).
\iffalse
\Ming{There is a subtlety I overlooked before: Since $\mb{t}$ lies in the
  topological K-group. The pushforward should also be in the topological setting.
  There is such a notion, but we want to prove the localization formula in the
  algebraic setting. What is a way to get around this problem?}
\fi
Definition \ref{definition-invariant} can be easily generalized to the case
when there are different insertions.
Suppose that we are given several Laurent polynomials $\mb{t}^{(a)} = \sum_{m}
\mb{t}^{(a)}_mq^m$ with $a = 1, \dots s$. Let $(k_1, \dots, k_s)$ be a
partition of $n$. Then we generalize Definition \ref{definition-invariant} and
define permutation-equivariant $\epsilon$-stable quasimap $K$-invariants with
symmetry group $S_{k_1} \times \dots \times S_{k_s}$ by
\begin{align*}
  \big \langle \mb{t}^{(1)}, \dots, \mb{t}^{(1)}, \mb{t}^{(2)}, \dots,
  \mb{t}^{(2)}, \dots, \mb{t}^{(s)}, \dots, \mb{t}^{(s)} \big \rangle_{g, n,
  \beta}^{S_{k_1} \times \dots \times S_{k_s}, \epsilon}: = \\ \pi_*
  \big(\ca{O}_{Q^\epsilon_{g, n}(X, \beta)}^{\text{vir}} \cdot \prod_{a = 1}^s
  \prod_{i = 1}^{k_a} \big(\sum_{m} \ev_i^*(\mb{t}^{(a)}_m)L_i^m \big)
  \big),
\end{align*}
where $\pi_*$ is the proper pushforward along the projection
\[
  \pi: \big[Q^\epsilon_{g, n}(X, \beta)/S_{k_1} \times \dots \times S_{k_s} \big]
  \rightarrow \text{Spec} \, \bb{C}.
\]

According to \cite[Example 5]{givental11}, we have the
\emph{permutation-equivariant binomial formula}:
\[
  \big \langle \mb{t} + \mb{t}', \dots, \mb{t} + \mb{t}' \big \rangle_{g, n,
    \beta}^{S_n, \epsilon} = \sum_{i + j = n} \big \langle \mb{t}, \dots, \mb{t},
  \mb{t}', \dots, \mb{t}' \big \rangle_{g, n, \beta}^{S_i \times S_j, \epsilon},
  \quad \mb{t}, \mb{t}' \in \ca{K}_ + .
\]
It follows from the following equality of $S_n$-modules:
\[
 \big [ \mb{t} + \mb{t}', \dots, \mb{t} + \mb{t}' \big ]_{g, n,
    \beta}^{S_n, \epsilon} = \sum_{i + j = n} \mathrm{Ind}^{S_n}_{S_i\times S_j}\big [ \mb{t}, \dots, \mb{t},
  \mb{t}', \dots, \mb{t}' ]_{g, n, \beta}^{S_i \times S_j, \epsilon}.
\]
Here $\mathrm{Ind}^{G}_H$ denotes the operation of inducing a $G$-module from an
$H$-module. For any $H$-module $V$, we have $(\operatorname{Ind}^G_HV)^G=V^H$. By using the above binomial formula and induction on $k$, one can prove the
following \emph{permutation-equivariant multinomial formula}:
\[
  \big \langle \sum_{i = 1}^m \mb{t}_i, \dots, \sum_{i = 1}^m \mb{t}_i \big
  \rangle_{g, n, \beta}^{S_n, \epsilon} =
  \sum_{k_1 + k_2 + \cdots + k_m = n}
  \big \langle \mb{t}_1, \dots, \mb{t}_1, \mb{t}_2, \dots, \mb{t}_2, \dots,
  \mb{t}_m, \dots, \mb{t}_m \big \rangle_{g, n, \beta}^{S_{k_1} \times S_{k_2}
    \times \cdots S_{k_m}, \epsilon},
\]
where $\mb{t}_1, \dots \mb{t}_m \in \ca{K}_ + $. In general, suppose we have an
$S_n$-equivariant proper morphism
$
\tilde{\pi}:Q^\epsilon_{g, n}(X, \beta)
\rightarrow Y,
$
where $Y$ is a Deligne-Mumford stack on which $S_n$ acts trivially. The permutation-equivariant
multinomial formula still holds if we replace $\langle \cdot \rangle_{g, n, \beta}^{S_{k_1} \times S_{k_2}
  \times \cdots S_{k_m}, \epsilon}$ by the proper pushforward along the induced morphism
\[
\pi:\big[Q^\epsilon_{g, n}(X, \beta)/S_{k_1} \times \dots \times S_{k_s} \big]
  \rightarrow Y.
\]
%\Ming{In Section \ref{section-wall-crossing}, we will use a slightly more
 % general identity to deal with $K$-classes like $\sum_{\vb}*$, but I haven't
 % figured out the shortest and clearest way to state the identity and prove it.
  %I mentioned vaguely that some of the identities follow from the
 % permutation-equivariant binomial and multinomial formulas.}

%\Ming{What I omit in the actual computation in the last section are the steps of
%converting between ordered invariants and unordered invariants. It is more
%symmetric using the ordered invariants and the notation is lighter.}

%Let
%$\mathbf{t}_{i}^{(j)},(\mathbf{t}')_{i}^{(j)}$ be two sequences of elements in
%$\ca{K}_+$, where $i=1,\dots,m$ and $j$ takes values in a subset of $\Z$ which
%is bounded below. Set $\mathbf{t}_{i}:=\sum_j\mathbf{t}_{i}^{(j)}$ and $(\mathbf{t}')_{i}:=\sum_j(\mathbf{t}')_{i}^{(j)}$. We have
%\begin{align*}
 % & \big \langle
  %  \mb{t}_{1}+(\mb{t}')_1,
%  \dots,
 % \mb{t}_{1}+(\mb{t}')_1
%  ,
%\mb{t}_{2}+(\mb{t}')_2
%\dots,
%\mb{t}_{2}+(\mb{t}')_2 
%,
%\dots,
%\mb{t}_{m}+(\mb{t}')_m 
%, \dots,
%\mb{t}_{m}+(\mb{t}')_m
%\big \rangle_{g, n, \beta}^{S_{k_1} \times S_{k_2}
%  \times \cdots S_{k_m}, \epsilon}\\
%  =&\sum_{N\subsetneq\{1,\dots,m\}}
%     \big \langle
%s
%     \big \rangle_{g, n, \beta}^{S_{k_1} \times S_{k_2}
 % \times \cdots S_{k_m}, \epsilon}
%\end{align*}

Define the \emph{genus-$g$ descendant potential} of $X$:
\begin{equation}\label{eq:equation-potential}
  F^{\epsilon}_{g}(\mathbf t(q)) = \sum_{n = 0}^\infty \sum_{\beta \geq
    0}Q^\beta \langle \mathbf t(L)
  , \ldots, \mathbf t(L) \rangle^{S_n, \epsilon}_{g, n, \beta}
\end{equation}

We now introduce two genus-0 generating series, the $J$-function and the
$I$-function, using $K$-theoretic residues on the graph space $QG^{0+}_{0,
  1}(X, \beta)$. Let \begin{equation*}
  \check{\ev}_{\star}:QG^{0+}_{0, 1}(X, \beta) \rightarrow
  \bar{I}X
\end{equation*}
be the rigidified evaluation map at the unique marking $x_{\star}$ composed with the
involution on
$\bar{I}X$. Let $\mb{r}$ be the locally constant function on $\bar{I}X$ that
takes value $r$ on the component $\bar{I}_rX$ and set $\mb{r}_\star =
(\check{\ev}_{\star})^*(\mb{r})$.
\begin{definition} \label{K-theoretic-I-function}
\begin{enumerate}
\item  Define the permutation-equivariant $K$-theoretic big $J$-function by
%\begin{align*}
 % J_{S_\infty}^{\epsilon}(\mathbf t(q), Q): =
 % &
 %   1-q + \mathbf t(q) + \\
% &
% + (1-q^{-1})(1-q)\sum_{(k,\beta)\neq (0,0),(1,0)}
 %  Q^\beta (\check{\mathrm{ev}}_{\star})_{*}
  % \left(
%\frac{\theta^{\mb{r}_\star}(\tilde{L}_\star) \cdot \ca{O}_{F^{k,
 %         \beta}_{\star, 0}}^{\text{vir}}}
  %  {\lambda_{-1}^{\C^*} \big(N^\vee \big)}\cdot\prod_{i=1}^k\mathrm{ev}_i^*(\mb{t}(L_i))
  % \right),
%\end{align*}

\begin{align*}
  J_{S_\infty}^{\epsilon}(\mathbf t(q), Q): =
  &
    1-q + \mathbf t(q) 
 + (1-q^{-1})(1-q)\sum_{0<\operatorname{\beta}\leq 1/\epsilon}
    Q^\beta (\check{\ev}_{\star})_{*}
     \left(
    \frac{\ca{O}_{F^{0,
          \beta}_{\star, 0}}^{\text{vir}}}
    {\lambda_{-1}^{\C^*} \big(N_{F^{0,
          \beta}_{\star, 0}/QG^{0+}_{0,1}(X,\beta)}^{\mathrm{vir},\vee} \big)}
    \right)
  \\
  &\quad +\hspace{1cm}
    \sum_{\substack{(k \geq1, \beta \geq 0), \, (k, \beta) \neq(1, 0)\,\mathrm{or} \\k=0,\mathrm{deg}(\beta)>1/\epsilon}}
  Q^\beta (\check{\ev}_{1})_{*}
   \left(
\frac{ \ovir_{Q^{\epsilon}_{0,1+k}(X,\beta)}}
    {1-qL_1}\cdot\prod_{i=1}^k\mathrm{ev}_i^*(\mb{t}(L_i)).
  \right).
\end{align*}
where
%$N: = N^{\vir}_{F^{k, \beta}_{\star, 0}/QG_{0, 1+k}^{\epsilon}(X, \beta)}$
 % denotes the virtual normal bundle of the fixed locus
 % $F^{k, \beta}_{\star, 0}$ in $QG^{\epsilon}_{0, 1+k}(X, \beta)$,
$\tilde{L}_\star$ is
  the
  line bundle formed by the relative orbifold cotangent space at $x_\star$.
\item Define the $K$-theoretic (small) $I$-function by
  \[
I(Q,q):=J_{S_\infty}^{\epsilon=0+}(0, Q)/(1-q)= 1 + (1-q^{-1}) \cdot \sum_{\beta>0}Q^\beta  (\check{\ev}_{\star})_{*}
    \left(
    \frac{ \ca{O}_{F^{0,
          \beta}_{\star, 0}}^{\text{vir}}}
    {\lambda_{-1}^{\C^*} \big(N_{F^{0,
          \beta}_{\star, 0}/QG^{0+}_{0,1}(X,\beta)}^{\mathrm{vir},\vee} \big)}
    \right).
  \]
\end{enumerate}

\end{definition}
When $\epsilon=\infty$, the definition of $J$-function coincides with that in
quantum $K$-theory of orbifolds introduced in~\cite{Tonita1}.

In general, by the $K$-theoretic localization theorem (Theorem \ref{K-theoretic-localization}), the action of the $K$-theoretic Euler class
$\lambda_{-1}^{\C^*} \big(N_{F^{0,
          \beta}_{\star, 0}/QG^{0+}_{0,1}(X,\beta)}^{\mathrm{vir},\vee} \big)$ on $K^{\C^*}_\circ(F^{k, \beta}_{\star, 0})\otimes_{\Q[q,q^{-1}]}\Q(q)$ is invertible. Similarly, the action of $1-qL_{1}$ on $K^{\C^*}_\circ(Q^{\epsilon}_{0,1+k})\otimes_{\Q[q,q^{-1}]}\Q(q)$ is invertible. Indeed, this follows by applying~\cite[Proposition 5.13]{Kiem-Savvas} to the embedding $Q^{\epsilon}_{0,1+k}\hookrightarrow L_{1}$ as the zero section.
          Hence by definition, the $J$-function and $I$-function are elements of the loop space $\ca{K}$.
We define
\[
  \mu_{\beta}(q) \in K(\bar{I}X) \otimes \bb{C}[q, q^{-1}]
\]
to be the coefficient of $Q^\beta$ in $[(1-q)I(Q, q)-(1-q)]_ + $. For $\epsilon
\in \bb{Q}_{\geq 0} \cup \{0, \infty \}$, we define
\[
  \mu^{\geq \epsilon}(Q, q) = \sum_{0< \deg(\beta)\leq 1/ \epsilon}
  \mu_\beta(q)Q^{\beta}.
\]
If we write
\[
  I(Q, q) = 1 + \sum_{\beta>0}I_{\beta}(q)Q^\beta,
\]
then $\mu_{\beta}(q) = [(1-q)I_\beta(q)]_ + $.

\begin{remark} \label{remark-small-I-function}
  The $K$-theoretic small $I$-functions are of $q$-hypergeometric-type and
  explicitly computable in general. For example, the explicit formulas of
  $I$-functions for toric bundles and toric complete intersections can be found
  in \cite{givental15}. See also~\cite[Example 5.3]{Gonzalez-Woodward}. It is shown in \cite{RZ1} that $K$-theoretic
  $I$-functions with level structure recovers almost all Ramanujan's mock theta
  functions.
  The formulas of $I$-functions for type A flag manifolds can be found
  in \cite{Taipale} and the $I$-function of the Grassmannian with nontrivial level
  structure is studied in~\cite{Dong-Wen} and~\cite{Givental-Xiaohan}. For general nonabelian GIT quotients, one can use the
  $K$-theoretic abelian/nonabelian correspondence \cite{Wen} to compute their
  $K$-theoretic $I$-functions.

  As shown recently in~\cite{Ruan-Wen-Zhou},
  $I$-functions with non-trivial level structures play an important role in understanding 3d
  $\ca{N}=2$ mirror symmetry. 
\end{remark}

\subsection{Twisted theory and level structure} \label{twisted-theory}

One can consider two types of twistings in $K$-theoretic quasimap theory. The
first type was introduced in \cite{Tonita2, givental22} and includes Eulerian
twistings as special cases. It can be used to compute $K$-theoretic invariants
of total spaces of vector bundles or zero loci of sections of vector bundles.
The second type was introduced in \cite{RZ1}, which is related to Verlinde
algebras. We will give a common generalization here.

For each non-zero integer $m$, fix an element $E^{(m)} \in K^\circ_G(W)$. Let
$\pi: \ca{C} \rightarrow Q^\epsilon_{g, n}(X, \beta)$ be the universal curve,
with universal principal $G$-bundle $\ca{P}$ on it. Let $u: \ca{C} \rightarrow
\ca{P} \times_GW$ be the universal section. We define the following $K$-theory
class
\begin{equation} \label{eq:twisting-bundle} E^{(m)}_{g, n, \beta}: = R^\bullet
  \pi_*u^* \big(\ca{P} \times_G E^{(m)} \big)
\end{equation}
on $Q^\epsilon_{g, n}(X, \beta)$. To avoid potential divergence issues of
twisted virtual structure sheaves and invariants, we consider fiberwise scalar
actions by $\bb{C}^*$ on the classes $E^{(m)}$. We choose two copies of $\C^*$
whose equivariant parameters are denoted by $t_+$ and $t_-$. Fix an integer $a$.
Let $\C_{t_+^a}$ denote the $a$-th tensor power of the standard representation
of $\C^*$.
%\Yang{ What are $\lambda_\pm$? There seems to be too many $\lambda$:
 % $\lambda_0 , \ldots, \lambda_N$, $\lambda_{-1}$, $\lambda_{\pm}$. } \Ming{I
 % made some modifications.}
For $m>0$, we apply the construction
(\ref{eq:twisting-bundle}) to the $G \times \bb{C}^*$-equivariant bundle
$E^{(m)} \otimes \bb{C}_{t_ +^a}$ and obtain
\[
  E^{(m)}_{g, n, \beta}(t_ +) \in K^\circ(Q^\epsilon_{g, n}(X, \beta))
  \otimes\Q[t_ + , t_ + ^{-1}].
\]
Similarly, for $m<0$, we construct a class
\[
  E^{(m)}_{g, n, \beta}(t_-) \in K^\circ(Q^\epsilon_{g, n}(X, \beta)) \otimes
  \Q[t_-, t_-^{-1}]
\]
using $E^{(m)} \otimes \bb{C}_{t_ -^b}$, for some fixed integer $b$.

To introduce the determinantal twisting or \emph{level structure} (c.f.
\cite{RZ1}), we fix another $G$-equivariant bundle $R$ on $W$ and an integer
$l$.\footnote{In general, the theory works for any rational number $l$ such
  that the determinant line bundle $\ca{D}^{R, l}$ exists.} For example, we
usually choose $R$ to be of the form $W \times \tilde{R}$, where $\tilde{R}$ is
a finite-dimensional $G$-module. We define the level-$l$ determinant line bundle
over $Q^\epsilon_{g, k}(X, \beta)$ by
\begin{equation}\label{eq:equation-def-level-structure}
  \ca{D}^{R, l}: = \big(\text{det} \, R \pi_*(R_{g, n, \beta}) \big)^{-l}.
\end{equation}
For $l \neq0$, the level structure corresponds to a specific choice of
non-trivial Chern–Simons level in 3d $\ca{N} = 2$ supersymmetric gauge theory
\cite{Jockers-Mayr1, Jockers-Mayr3, Ueda-Yoshida}. It would be interesting to
find the geometric construction of twists corresponding to generic Chern-Simons
levels (c.f. \cite[\textsection{4.1}]{Ueda-Yoshida}).

The $K$-group $K^\circ(Q^\epsilon_{g,n}(X,\beta))$ has a natural
$\lambda$-ring structure and we denote its Adams operations by
$\Psi^m,m>0$. For the sake of applications, we extend the Adams operations for negative values of $m$ by
$\Psi^{m}(V):=\Psi^{-m}(V^\vee)$ for any $K$-theory class $V$. We also set $\Psi^m(t_\pm)=t_\pm^m$. 
By combining~\eqref{eq:twisting-bundle} and~\eqref{eq:equation-def-level-structure},
we define the twisting class
\begin{equation}\label{eq:combined-twisting-class}
  \ca{T}^{\mb{E}^{(\bullet)}, R, l}_{g,n,\beta}:=\text{exp}
  \Big(\sum_{m<0} \Psi^m \big(E^{(m)}_{g, n, \beta}(t_-) \big) + \sum_{m>0}
  \Psi^m \big(E^{(m)}_{g, n, \beta}(t_ +) \big) \Big) \cdot \ca{D}^{R, l}
\end{equation}
and the twisted virtual
structure sheaf of level $l$
  
\begin{equation} \label{eq:twisted-virtual-structure-sheaf} \ca{O}^{\vir,
    \mb{E}^{(\bullet)}, R, l}_{Q^\epsilon_{g, n}(X, \beta)} : =
  \ca{O}^{\vir}_{Q^\epsilon_{g, n}(X, \beta)} \cdot \ca{T}^{\mb{E}^{(\bullet)},
    R, l}.
\end{equation}
The twisting class $\ca{T}^{\mb{E}^{(\bullet)}, R, l}_{g,n,\beta}$ is an element
in $K^\circ(Q^\epsilon_{g, n}(X, \beta))
  \widehat{\otimes}\Q[t_\pm , t_\pm ^{-1}]$ where the completion depends on the
  signs of $a$ and $b$. For example, if we choose $a=b=1$, then the twisting
  class lies in $K^\circ(Q^\epsilon_{g, n}(X, \beta))_{\Q}[[t_\pm ]]$, the ring of
  formal power series in $t_+$ and $t_-$ with coefficients in $K^\circ(Q^\epsilon_{g, n}(X, \beta))_{\Q}$.

The twisted invariants of level $l$ are defined by the same formula as in
Definition \ref{definition-invariant}, i.e., we define
\[
  \big \langle \mb{t}(L), \dots, \mb{t}(L) \big \rangle_{g, n, \beta}^{S_n,
    \mb{E}^{(\bullet)}, R, l, \epsilon}: = p_* \big(\ca{O}_{Q^\epsilon_{g, n}(X,
    \beta)}^{\text{vir}, \mb{E}^{(\bullet)}, R, l} \cdot \prod_{i =
    1}^n\mathrm{ev}_i^*( \mb{t}(L_i)) \big),
\]
where $p$ denotes the unique morphism from $\big[Q^\epsilon_{g, n}(X, \beta)/S_n
\big]$ to $\text{Spec} \, \bb{C}$. The pairing on the $K$-group is modified
accordingly. To be more precise, let $\widetilde{E^{(m)}}$ and $\widetilde{R}$ denote
the induced
%\marginpar{\footnotesize{ \Yang{How? Is any kind of
 %     twisting by the normal bundle happening here?} \Ming{The twisting bundles
  %    are defined over the quotient stack $[W/G]$. The components of the inertia
      %stack are substacks. The K-theory classes are induced via restriction. I
      %don't think we need to twist the normal bundle. I will double check
      %Tseng's paper on orbifold quantum Lefschetz.} }}
$K$-theory classes on $IX$. We define the twisting class
\begin{equation}\label{eq:classical-twisting-class}
  T:=\mathrm{exp} \big(\sum_{m \neq0}
  \Psi^m(\widetilde{E^{(m)}}))/m \big) \cdot(\det \, \widetilde{R} )^{-l}\in K(IX).
\end{equation}
Recall that $\varpi:IX\rightarrow \bar{I}X$ is the natural projection. We define
the twisting class $\overline{T}=\varpi_*(T)$ in $K(\bar{I}X)$.
The twisted $K$-theoretic pairing of level $l$ is defined by
\[
  (\alpha, \beta)^{R, l}_{\mb{E}^{(\bullet)}} = \chi \Big(I X, \alpha \cdot
  \iota^* \beta \cdot \overline{T}\Big),\quad \alpha, \beta \in K(\bar{I}X)
\]

%As before, we assume this twisted pairing is
%{\color{red}non-degenerate}\Ming{The previous sentence is to be deleted}.
%\marginpar{\footnotesize{ \Yang{ There are parameters $\lambda_{\pm}$. What does
 %     non-degenerate mean? Does it mean for non-degenerate for generic
  %    $\lambda_{\pm}$ or what? } \Ming{My original motivation to include this
   %   sentence is to define J-functions. Now I don't need the pairing in the
    %  definition. I will delete it.} }}

We define the genus-$g$ twisted potential of level $l$
by~\eqref{eq:equation-potential} in which all invariants are replaced by their
twisted counterparts. The definition of the twisted $J$-function of level $l$ is
also similar to that in the untwisted theory. Note that the twisting
class~\eqref{eq:combined-twisting-class} can be defined similarly over genus-0
graph spaces. By abuse of notation, we still denote it by
$\ca{T}^{\mb{E}^{(\bullet)}, R, l}_{0,n,\beta}$. Using the notation introduced
in Definition \ref{K-theoretic-I-function}, we define the twisted small $I$-function of level $l$ by
\begin{equation*}
  I^{{\bf E}^{(\bullet)},R,l}(Q,q)
  :  = 1 + (1-q^{-1}) \cdot \overline{T}^{-1}\cdot \sum_{\beta>0}Q^\beta  (\check{\ev}_{\star})_{*}
     \left(
     \frac{ \ca{O}_{F^{0,
     \beta}_{\star, 0}}^{\text{vir},{\bf E^{(\bullet)}},R,l}}
     {\lambda_{-1}^{\C^*} \big(N_{F^{0,
     \beta}_{\star, 0}/QG^{0+}_{0,1}(X,\beta)}^{\mathrm{vir},\vee} \big)}
     \right).
\end{equation*}
Note that in the above formula, the restriction of the twisting class $ \ca{T}^{\mb{E}^{(\bullet)}, R, l}_{0,1+k,\beta}\big|_{F^{0, \beta}_{\star,
    0}}$ can have non-trivial $\C^*$-weights. Let $I_\beta^{{\bf
    E}^{(\bullet)},R,l}(q)$ be the coefficient of $Q^\beta$ in $ I^{{\bf E}^{(\bullet)},R,l}(Q,q)$. The twisted $K$-theoretic big
$J$-function of level $l$ is defined by
\begin{align*}
  J_{S_\infty}^{\epsilon,{\bf E^{(\bullet)}},R,l}(\mathbf t(q), Q)
  =
  &
    1-q + \mathbf t(q) +
    (1-q)\sum_{0< \operatorname{deg}(\beta)\leq 1/\epsilon}I_\beta^{{\bf E}^{(\bullet)},R,l}(q)Q^\beta
  \\
  +&
     \overline{T}^{-1}\sum_{\substack{
     (k \geq1, \beta \geq 0), \, (k, \beta) \neq(1, 0)\,\mathrm{or} \\k=0,\mathrm{deg}(\beta)>1/\epsilon}}
  Q^\beta (\check{\ev}_{1})_{*}
     \left(
     \frac{\ca{O}_{Q^{\epsilon}_{0,1+k}(X,\beta)}^{\text{vir},{\bf E^{(\bullet)}},R,l}}
     {1-qL_1}\cdot\prod_{i=1}^k\mathrm{ev}_i^*(\mb{t}(L_i))
     \right).
\end{align*}
By definition, we have $ I^{{\bf E}^{(\bullet)},R,l}(Q,q)= J_{S_\infty}^{\epsilon=0+,{\bf E}^{(\bullet)},R,l}(0, Q)/(1-q)$.
%\begin{align*}
 % &J_{S_\infty}^{\epsilon,{\bf E^{(\bullet)}},R,l}(\mathbf t(q), Q)\\
 % : =
 % &
 %   1-q + \mathbf t(q)  \\
 % +&
 %    (1-q^{-1})(1-q)T^{-1} \cdot\sum_{(k,\beta)\neq (0,0),(1,0)}
  %   Q^\beta(\check{\mathrm{ev}}_{\star})_{*}
   %  \left(
    % \frac{\theta^{\mb{r}_\star}(\tilde{L}_\star) \cdot \ca{O}_{F^{k,
    % \beta}_{\star, 0}}^{\text{vir},{\bf E^{(\bullet)}},R,l}}
    % {\lambda_{-1}^{\C^*} \big(N^\vee \big)}\cdot\prod_{i=1}^k\mathrm{ev}_i^*(\mb{t}(L_i))
    % \right),
%\end{align*}
%where
%\[
 % \ca{O}_{F^{k,\beta}_{\star, 0}}^{\text{vir},{\bf E}^{(\bullet)},R,l}
  %:=\ca{O}_{F^{k, \beta}_{\star, 0}}^{\text{vir}}\cdot
  %\ca{T}^{\mb{E}^{(\bullet)}, R, l}_{0,1+k,\beta}\big|_{F^{k, \beta}_{\star,
 %     0}}.
%\]

An important property of the twisting $K$-theory class is that it factorizes
``nicely'' over nodal strata. To be more precise, let $(C, P, u)$ be a quasimap
node
$p$ into two connected components $C'$ and $C''$. Consider the normalization
exact sequence on $C$:
\[
  0 \rightarrow \ca{O}_C \rightarrow \ca{O}_{C'} \oplus \ca{O}_{C''} \rightarrow
  \ca{O}_p \rightarrow0
\]
coming from splitting $C$ at the node $p$. This gives rise to the following
identity in $K$-theory:
\begin{equation}\label{eq:equation-normalization-sequence}
  H^*(C, P \times_G E) \oplus \bar{E}_{u(p)} = H^*(C', P|_{C'} \times_G E) \oplus
  H^*(C'', P|_{C''} \times_G E)
\end{equation}
for any $E \in K^\circ_G(W)$. Since the two operations
$\text{exp}(\Psi^m(\cdot)/m)$ and $\text{det}(\cdot)$ are
additive-multiplicative, the twisting $K$-theory class in
  (\ref{eq:twisted-virtual-structure-sheaf}) factors. We refer the reader to
  \cite[Proposition 2.9]{RZ1} for more details in the case of level structure.
  %\marginpar{\footnotesize{ \Yang{I kinda know what you are trying to say here but
   %   it does not seem very clear.} \Ming{The precise statement can be found in
    %  Proposition 2.9 of my Arxiv paper. arXiv:1804.06552. I am not sure if I
     % want to spell out the precise definition here. But I probably should
      %include a reference.} }}

%%% Local Variables:
%%% mode: latex
%%% TeX-master: "main"
%%% End:

%\input{entangled-tails}

\section{$K$-theoretic quasimap invariants with entangled tails}
\label{section-entangled-tails}
To construct the master space, the second author introduced weighted twisted
curves and quasimaps with entangled tails  in \cite{Zhou2}. We review their
definitions in this section and study the properties of (twisted) virtual
structure sheaves on their moduli spaces.

\subsection{The entanglement of weighted twisted curves}
\label{sec:weighted-twisted-curves}

The space $\bb{Q}_{\geq 0} \cup \{0, \infty \}$ of stability conditions is
divided into chambers by walls $\{1/d\}$. Let $\epsilon_0 =
1/d_0$ be a wall. We fix non-negative integers $g, n, d$ such that
$2g-2 + n + \epsilon_0 d \geq 0$. When $g = n = 0$, in additional we require
that
$\epsilon_0 d>2$.

We consider $n$-pointed weighted twisted
curves of genus $g$ and degree $d$. Here ``weighted'' means the assignment of a
nonnegative integer to each
irreducible component of twisted curves (c.f. \cite{costello2006higher,
  hu2010genus}). These integers will correspond to the degrees of quasimaps.
Hence we will also refer to them as degrees. We suppress the degrees from the
notation and denote an $n$-pointed weighted twisted curve by $(C, x_1 , \ldots,
x_n)$, or by
$(C, \mathbf x)$ for short. From now on, all curves are weighted twisted curves
unless otherwise specified.

A rational \emph{tail} (resp. \emph{bridge}) of $(C, \mathbf x)$ is a smooth
rational irreducible component of $C$ whose normalization has one (resp. two)
special points, i.e. \ the preimage of nodes or markings. Let $\mathfrak
M^{\mathrm{wt}}_{g, n, d}$ be the moduli stack of $n$-pointed weighted twisted
curves of genus $g$ and degree $d$.
\begin{definition}
  We define the open substack stack $\mathfrak M^{\mathrm{wt}, \mathrm{ss}}_{g,
    n, d} \subset \mathfrak
  M^{\mathrm{wt}}_{g, n, d}$ of $\epsilon_0$-semistable  weighted curves by the
  following conditions:
  \begin{itemize}
  \item the curve has no degree-$0$ rational bridge,
  \item the curve has no rational tail of degree strictly less that $d_0$.
  \end{itemize}
\end{definition}
Note that $\mathfrak M^{\mathrm{wt}, \mathrm{ss}}_{g, n, d}$ is a smooth Artin
stack of dimension $3g-3 + n$.

We recall the blowup construction of the moduli stack of
$\epsilon_0$-semistable curves with entangled tails in \cite{Zhou2}. Set
\[
  m = \lfloor d/d_0 \rfloor, \quad \mathfrak U_m = \mathfrak M^{\mathrm{wt},
    \mathrm{ss}}_{g, n, d}.
\]
The integer $m$ is the maximum of the number of degree-$d_0$ rational tails.
Let $\mathfrak Z_i \subset \mathfrak U_m$ be the reduced closed substack
parametrizing curves with at least $i$ rational tails of degree $d_0$. Note
that $\mathfrak Z_{1}$ is a normal crossing divisor and $\mathfrak Z_{i}$ is a
codimension-$i$ stratum in $\mathfrak U_m$.

We start with the deepest stratum $\mathfrak Z_{m}$ which is smooth. Let
\[
  \mathfrak U_{m-1} \rightarrow \mathfrak U_{m}
\]
be the blowup along $\mathfrak Z_{m}$ and let $\mathfrak E_{m-1} \subset
\mathfrak
U_{m-1}$ be the exceptional divisor. Inductively for $i = m-1, \ldots, 1$, let
\[ \mathfrak Z_{(i)} \subset \mathfrak U_{i}
\]
be the proper
transform of $\mathfrak Z_i$
and let
\[
  \mathfrak U_{i-1} \rightarrow \mathfrak U_{i}
\]
be the blowup along ${\mathfrak Z_{(i)}}$ with exceptional divisor $\mathfrak
E_{i-1} \subset \mathfrak U_{i-1}$.

\begin{definition}
  We call $\tilde{\mathfrak M}_{g, n, d}: = \mathfrak U_{0}$ the moduli stack
  of genus-$g$ $n$-marked $\epsilon_0$-semistable curves	of degree $d$ with
  entangled tails.
\end{definition}
We refer the reader to \cite[\textsection{2.3}]{Zhou2} for a concrete
description of the closed points in $\tilde{\mathfrak M}_{g, n, d}$.

Let
\[{\mathfrak M}^{\mathrm{wt}, \mathrm{ss}}_{g, n + k, d-kd_0} \times' {\left(
      \mathfrak
      M_{0, 1, d_0}^{\mathrm{wt}, \mathrm{ss}} \right)}^k
\]
denote the fiber product ${\mathfrak M}_{g, n + k, d-kd_0} \times_{\bb{N}^k}
{\left( \mathfrak
    M_{0, 1} \right)}^k$ formed by matching the sizes of the automorphism
groups at the last $k$-markings. Since our markings are trivialized gerbes, we
can define the gluing morphism
\[
  \mathrm{gl}_k: {\mathfrak M}_{g, n + k, d-kd_0}^{\mathrm{wt}, \mathrm{ss}}
  \times '{\left( \mathfrak
      M_{0, 1, d_0}^{\mathrm{wt}, \mathrm{ss}} \right)}^k
  \rightarrow
  \mathfrak Z_{k} \subset \mathfrak M^{\mathrm{wt, ss}}_{g, n, d}
\]
which glues the $k$ rational tails from ${\left( \mathfrak M_{0, 1,
      d_0}^{\mathrm{wt}, \mathrm{ss}} \right)}^k$ to the last $k$ markings of the
universal curve over ${\mathfrak M}_{g, n + k, d-kd_0}^{\mathrm{wt},
  \mathrm{ss}}$. Since ${\mathfrak M}_{g, n + k, d-kd_0}^{\mathrm{wt},
  \mathrm{ss}} \times' {\left( \mathfrak
    M_{0, 1, d_0}^{\mathrm{wt}, \mathrm{ss}} \right)}^k$ is smooth, the
gluing morphism factors through the
normalization $\mathfrak Z^{\mathrm{nor}}_{k}$ of $\mathfrak Z_{k}$.

For $1 \leq k \leq m$, we recall the structure of $\mathfrak Z_{(k)}$ from
\cite{Zhou2}. According to Lemma 2.2.2 and Lemma 2.2.3 in \cite{Zhou2}, there
is a unique fibered diagram
\begin{equation} \label{eq:gluing-morphism}
  \begin{tikzcd}
    \tilde{\mathfrak M}_{g, n + k, d-kd_0} \times' {\left( \mathfrak
        M_{0, 1, d_0}^{\mathrm{wt}, \mathrm{ss}} \right)}^k \arrow[r, "
    \tilde{\mathrm{gl}}_k"] \arrow[d] & \mathfrak Z_{(k)} \arrow[d] \\
    \mathfrak M^{\mathrm{wt, ss}}_{g, n + k, d-kd_0} \times' {\left(
        \mathfrak
        M_{0, 1, d_0}^{\mathrm{wt}, \mathrm{ss}} \right)}^k \arrow[r, "
    \mathrm{gl}^{\mathrm{nor}}_k" ] & \mathfrak Z_{k}^{\mathrm{nor}}
  \end{tikzcd},
\end{equation}
where the horizontal arrows are \'etale of degree
$k!/ \prod_{i = n + 1}^{n + k} \mathbf r_i$. Here $\mathbf r_i$ is the order
of the automorphism group at the $(n + i)$-th marking of $\mathfrak
M^{\mathrm{wt, ss}}_{g, n + k, d-kd_0}$.

\subsection{The boundary divisors of $\tilde{\mathfrak M}_{g, n, d}$ and
  inflated projective bundles} \label{inflated-bundle-boundary-divisor}

By the construction of $\tilde{\mathfrak M}_{g, n, d}$ we have natural
projections
\[
  \tilde{\mathfrak M}_{g, n, d} \rightarrow \mathfrak U_{i}, \quad i = 0,
  \ldots, m.
\]
\begin{definition}
  Let $\xi$ be a geometric point of $\tilde{\mathfrak M}_{g, n, d}$ and let
  $\{E_1, \ldots, E_k \}$ be
  \textit{a} set of degree-$d_0$ rational tails of $\xi$. Then
  $E_1, \ldots, E_k$ are called \textit{entangled} tails of $\xi$ if
  \begin{enumerate}[(1)]
  \item the image of $\xi$ in $\mathfrak U_k$ lies in $\mathfrak Z_{(k)}$;
  \item the tails $E_1, \ldots, E_k$ are those from  ${\left( \mathfrak
        M_{0, 1} \right)}^k$ via the gluing morphism $\tilde{\mathrm{gl}}_k$;
  \item the image of $\xi$ in $\mathfrak U_i$ does not lie in $\mathfrak
    Z_{(i)}$ for any $i<k$.
  \end{enumerate}
\end{definition}

Let $\mathfrak D_{k-1} \subset \tilde{\mathfrak{M}}_{g, n, d}$ be the
proper transform of $\mathfrak E_{k-1}$. According to \cite[Lemma
2.5.4]{Zhou2}, the boundary divisor $\mathfrak D_{k-1}$ is the closure of the
locally closed reduced locus where there are exactly $k$ entangled tails. By
construction, we have morphisms $\mathfrak D_{k-1} \to \mathfrak
Z_{(k)}$ and
\[
  \tilde{\mathrm{gl}}_{k}^* \mathfrak D_{k-1}
  \rightarrow
  \tilde{\mathfrak M}_{g, n + k, d-kd_0} \times'
  {\left(\mathfrak M_{0, 1, d_0}^{\mathrm{wt}, \mathrm{ss}} \right)}^k.
\]
To describe the structure of the above morphism, the second author introduced
the notion of \emph{inflated projective bundles} in \cite{Zhou2}. We recall the
definition here.

Let $X$ be an algebraic stack and let $L_1, \ldots, L_k$ be line bundles on
$X$. Let
\[
  P = \mathbb P(L_1 \oplus \dots \oplus L_k) \rightarrow X
\]
be the projective bundle. Consider the coordinate hyperplanes
\[
  H_i = \mathbb P(L_1 \oplus \cdots \oplus \{0 \} \oplus \cdots \oplus L_k),
\]
where the $\{0 \}$ appears in the $i$-th place only.
The construction of the inflated projective bundle is analogous to that of
$\tilde {\mathfrak M}_{g, n, d}$. More specifically, for $i = 1, \ldots, k-1$,
let $Z_i \subset P$ be the union of the codimension-$i$
coordinate subspaces,
i.e. \
\[
  \textstyle
  Z_i = \bigcup H_{j_1} \cap \cdots \cap H_{j_i},
\]
where $\{j_1, \ldots, j_i \}$ runs through all subsets of $\{1, \ldots, k \}$
of size $i$.
First we set $P_{k-1} = \mathbb P(L_1, \cdots, L_k)$. Inductively for $i = k-1,
\cdots, 1$, let
$ Z_{(i)} \subset P_i$ be the proper transform of $Z_i$ and let
\[
  P_{i-1} \rightarrow P_i
\]
be the blowup along $Z_{(i)}$ with exceptional divisor $E_{i-1} \subset
P_{i-1}$.
\begin{definition}
  \label{def:inflated-projective-bundle}
  We call $\tilde {\mathbb  P }(L_1, \ldots, L_k) : = P_0 \to X$ the
  \textit{inflated projective bundle} associated to $L_1, \ldots, L_k$.
\end{definition}

We denote by $D_i \subset \tilde {\mathbb P }(L_1, \ldots, L_k)$ the proper
transforms of $E_i$, for $i = 0, \ldots, k-2$, and refer to it as the $i$-th
tautological divisor of the inflated projective bundle. Note that
$\tilde{\mathbb P}(L_1, \ldots, L_k)$ is smooth over $X$ of relative dimension
$k-1$ and $D_i$ are relative effective Cartier divisors. We denote by $\mathcal
O_{
  \tilde{\mathbb P}}(-1) $ the pullback of the tautological bundle on $P$.

For $i = 1, \ldots, k$, let $\Theta_i$ be the line bundles on
$\tilde{\mathfrak M}_{g, n + k, d-kd_0} \times' { (\mathfrak M_{0, 1,
    d_0}^{})}^k$ formed by
the tensor product of two orbifold tangent lines to the curves, one
at the
$(n + i)$-th marking of $\tilde{\mathfrak M}_{g, n + k, d-kd_0}$, and the other
at the unique marking of the $i$-th copy of $ \mathfrak
M_{0, 1}$.

\begin{lemma}[\cite{Zhou2}]
  \label{lem:str-D}
  The diagram
  \[
    \begin{tikzcd}
      \tilde{\mathrm{gl}}_k^* \mathfrak D_{k-1}  \arrow[r, "{\iota_{\mathfrak
          D}}"] \arrow[d] &
      \tilde{\mathfrak{M}}_{g, n, d} \arrow[d] \\
      \tilde{\mathrm{gl}}_k^* \mathfrak E_{k-1}  \arrow[r, "{\iota_{\mathfrak
          E}}"] & \mathfrak U_{k-1}
    \end{tikzcd}
  \]
  is a fibered diagram. The morphism
  \begin{equation}
    \label{eq:D-as-inf-proj-bdl}
    {\tilde{\mathrm{gl}}_k^*} \mathfrak D_{k-1} \rightarrow
    \tilde{\mathfrak M}_{g, n + k, d-kd_0}
    \times' {\left(\mathfrak M_{0, 1, d_0}^{\mathrm{wt}, \mathrm{ss}}
      \right)}^k
  \end{equation} realizes ${\tilde{\mathrm{gl}}_k^*} \mathfrak D_{k-1}$ as the
  inflated projective bundle
  \[
    \tilde{\mathbb P}: = \tilde{\mathbb P} \left( \Theta_1, \ldots , \Theta_k
    \right) \rightarrow
    \tilde{\mathfrak M}_{g, n + k, d-kd_0} \times' {\left(\mathfrak M_{0, 1,
          d_0}^{\mathrm{wt}, \mathrm{ss}} \right)}^k.
  \]
\end{lemma}

Abusing the notation, we still denote the pullback of $\Theta_i$ to
${\tilde{\mathrm{gl}}_k^*} \mathfrak D_{k-1}$ by itself. The pullbacks of the
boundary divisors $\mathfrak D_\ell$ to $\tilde{\mathrm{gl}}_k^* \mathfrak
D_{k-1}$ are described in Lemma 2.7.3 and Lemma 2.5.1 in \cite{Zhou2}. We
summarize them in the following lemma:
\begin{lemma} \label{lem:divisors1}
  Consider the morphism $\iota_{\mathfrak D}: \tilde{\mathrm{gl}}_k^* \mathfrak
  D_{k-1}  \rightarrow \tilde{\mathfrak{M}}_{g, n, d}$. Then we have
  \begin{enumerate}
  \item
    for $0 \leq \ell \leq k-2$, the divisor pullback $\iota_{\mathfrak
      D}^* \mathfrak D_\ell$ is equal to the
    $\ell$-th tautological divisor $D_\ell$ of the inflated projective bundle
    \eqref{eq:D-as-inf-proj-bdl};
  \item
    for $\ell \geq k$, the divisor pullback $\iota_{\mathfrak D}^* \mathfrak
    D_{\ell}$ is equal to $\operatorname{pr}_1^*(\mathfrak D^\prime_{\ell-k})$,
    where $\operatorname{pr}_1$ is the composition of the projections
    \[
      \mathrm{pr}_1: {\tilde{\mathrm{gl}}_k^*} \mathfrak D_{k-1}
      \rightarrow  \tilde{\mathfrak M}_{g, n + k, d-kd_0}
      \times' {\left(\mathfrak M_{0, 1, d_0}^{\mathrm{wt}, \mathrm{ss}}
        \right)}^k \rightarrow  \tilde{\mathfrak M}_{g, n + k, d-kd_0},
    \]
    and $\mathfrak D^{\prime}_{\ell-k}$ is the boundary divisor on
    $\tilde{\mathfrak M}_{g, n + k, d-kd_0}$;
  \item the pullback of the line bundle
      $\mathcal O_{\tilde{\mathfrak M}_{g, n, d}}(\mathfrak D_{k-1})$ is
      canonically isomorphic to
  %  \marginpar{\footnotesize{
   %     \Yang{This is not proved in \cite{Zhou2} so we need to prove it. This
    %      will require us to recall a lot of geometric details of the entangled
     %     tail construction in this paper, which is kind of a pain in the ass.
      %  }
       % \Ming{We can just say that we use the notation as in your other paper.}
      %}}
    \[
      \mathcal O_{ \tilde{\mathbb P}}(-1) \otimes  \iota_{\mathfrak
        D}^* \mathcal O_{\tilde{\mathfrak M}_{g, n, d}} (- \sum_{i = k}^{m-1}
      \mathfrak
      D_{i}).
    \]
  \end{enumerate}
\end{lemma}
  \begin{proof}
    (1) and (2) are in Lemma~2.5.1 of \cite{Zhou2}. (3) follows from the proof of
    the that lemma, where it is shown that $\iota^* \mathcal O_{\tilde{\mathfrak
        M}_{g, n, d}}(\mathfrak D_{k-1})$ is canonically isomorphic to the pullback of
    the relative $\mathcal O(-1)$ of the projective bundle
    \[
      \mathfrak E_{k-1} \rightarrow  \mathfrak Z_{(k-1)}.
    \]
    Note that in Lemma~2.5.1 of \cite{Zhou2}, we have
    \[
      \tilde{\mathrm{gl}}^*_k \mathfrak E_{k-1} \cong \mathbb P \big((\Theta_1
      \oplus
      \cdots \oplus \Theta_{k}) \otimes
      \mathcal O_{\mathfrak U_{k}}(- \sum_{i = k}^{m-1} \mathfrak E_{i})
      \big).
    \]
    This is isomorphic to $\mathbb P(\Theta_1 \oplus \cdots \oplus \Theta_{k})$
    but the relative tautological differs by $\mathcal O_{\mathfrak
      U_{k}}(- \sum_{i = k}^{m-1} \mathfrak E_{i})$. Finally observe that the divisor
    pullback of $\mathfrak E_{i}$ to $\tilde{\mathfrak M}_{g, n, d}$ is equal to
    $\mathfrak D_{i}$ by Lemma~2.7.2 of \cite{Zhou2}.
  \end{proof}

As a corollary, we have the following refinement of \cite[Lemma 2.7.4]{Zhou2}.
\begin{lemma} \label{pullback-sum-of-all-divisors}
  Along the map
  \[
    \iota_{\mathfrak D}:{\tilde{\mathrm{gl}}_k^*} \mathfrak D_{k-1}
    \rightarrow \tilde{\mathfrak
      M}_{g, n, d},
  \]
  the line bundle
  \[
    \mathcal O_{\tilde{\mathfrak M}_{g, n, d}} \left( \mathfrak D_0 + \mathfrak
      D_1 + \cdots + \mathfrak D_{m-1} \right)
  \]
  pulls back to
  \[
    \mathcal O_{{\tilde{\mathrm{gl}}_k^*} \mathfrak D_{k-1}} \left(
      D_0 + \cdots +  D_{k-2}
    \right) \otimes \mathcal O_{ \tilde{\mathbb P}}(-1).
  \]
  where the divisors $D_0, \ldots, D_{k-2}$ are the tautological divisors of
  the
  inflated projective bundle \eqref{eq:D-as-inf-proj-bdl}.
\end{lemma}

\subsection{The calibration bundle and $\mathbb
  M$aster space} As before, we assume that $2g-2 + n + \epsilon_0 d \geq 0$, and
$\epsilon_0 d>2$ when $g = 0$.
\begin{definition}
  \label{def:calibration-bundle}
  When $(g, n, d) \neq (0, 1, d_0)$, the universal calibration bundle is
  defined to
  be the line
  bundle $\mathcal O_{\mathfrak
    M_{g, n, d}^{\mathrm{wt, ss}}}(- \mathfrak Z_1)$; when $(g, n, d) = (0, 1,
  d_0)$, the
  universal calibration bundle is the relative cotangent bundle at the unique
  marking.

  For an $S$-family of $\epsilon_0$-semistable, genus-$g$,
  degree-$d$ weighted curves, its calibration bundle is the pullback of universal
  calibration bundle along the classifying morphism $S \to  \mathfrak
  M_{g, n, d}^{\mathrm{wt, ss}}$.
\end{definition}
Let us focus on the case $(g, n, d) \neq (0, 1, d_0)$. For a curve $C$ with
degree-$d_0$ rational tails
$E_1 , \ldots, E_k$,
its calibration bundle is naturally isomorphic to $(\Theta_1 \otimes  \cdots
\otimes \Theta_k)^{\vee}$, where $\Theta_i$ is the one dimensional vector space
of
infinitesimal smoothings of the node on $E_i$. We refer the reader to
\cite[\textsection{2.8}]{Zhou2} for more details.

\begin{definition}
  \label{def:moduli-with-calibrated-tails}
  The moduli of $\epsilon_0$-semistable curves with calibrated tails is defined
  to be
  \[
    M \tilde {\mathfrak M}_{g, n, d}  : = \mathbb P_{\tilde {\mathfrak
        M}_{g, n, d}}(\mathbb M_{\tilde {\mathfrak
        M}_{g, n, d}} \oplus \mathcal O_{\tilde {\mathfrak M}_{g, n, d}}),
  \]
  where $\mathbb M_{\tilde {\mathfrak M}_{g, n, d}}$ is the calibration bundle
  of
  the universal family over $\tilde{\mathfrak M}_{g, n, d}$.
\end{definition}
Following \cite{Zhou2}, we call $M \tilde {\mathfrak M}_{g, n, d}$ the $\mathbb
M$aster space.
Let $S$ be a scheme. An $S$-point of the $\mathbb
M$aster space consists of
\[
  (\pi: \mathcal C \to S, \mathbf x , e, N, v_1, v_2)
\]
where
\begin{itemize}
\item
  $(\pi: \mathcal C \to S, \mathbf x , e) \in \tilde {\mathfrak M}_{g, n,
    d}(S)$;
  % an $S$-point of $\tilde {\mathfrak M}_{g, n, d}$;
\item
  $N$ is a line bundle on $S$;
\item
  $v_1 \in \Gamma(S, \mathbb  M_{S} \otimes  N)$, $v_2 \in \Gamma(S,
  N)$ have no common zero, where $\mathbb  M_{S}$ is the calibration bundle for
  the family of curves $\pi: \mathcal{C} \rightarrow S$.
\end{itemize}
For two families
% We denote an $S$ point of
% By definition, an arrow between two families
\[
  (\pi: \mathcal C \to S, \mathbf x , e, N, v_1, v_2) \quad \text{and} \quad
  (\pi^\prime: \mathcal C^\prime \to S^\prime, \mathbf x^\prime
  , e^\prime, N^\prime, v^\prime_1, v^\prime_2),
\]
an arrow between them consists of a triple
\[
  (f, t, \varphi),
\] where
\begin{itemize}
\item
  $f:S \to S^\prime$ is a morphism;
\item
  $t: (\pi: \mathcal C \to S, \mathbf x , e) \to f^*(\pi^\prime: \mathcal
  C^\prime \to S^\prime, \mathbf x^\prime , e^\prime)$ is a $2$-morphism in
  $\tilde {\mathfrak M}_{g, n, d}(S)$;
\item
  $\varphi: N \to f^*N^\prime$ is an isomorphism of line bundles, such that the
  morphisms $1 \otimes \varphi: \mathbb  M_{S} \otimes N \rightarrow \mathbb
  M_{S} \otimes f^*N' = f^*(\mathbb  M_{S'} \otimes N')$ and $\varphi$ sends
  $(v_1, v_2)$ to $(f^*v_1', f^*v_2')$.
\end{itemize}

\subsection{The moduli and its virtual structure sheaf}
\label{virtual-structure-sheaf}

Fix $g, n$ and a curve class $\beta$. Let $d = \deg(\beta)$. Recall that
$L_{\theta}$ is the polarization on $[W/G]$. Let $\mathfrak {Qmap}_{g, n}(X,
\beta)$ be the stack of genus-$g$, $n$-marked quasimaps to $X$ with curve class
$\beta$. Consider the open substack $\mathfrak {Qmap}^{\mathrm{ss}}_{g, n}(X, \beta)
\subset \mathfrak {Qmap}^{\mathrm{ss}}_{g, n}(X, \beta)$ parametrizing quasimaps with no
rational tails of degree $<d_0$, no
rational bridges of degree $0$, or base point of length $>d_0$. There is a
natural forgetful morphism
\[
  \mathfrak {Qmap}^{\mathrm{ss}}_{g, n}(X, \beta) \rightarrow \mathfrak
  M^{\mathrm{wt, ss}}_{g, n, d}
\]
defined by taking the underlying curves weighted by the degrees of the
quasimaps.

\begin{definition}
  \label{def:semistable-quasimaps-entangled-tails}
  We define the stack of genus-$g$, $n$-pointed, $\epsilon_0$-semistable
  quasimaps with \emph{entangled tails} to $X$ with curve class $\beta$ to be
  \[
    \mathfrak {Qmap}^{\sim}_{g, n}(X, \beta): = \mathfrak {Qmap}^{\mathrm{ss}}_{g, n}(X,
    \beta) \times_{\mathfrak M^{\mathrm{wt, ss}}_{g, n, d}} \tilde{\mathfrak M}_{g,
      n, d}. \]
\end{definition}

\begin{definition}
  An $S$-family of $\epsilon_0$-semistable quasimaps with entangled tails is
  $\epsilon_ + $-stable if the underlying family of quasimaps is
  $\epsilon_ + $-stable. In other words, it is stable if there is no
  length-$d_0$ base point.
\end{definition}

Let $\tilde Q^{\epsilon_ + }_{g, n}(X, \beta)$ denote the moduli of
genus-$g$, $n$-pointed
$\epsilon_ + $-stable quasimaps to $X$ with entangled tails of curve class
$\beta$.
According to \cite{Zhou2}, there is a natural isomorphism
\[
  \tilde Q^{\epsilon_ + }_{g, n}(X, \beta) \cong Q^{\epsilon_ + }_{g, n}(X,
  \beta)
  \underset{\mathfrak M_{g, n, d}^{\mathrm{wt, ss}}}{\times} \tilde{\mathfrak
    M}_{g, n, d}
\]
and hence $\tilde{Q}^{\epsilon_ + }_{g, n}(X, \beta)$ is a proper
Deligne--Mumford stack.

%\marginpar{\footnotesize{
 %   \Yang{
  %    Previously we used $u$ to denote the section $C \to P \times_{G} W$ and
   %   $[u]$
    %  to denote the map to the stack.
    %}
   % \Ming{Corrected.}
  %}}
Let $\pi: \mathcal C \to
Q^{\epsilon_ + }_{g, n}(X, \beta)$ be the universal curve and let $[u]: \mathcal C \to
  [W/G]$ be the universal map.
According to \cite{kim1, Kim4}, the moduli stack $Q^{\epsilon_ + }_{g, n}(X,
\beta)$ has a relative perfect obstruction theory
\begin{equation*}
  (R \pi_*u^* \mathbb T_{[W/G]})^{\vee}
  \rightarrow
  \mathbb L_{Q^{\epsilon_ + }_{g, n}(X, \beta)/ \mathfrak M_{g, n,
      d}^{\mathrm{wt, ss}}}
\end{equation*}
for the forgetful mophism $\nu:Q^{\epsilon_ + }_{g, n}(X, \beta) \rightarrow
\mathfrak M_{g, n, d}^{\mathrm{wt, ss}}$. According to \cite[Definition
2.2]{Qu},
using the above relative perfect obstruction theory, we can define a virtual
pullback
\[
  \nu^!:K_\circ(\mathfrak M_{g, n, d}^{\mathrm{wt, ss}})
  \rightarrow
  K_\circ(Q^{\epsilon_ + }_{g, n}(X, \beta)),
\]
and a virtual structure sheaf
\[\ovir_{Q^{\epsilon_ + }_{g, n}(X, \beta)}: = \nu^! \ca{O}_{\mathfrak M_{g, n,
      d}^{\mathrm{wt, ss}}} \in K_\circ(Q^{\epsilon_ + }_{g, n}(X, \beta)).
\]
Note that we use $\mathfrak M_{g, n, d}^{\mathrm{wt, ss}}$ in place of
$\mathfrak M_{g, n}$. Let $\nu':Q^{\epsilon_ + }_{g, n}(X, \beta) \rightarrow
\mathfrak M_{g, n}$ be the composition of $\nu$ and the \'etale morphism $\mu:
\mathfrak M_{g, n, d}^{\mathrm{wt, ss}} \rightarrow \mathfrak M_{g, n}$. We
will get the same virtual structure sheaf if we take the virtual pullback of
$\ca{O}_{\mathfrak M_{g, n}}$ along $\nu'$. This follows from the functoriality
of virtual pullbacks \cite[Proposition 2.11]{Qu} and the fact that $\mu^*
\ca{O}_{\mathfrak M_{g, n}} = \ca{O}_{\mathfrak M_{g, n, d}^{\mathrm{wt,
      ss}}}$.

Similarly, let $\tilde \pi: \tilde{\mathcal C} \to \tilde
Q^{\epsilon_ + }_{g, n}(X, \beta)$ be the universal curve and let $[\tilde u]:
\tilde{\mathcal C} \to [W/G]$ be the universal quasimap. Let $\tilde{\nu}:
\tilde
Q^{\epsilon_ + }_{g, n}(X, \beta) \rightarrow \tilde{\mathfrak M}_{g, n, d}$
denote the forgetful morphism. We have a relative perfect obstruction theory
\[
  (R \tilde \pi_* \tilde u^* \mathbb T_{[W/G]})^{\vee}
  \rightarrow
  \mathbb L_{\tilde Q^{\epsilon_ + }_{g, n}(X, \beta)/ \tilde{\mathfrak M}_{g,
      n, d}},
\]
which in turn defines a virtual pullback
$\tilde{\nu}^!:K_\circ(\tilde{\mathfrak M}_{g, n, d}) \rightarrow K_\circ(
\tilde
Q^{\epsilon_ + }_{g, n}(X, \beta))$. The virtual structure of $\tilde
Q^{\epsilon_ + }_{g, n}(X, \beta)$ is defined by
\[\ovir_{\tilde Q^{\epsilon_ + }_{g, n}(X, \beta)}: = \tilde \nu^!
  \ca{O}_{\tilde{\mathfrak M}_{g, n, d}} \in K_\circ(\tilde Q^{\epsilon_ + }_{g,
    n}(X, \beta)).
\]

Suppose we choose an integer $l$ and $K$-theory classes $R, E^{(m)}\in K^\circ_G(W),\,m\neq
0$.
%\Ming{In the previous version, I write $E^{(k)}$ to denote the twisting
 % class. Unfortunately, the index $k$ is also used to represent the number of
  %markings. So I change $k$ to $m$. If you see $E^{(k)}$, reminder me to change
  %it.}
As explained in Section \ref{twisted-theory}, we can define twisted virtual structure sheaves $ \ca{O}^{\vir,\mb{E}^{(\bullet)}, R, l}_{Q^{\epsilon_ + }_{g,n}(X, \beta)} $
and $ \ca{O}^{\vir,\mb{E}^{(\bullet)}, R, l}_{\tilde Q^{\epsilon_+}_{g, n}(X, \beta)} $ by~(\ref{eq:twisted-virtual-structure-sheaf}).

Let $\tau: \tilde Q^{\epsilon_ + }_{g, n}(X, \beta) \to
Q^{\epsilon_ + }_{g, n}(X, \beta)$ be the forgetful morphism. The following
lemma compares the two virtual structures and their twisted counterparts.
\begin{lemma}
  \label{lem:vir-str-sheaf-comp-entangled-tails}
  We have
  \[
    \tau_* \ovir_{\tilde Q^{\epsilon_ + }_{g, n}(X, \beta)} = \ovir_{Q^{\epsilon_ +
      }_{g, n}(X, \beta)}
    \]
    and
      \[
    \tau_* \ca{O}^{\vir,\mb{E}^{(\bullet)}, R, l}_{\tilde Q^{\epsilon_ + }_{g, n}(X, \beta)} = \ca{O}^{\vir,\mb{E}^{(\bullet)}, R, l}_{Q^{\epsilon_ +
      }_{g, n}(X, \beta)}.
  \]
\end{lemma}
\begin{proof}
  Consider the following fibered diagram.
  \[
    \begin{tikzcd}
      \tilde Q^{\epsilon_ + }_{g, n}(X, \beta) \arrow[d, " \tilde{\nu}"']  \arrow[r,
      " \tau"] & Q^{\epsilon_ + }_{g, n}(X, \beta) \arrow[d, " \nu"] \\
      \tilde{\mathfrak M}_{g, n, d} \arrow[r, " \tau'"]  & {\mathfrak M}_{g, n,
        d}^{\mathrm{wt}, \mathrm{ss}}
    \end{tikzcd}
  \]
  Since $\tau'$ is a sequence of blowups along smooth centers, we have
  \begin{equation} \label{eq:pf-str-sheaf}
    \tau'_* \ca{O}_{\tilde{\mathfrak M}_{g, n, d}} = \ca{O}_{{\mathfrak M}_{g, n,
        d}^{\mathrm{wt}, \mathrm{ss}}}.
  \end{equation}
  Note that the $\tilde{\nu}$-relative perfect obstruction theory is the
  pullback of $\nu$-relative perfect obstruction theory via $\tau$. The first part of the lemma follows from (\ref{eq:pf-str-sheaf}) and the fact that virtual
  pullbacks commute with proper pushforwards (c.f. \cite[Proposition 2.4]{Qu}).
  Since the universal principal $G$-bundle over $ \tilde Q^{\epsilon_ + }_{g,
    n}(X, \beta)$ is the pullback of that over $Q^{\epsilon_ + }_{g, n}(X,
  \beta)$, the same holds for twisting classes
  $E^{(m)}_{g,n,\beta}(t_\pm)$ and $\ca{D}^{R,l}$. Hence the second identity in
  the lemma follows from the first one and the projection formula.
 
\end{proof}

\subsection{Splitting off entangled tails} \label{split-tail}
Recall that for $k = 1, \dots, m$, the boundary divisors $\mathfrak D_{k-1}
\subset{\mathfrak M}^{\mathrm{wt, ss}}_{g, n, d}$ is the closure of the locally
closed reduced locus where there are exactly $k$ entangled tails. We describe
the pullback of the virtual structure sheaf to $\tilde{\mathrm{gl}}_k^*
\mathfrak D_{k-1}$ via the gluing morphism $\iota_{\mathfrak D}:
\tilde{\mathrm{gl}}_k^* \mathfrak D_{k-1}  \rightarrow \tilde{\mathfrak{M}}_{g,
  n, d}$.

Define $\tilde Q^{\epsilon_ + }_{g, n}(X, \beta)|_{{\tilde{\mathrm{gl}}_k^*}
  \mathfrak D_{k-1}}$ by the
following fibered diagram

\begin{equation*}
  \begin{tikzcd}
    \tilde Q^{\epsilon_ + }_{g, n}(X, \beta)|_{{\tilde{\mathrm{gl}}_k^*}
      \mathfrak D_{k-1}} \arrow[d] \arrow[r]
    & \tilde Q^{\epsilon_ + }_{g, n}(X, \beta) \arrow[d] \\
    {\tilde{\mathrm{gl}}_k^*} \mathfrak D_{k-1} \arrow[r, "{\iota_{\mathfrak
        D}}"] &  \tilde{\mathfrak M}_{g, n, d}
  \end{tikzcd},
\end{equation*}
and define
\[
  \ovir_{\tilde Q^{\epsilon_ + }_{g, n}(X, \beta)|_{{\tilde{\mathrm{gl}}_k^*}
      \mathfrak D_{k-1}}}
  : = \iota_{\mathfrak D}^!(\ovir_{\tilde Q^{\epsilon_ + }_{g, n}(X, \beta)}).
\]
Here $\iota_{\mathfrak D}^!$ denotes the Gysin pullback. The twisted
virtual structure sheaf $\ca{O}^{\vir,\mb{E}^{(\bullet)}, R, l}_{\tilde Q^{\epsilon_ +}_{g, n}(X, \beta)|_{{\tilde{\mathrm{gl}}_k^*}
      \mathfrak D_{k-1}}}$ is defined similarly as the Gysin pullback of the
  twisted virtual structure sheaf on $\tilde Q^{\epsilon_ + }_{g, n}(X, \beta)$.

According to \cite[\textsection{3.2}]{Zhou2}, there is a fibered diagram
\begin{equation*}
  \label{eq:splitting-nodes}
  \begin{tikzcd}
    \tilde Q^{\epsilon_ + }_{g, n}(X, \beta)|_{{\tilde{\mathrm{gl}}_k^*}
      \mathfrak D_{k-1}} \arrow[d]
    \arrow[r, "{p}"] & \coprod_{\vec \beta} \tilde
    Q^{\epsilon_ + }_{g, n + k}(X, \beta') \underset{(IX)^k}{\times} \prod_{i = 1}^k
    Q^{\epsilon_ + }_{0, 1}(X, \beta_i)
    \arrow[d] \\
    {\tilde{\mathrm{gl}}_k^*} \mathfrak D_{k-1} \arrow[r] & \tilde{\mathfrak
      M}_{g, n + k, d-kd_0} \times'
    (\mathfrak M_{0, 1, d_0}^{\mathrm{wt}, \mathrm{ss}})^k
  \end{tikzcd},
\end{equation*}
where $\vec \beta = (\beta', \beta_1 , \ldots, \beta_k)$ runs through all the decompositions of effective curve classes
\[
  \beta = \beta' + \beta_1+\cdots + \beta_k
\]
such that $\deg(\beta_i) = d_0$ for $i \geq 1$. By Lemma~ \ref{lem:str-D}, the
map $p$ above
is the inflated projective bundle $\tilde{\mathbb P}(\Theta_1 \oplus \cdots
\oplus
\Theta_k)$, and in particular, it is flat. Consider the fibered diagram
\[
  \begin{tikzcd}
    Q^{\epsilon_ + }_{g, n + k}(X, \beta')
\underset{(IX)^k}{\times} 
\prod_{i = 1}^k
  Q^{\epsilon_ + }_{0, 1}(X, \beta_i)
  \arrow[r,"\Phi"] \arrow[d,"\mathrm{ev}^k"]  &
    Q^{\epsilon_ + }_{g, n + k}(X, \beta') \times \prod_{i = 1}^k
    Q^{\epsilon_ + }_{0, 1}(X, \beta_i)
    \arrow[d]\\
    (IX)^k\arrow[r,"\Delta_{(IX)^k}"]  &
    (IX)^k\times (IX)^k,
  \end{tikzcd}
\]
where the bottom line is the diagonal morphism. 

\begin{lemma}
  \label{lem:int-with-D-k-1}
  We have
  \begin{equation*}
    %\label{eq:int-with-D-k-1}
    \ovir_{\tilde
      Q^{\epsilon_ + }_{g, n}(X, \beta)|_{{\tilde{\mathrm{gl}}_k^*} \mathfrak
        D_{k-1}}} =
    p^* \Big( \sum_{\vec \beta} \Delta_{ (IX)^k}^!
    \ovir_{\tilde{Q}^{\epsilon_ + }_{g, n + k}(X, \beta')} \boxtimes
    \prod_{i = 1}^k \ovir_{
      {Q}^{\epsilon_ + }_{0, 1}(X, \beta_i)} \Big)
  \end{equation*}
  and 
\[
  \ca{O}^{\vir,\mb{E}^{(\bullet)}, R, l}_{\tilde
      Q^{\epsilon_ + }_{g, n}(X, \beta)|_{{\tilde{\mathrm{gl}}_k^*} \mathfrak
        D_{k-1}}} =
    p^* \Big( \sum_{\vec \beta} \Delta_{ (IX)^k}^!
    \ca{O}^{\vir,\mb{E}^{(\bullet)}, R, l}_{\tilde{Q}^{\epsilon_ + }_{g, n + k}(X, \beta')} \boxtimes
    \prod_{i = 1}^k \ca{O}^{\vir,\mb{E}^{(\bullet)}, R, l}_{
     {Q}^{\epsilon_ + }_{0, 1}(X, \beta_i)} \Big)
                   \cdot p^*(\mathrm{ev}^k)^*(T^{-1})^{\boxtimes k},
\]
                   where $T=\operatorname{exp}\,
                    (\sum_{m \neq0}
                     \Psi^m(\widetilde{E^{(m)}}))/m )
                     \cdot(\det \, \widetilde{R} )^{-l} $ is the twisting class
                     introduced in Section \ref{twisted-theory}. 
  
                   \end{lemma}
                   %\Ming{double check if some of the evaluations are
                    % $\mathrm{ev}$ or $\check{\mathrm{ev}}$.}
\begin{proof}
  The first identity follows from the argument used in \cite[Lemma 3.2.1]{Zhou2} and the
  functoriality of virtual structure sheaves \cite[Proposition 4]{Lee}. In the
  twisted case, for various moduli spaces $\ca{M}$ of quasimaps, we denote by
  $\ca{T}_{\ca{M}}$ the twisting class~\eqref{eq:combined-twisting-class} over $\ca{M}$. The family version of~\eqref{eq:equation-normalization-sequence}
    implies that
    \[
      \Phi^*\big(\ca{T}_{ Q^{\epsilon_ + }_{g, n + k}(X, \beta')}
      \boxtimes
      \prod_{i=1}^k\ca{T}_{
        Q^{\epsilon_ + }_{0, 1}(X, \beta_i)}\big)
      =
\ca{T}_{ Q^{\epsilon_ + }_{g, n + k}(X, \beta') \times_{(IX)^k} \prod_{i = 1}^k
  Q^{\epsilon_ + }_{0, 1}(X, \beta_i)}
\cdot(\mathrm{ev}^k)^*T^{\boxtimes k}.
\]
Since the universal families and universal principal $G$-bundles are preserved
under the pullback of $p$, we have
\[
p^*(\ca{T}_{ Q^{\epsilon_ + }_{g, n + k}(X, \beta') \times_{(IX)^k} \prod_{i = 1}^k
  Q^{\epsilon_ + }_{0, 1}(X, \beta_i)})=\ca{T}_{ \tilde Q^{\epsilon_ + }_{g, n}(X, \beta)|_{{\tilde{\mathrm{gl}}_k^*}
      \mathfrak D_{k-1}}}.
\]
This concludes the proof of the second identity.
\end{proof}

%%% Local Variables:
%%% mode: latex
%%% TeX-master: "main"
%%% End:

%\input{localization}
\section{$K$-theoretic localization on the {$M$aster} space}
\label{section-master-space}
We recall in this section the definition of the master space and the
description of its $\bb{C}^*$-fixed point loci studied in \cite{Zhou2}. We will
compute the $\bb{C}^*$-equivariant $K$-theoretic Euler classes of the virtual
normal bundles of all fixed-point components.

\subsection{The $M$aster space and its $\bb{C}^*$-fixed loci}
%\marginpar{\footnotesize{
 %   \Yang{
  %    $M$ or $\mathbb M$?
   % }
  %\Ming{I am confused about the convention here. I thought
   % $\bb{M}$ is reserved for the
  %construction over the moduli space of weighted curves, see Section 3.3. $M$ is
%reserved for this section}}}
As before, we fix the numerical data $g, n, d$. Let $\epsilon_0 = 1/d_0$ be a
wall. We assume that $2g-2 + n + \epsilon_0 d \geq 0$, and $\epsilon_0 d>2$
when $g = n = 0$.

Recall from Definition~ \ref{def:moduli-with-calibrated-tails} the moduli stack
$M \tilde {\mathfrak M}_{g, n, d}$ of curves with calibrated tails. Let $S$ be
a scheme. An $S$-family of genus-$g$, $n$-pointed $\epsilon_0$-semistable
quasimaps with calibrated tails to $X$ of curve class $\beta$ is given by a
tuple
\[
  (\pi: \mathcal C \to S, \mathbf x , e, u, N, v_1, v_2)
\]
where
\begin{itemize}
\item $(\pi: \mathcal C \to S, \mathbf x , e, u)$ is an
  $\epsilon_0$-semistable, genus-$g$, $n$-pointed quasimaps to $X$ with entangled
  tails of curve class $\beta$, and
\item $(\pi: \mathcal C \to S, \mathbf x , e, N, v_1, v_2) \in M \tilde
  {\mathfrak
    M}_{g, n, d}(S)$.
\end{itemize}
Let $M \mathfrak {Qmap}^{\sim}_{g, n}(X, \beta)$ denote the category
%\marginpar{\footnotesize{
 %   \Yang{
  %    I suggest we either say ``category fibered in groupoids over the category
   %   of schemes'' or just say ``category''. Since ``groupoids'' refers to the
    %  fibers of the forgetful function from $M \mathfrak
     % {Qmap}^{\sim}_{g, n}(X, \beta)$ to the category of schemes.
    %}
  %\Ming{changed it}}}
parameterizing such families. According to
\cite[\textsection{4.1}]{Zhou2}, $M \mathfrak {Qmap}^{\sim}_{g, n}(X, \beta)$
is
an Artin stack of finite type with finite-type separated diagonal.

A degree-$d_0$ rational tail $E \subset C$ is called a \textit{constant tail}
if $E$
contains a base point of length $d_0$. We recall from
\cite[\textsection{4.1}]{Zhou2} the stability condition on the master space.
\begin{definition}
  \label{def:mspace-stability}
  An $S$-family of $\epsilon_0$-semistable quasimaps with calibrated tails
  \[
    (\pi: \mathcal C \to S, \mathbf x , e, u, N, v_1, v_2)
  \]
  is $\epsilon_0$-stable if over every geometric point $s$ of $S$,
  \begin{enumerate}
  \item any constant tail in $\mathcal C_s = \pi^{-1}(s)$ is an
    entangled tail;
  \item if $\mathcal C_s$ has at least one rational tail of degree $d_0$,
    then
    length-$d_0$ base points \textit{only} lie on  degree-$d_0$ rational tails
    of $\mathcal C_s$;
  \item if $v_1(s) = 0$, then  $(\pi: \mathcal C \to S, \mathbf x, u)|_{s}$ is
    an
    $\epsilon_ + $-stable quasimap;
  \item if $v_2(s) = 0$, then $(\pi: \mathcal C \to S, \mathbf x, u)|_{s}$ is
    an $\epsilon_-$-stable quasimap.
  \end{enumerate}
\end{definition}
Let $MQ^{\epsilon_0}_{g, n}(X, \beta)$ denote the category fibered in groupoids
parameterizing genus-$g$, $n$-pointed, $\epsilon_0$-stable quasimaps with
calibrated tails to $X$
of curve class $\beta$.
%\marginpar{\footnotesize{
 %   \Yang{
  %    Properness implies other assertions.
   % }}}
  According to Proposition 5.0.1 in
  \cite{Zhou2}, $MQ^{\epsilon_0}_{g, n}(X, \beta)$ is a  Deligne--Mumford
  % stack of
  % finite type
  proper over $\bb{C}$.
  % with finite type separated diagonal and
The space $MQ^{\epsilon_0}_{g, n}(X, \beta)$ is referred to as the $M$aster
space.

The construction of the virtual structure sheaf of $MQ^{\epsilon_0}_{g, n}(X,
\beta)$ is analogous to that of $Q^{\epsilon}_{g, n}(X, \beta)$. Let $\pi:
\mathcal C \to
MQ^{\epsilon_0}_{g, n}(X, \beta)$ be the universal curve and let $u: \mathcal C
\to
[W/G]$ be the universal map. Let $\nu_M:MQ^{\epsilon_0}_{g, n}(X, \beta)
\rightarrow M \tilde {\mathfrak
  M}_{g, n, d}$ be the forgetful morphism. There is a natural relative
perfect obstruction theory
\begin{equation*}
  (R \pi_*u^* \mathbb T_{[W/G]})^{\vee}
  \rightarrow
  \mathbb L_{\nu_M}.
\end{equation*}
We can define the virtual structure sheaf via the virtual pullback:
\[\ovir_{MQ^{\epsilon_0}_{g, n}(X, \beta)}: = \nu_M^! \ca{O}_{M \tilde
    {\mathfrak
      M}_{g, n, d}} \in K_\circ(MQ^{\epsilon_0}_{g, n}(X, \beta)).
\]
The construction of the twisted virtual structure with level structure on $MQ^{\epsilon_0}_{g, n}(X,
\beta)$ is also parallel to that on $Q^{\epsilon}_{g,n}(X,\beta)$.

Let $d=\mathrm{deg}(\beta)$. We denote by $\fr l$ the least common multiple of
$1,2,\dots,\lfloor d/d_0 \rfloor$. Consider the $\mathbb C^*$-action on $MQ^{\epsilon_0}_{g, n}(X, \beta)$
defined by scaling $v_1$:
%\marginpar{\footnotesize{
 %   \Yang{
  %    We should raise this action to its $k$-th power so that the exact
   %   sequence
    %  in Alper--Hall--Rydh splits. Also this eliminate the need of introducing
     % $q^{1/k}$.
    %}}}
\begin{equation}
  \label{eq:C*action-master-space}
  \lambda \cdot(\pi: \mathcal C \to S, \mb{x}, e, u, N, v_1, v_2) =  (\pi:
  \mathcal C \to
  S, \mb{x}, e, u, N, \lambda^{\fr l} v_1, v_2) , \quad \lambda \in \mathbb C^*.
\end{equation}
It is the $\fr l$-th power of the $\C^*$-action define in \cite[(6.1)]{Zhou2}.
The purpose of this modification is to trivialize the $\C^*$-action on the fixed-point
components and avoid fractional weights that usually show up in localization computations.

According to \cite[\textsection{6}]{Zhou2}, there are three types of
fixed-point components.

\subsubsection{$\epsilon_ + $-stable quasimaps with entangled tails}
\label{epsilon + }
Let $F_ + \subset MQ^{\epsilon_0}_{g, n}(X, \beta)$ be the Cartier divisor
defined by $v_1 = 0$. It is a fixed-point component. We have an isomorphism \[
  F_ + \cong \tilde Q^{\epsilon_ + }_{g, n}(X, \beta).
\]
which identifies the universal principal $G$-bundles over the universal families
and the perfect obstruction theories. Hence it also identifies their virtual
structure sheaves
\[
  \ovir_{F_ + } = \ovir_{\tilde Q^{\epsilon_ + }_{g, n}(X, \beta)}.
\]
and twisted virtual structure sheaves with level structure.
The virtual normal bundle is $\mathbb M_ + $, the calibration bundle of
$\tilde{Q}^{\epsilon_ + }_{g, n}(X, \beta)$ in Definition
\ref{def:calibration-bundle}, with a $\mathbb C^*$-action of weight $\fr l$.

\subsubsection{$\epsilon_-$-stable quasimaps} \label{epsilon-}
When $g = 0, n = 1, \text{deg}(\beta) = d_0$, the moduli stack
$Q^{\epsilon_-}_{g, n}(X, \beta)$ is empty and $v_2$ is non-vanishing
on $MQ^{\epsilon_0}_{g, n}(X, \beta)$. Otherwise, the
Cartier divisor $F_- \subset MQ^{\epsilon_0}_{g, n}(X, \beta)$
defined by $v_2 = 0$ is a fixed component. We have an isomorphism
\[
  F_- \cong  Q^{\epsilon_-}_{g, n}(X, \beta).
\]
Again, the above isomorphism identifies the universal principal $G$-bundles and the perfect obstruction theory, and
hence virtual structure sheaves
\[
  \ovir_{F_-} = \ovir_{Q^{\epsilon_-}_{g, n}(X, \beta)}.
\]
and their twisted counterparts.
The virtual normal bundle of $F_-$ in the {$M$aster} space is the line bundle
$\mathbb M_-^\vee$, the dual of the
calibration bundle $\mathbb M_-$ of $Q_{g, n}^{\epsilon_-}(X, \beta)$. The
$\bb{C}^*$-action on $\mathbb M_-^\vee$ has weight $(-\fr l)$.

The last type of fixed loci will be explained in the next subsection.
\subsection{The correction terms}
The other fixed-point components are closely related to the graph space $QG_{0,
  1}(X, \beta)$
where $\deg(\beta) = d_0$, and the $K$-theoretic localization contributions can
be expressed in terms of the $I$-function. Recall that
$F_{\star, \beta}: = F_{\star, 0}^{0, \beta} \subset QG_{0, 1}(X, \beta)$
denotes
the fixed-point component where the unique marking $x_\star$ is at $\infty$ and the
quasimap
$u$ has a base point of length $\text{deg}(\beta) = d_0$ at $0$. According to
\cite[\textsection{6.3}]{Zhou2}, the virtual normal bundle
$N^{\vir}_{F^{0, \beta}_{\star, 0}/QG_{0, 1}^{0 + }(X, \beta)}$ is isomorphic
to
  \[
    \big(R \pi_*(u^* \mathbb T_{[W/G]})|_{F_{\star, \beta}} \big)^{\mathrm{mv}}
    \oplus
    T_\infty \bb{P}^1
 \]
%\ChYang{
 % \begin{equation}
  %  \label{eq:decomposition-normal-bundle-graph-space}
   % \big(R \pi_*(u^* \mathbb T_{[W/G]})|_{F_{\star, \beta}} \big)^{\mathrm{mv}}
    %\ominus
   % T_0 \bb{P}^1,
  %\end{equation}}
  in the $K$-theory,
where $\pi: \ca{C} \rightarrow QG_{0, 1}(X, \beta)$ is the universal curve,
$u: \ca{C} \rightarrow[W/G]$ is the universal map, and the upper index ``mv''
denotes the moving part of the complex.
%\marginpar{\footnotesize{
 %   \Yang{
  %    Why saying this here?
   % }
  %\Ming{This detail is needed in order to compare $\ca{I}_\beta$ with the
   % $I$-action introduced in Def 2.3. But I haven't decided where to add this sentence.}}}

Define the following class in the localized $K$-group
\[
  \ca{I}_\beta(q) := \frac{1}{\lambda_{-1}^{\bb{C}^*} \big(\big(\big(R \pi_*(u^*
    \mathbb
    T_{[W/G]})|_{F_{\star, \beta}} \big)^{\mathrm{mv}} \big)^\vee \big)} \in
  K^\circ_{\bb{C}^*}(F_{\star, \beta}) \otimes_{\bb{Q}[q, q^{-1}]} \bb{Q}(q).
\]
Note that the tangent space $T_\infty \bb{P}^1$ has
  $\bb{C}^*$-weight 1 and hence
  $\lambda_{-1}^{\bb{C}^*}(T_\infty \bb{P}^1) = 1-q^{-1}$. It follows that the $I$-function in Definition \ref{K-theoretic-I-function} can be
rewritten as
\begin{equation}\label{eq:equation-two-I-functions}
  I(Q, q) = \sum_{\beta\geq  0} Q^{\beta} (\check{\ev}_\star)_*
  \Big(
  \mathcal I_{\beta}(q) \cdot
  \ovir_{F_{\star, \beta}}
  \Big).
\end{equation}
In the twisted case, we define
\[
  \ca{I}^{{\bf E}^{(\bullet)},R,l}_\beta(q):=\ca{I}_\beta(q)\cdot(\check{\ev}_\star)^*
  (\overline{T}^{-1}
   ).
 \]
 Formula~\eqref{eq:equation-two-I-functions} still holds if we replace $I(Q,q)$,
$\ca{I}_\beta(q)$ and $\ca{O}^{\vir}_{F_{\star,\beta}}$ by their twisted counterparts.

To describe the objects that the fixed-point components parametrize, we need
the following definition. Consider an $\epsilon_0$-stable quasimap with
calibrated tails
$$\xi = (C, \mathbf x , e, u, N, v_1, v_2) \in MQ^{\epsilon_0}_{g, n}(X,
\beta)(\bb{C}).$$
\begin{definition}
  Let $E \subset C$ be a degree-$d_0$ rational tail and let $y \in E$ be a node
  (or marking if $g = 0, n = 1$). The tail $E$ is called a \emph{fixed tail} if
  the automorphism group $\text{Aut}(E, y, u|_E)$ is infinite.
\end{definition}

\subsubsection{$g = 0, n = 1, \deg(\beta) = d_0$ case.}
In this case, the curve must be irreducible and $v_2$ is non-vanishing. Let
$F_\beta$ be the fixed-point component defined by
\[
  F_{\beta} : = \{\xi \mid \text{the domain curve of } \xi \text{ is a single
    fixed tail, $v_1 \neq 0, v_2 \neq 0$}
  \}.
\]
%Let $C$ be the domain curve and let $x_{\star}$ be the unique marking.
%By definition, the calibration bundle is
%the relative orbifold cotangent space at $x_{\star}$. We fix a
%nonzero
%tangent vector $v_{\infty}$ at $\infty$ to $\mathbb P^1$. Then there is unique
%morphism $C
%\to \mathbb P^1$ sending $x_{\star}$ to $\infty$, the unique base
%point to $0$, and sending $(v_2/v_1)^{\otimes \mathbf r_\star}$ to
%$v_{\infty}$.

%According to \cite[Lemma 6.4.1]{Zhou2}, there is an \'etale cyclic covering
%\[
 % \tau:F_{\beta} \longrightarrow  F_{\star, \beta}.
%\]
%of degree $\mathbf r_\star$. \Ming{Anything more needed to prove here?}
    We viewing the cotangent space $T^*_{\infty} \mathbb P^1$ at $\infty$ as a
    constant
    line bundle on $F_{\star, \beta}$. Recall that  $\tilde{L}_\star$ is the
    line bundle on $F_{\star, \beta}$ formed by the relative orbifold cotangent
    space at $x_\star$. By the definition of $F_{\star, \beta}$ we have an
    isomorphism $\tilde{L}_\star^{\otimes \mathbf r_{\star}} \cong
    T^*_{\infty} \mathbb P^1$. According to \cite[Lemma 6.4.1]{Zhou2} and
    its proof, there is a morphism
    \[
      \tau: F_{\beta} \rightarrow  F_{\star, \beta}.
    \]
    realizing $F_{\beta}$ as the cyclic covering of $\mathbf r_\star$-th roots
    in $\tilde{L}_\star$
    of the constant section of $T^*_{\infty} \mathbb P^1$ dual to $v_\infty$.
    In particular, $\tau$ is representable, finite \'etale of degree $\mathbf
    r_\star$.

\begin{lemma}
  We have
%  {\color{red}
 %   \[
  %    \tau_* \ca{O}_{F_{\star, \beta}} = \theta^{\mb{r}_\star}(\tilde{L}_*),
   % \]
    %\[
     % \tau^* \ovir_{F_{\star, \beta}} = \ovir_{F_\beta},
   % \]
   % }
    \[
      \tau_* \ca{O}_{F_{\beta}} = \theta^{\mb{r}_\star}(\tilde{L}_\star),
      \quad
      \tau^* \ovir_{F_{\star, \beta}} = \ovir_{F_\beta},
    \]
    and
  \[
    \frac{1}{\lambda^{\mathbb C^*}_{-1}
      \big(\big(N^{\mathrm{vir}}_{F_{\beta}/MQ^{\epsilon_0}_{0, 1}(X, \beta)}
      \big)^\vee \big)} =
    (1-q^{{\mathbf r_\star}\fr{l}}) \cdot \tau^* \mathcal
    I_{\beta}(q^{{\mathbf r_\star}\fr{l}}).
  \]
  In the twisted case, we have
  \[
      \tau^* \ca{O}^{\vir,{\bf E}^{(\bullet)},R,l}_{F_{\star, \beta}} =  \ca{O}^{\vir,{\bf E}^{(\bullet)},R,l}_{F_\beta}.
  \]
\end{lemma}
\begin{proof}
%  {\color{red}
  %  {The first identity follows from the fact that $\tau:F_{\beta} \to
%      F_{\star, \beta}$ is an \'etale cyclic cover obtained by taking the
 %     $\mb{r}_*$-th
 %     root of $\mu_\infty$ (c.f., for example,
   %   \cite[\textsection{3.5}]{Esnault-Viehweg}).
   % }}
    Recall that
    $\theta^{\mb{r}_\star}(\tilde{L}_\star) : = \sum_{j = 0}^{\mathbf{r}_\star-1}
    \tilde{L}^{j}_{\star}$.
    Thus the first identity is a standard fact for cyclic covers (c.f., for
    example,
    \cite[\textsection{3.5}]{Esnault-Viehweg}). Let $\mathbb E_{MQ}$ denote the
  absolute perfect obstruction theory of $MQ^{\epsilon_0}_{0, 1}(X, \beta)$. The
  second and the third equalities follow from the analysis of the fixed and moving parts of
  the restriction $\mathbb E_{MQ}|_{F_\beta}$ in \cite[Lemma 6.4.2]{Zhou2} and
  the
  functoriality of virtual structure sheaves (c.f. \cite[Proposition 4]{Lee}).
%  {\color{red}
 %   The equivariant weight $q^{\mb{r}_\star}$ comes from the fact that the
 %   weight-1
 %   $\bb{C}^*$-action on $v_1$ corresponds to a weight-$\mb{r}_\star$
 %   $\bb{C}^*$-action on the cotangent space $T^*_0 \bb{P}^1$.
 % }
    The $q^{\mb{r}_\star\fr{l}}$ comes from the fact that the
    $\mathbb C^*$-action \eqref{eq:C*action-master-space} on the master space
    % weight-1
    % $\bb{C}^*$-action on $v_1$
    corresponds to a weight-$(\mb{r}_\star\fr{l})$
    $\bb{C}^*$-action on the cotangent space $T^*_0 \bb{P}^1$.
    % The factor $(1-q^{\mathbf r_\star})$ comes from the summand $T_{\infty}
    % \mathbb
    % P^1$ in \eqref{eq:decomposition-normal-bundle-graph-space}. In fact,
    % the weight-1 $\bb{C}^*$-action on $v_1$ induces a $\mathbb C^*$-action on
    % corresponds to a weight-$\mb{r}_\star$
    % $\bb{C}^*$-action on the cotangent space $T^*_0 \bb{P}^1$.
    % The factor $(1-q^{\mathbf r_\star})$ comes from the facts that
    % $\lambda_{-1}^{\bb{C}^*}(T_\infty \bb{P}^1) = 1-q^{-1}$ with respect to the
    % $\mathbb C^*$-action on the target $\mathbb P^1$ in the definition of the
    % graph space, and that the weight-1
    % $\bb{C}^*$-action on $v_1$ corresponds to a weight-$(- \mb{r}_\star)$
    % $\bb{C}^*$-action on the cotangent space $T^*_0 \bb{P}^1$.
  The last equality between twisted virtual structure sheaves with level
  structure follows from the fact that the pullback of the universal curve over
  $F_\beta$ along with its universal principal $G$-bundle via $\tau$ are isomorphic to
  those over $F_{\star,\beta}$. 
  
\end{proof}

\begin{corollary} \label{normal-bundle-special-case}
  Write
  \[
    I(Q, q) = \sum_{\beta \geq0}  I_\beta(q)Q^{\beta}\quad and \quad I^{{\bf E}^{(\bullet)},R,l}(Q, q) = \sum_{\beta \geq0}  I^{{\bf E}^{(\bullet)},R,l}_\beta(q)Q^{\beta}
  \]Then we have
 % {\color{red}
  %  \[
   %   \check{\mathrm{ev}}_* \tau_*
    %  \frac{\ovir_{F_{\star, \beta}}}{\lambda^{\mathbb C^*}_{-1}
     %   \big(\big(N^{\mathrm{vir}}_{F_{\beta}/MQ^{\epsilon_0}_{0, 1}(X, \beta)}
      %  \big)^\vee \big)}
     % = (1-q^{\mb{r}})I_\beta(q^{\mb{r}}).
    %\]
  %}
    \begin{align*}
      (\check{\ev}_\star)_* \tau_*
      \frac{\ovir_{F_{\beta}}}{\lambda^{\mathbb C^*}_{-1}
        \big(\big(N^{\mathrm{vir}}_{F_{\beta}/MQ^{\epsilon_0}_{0, 1}(X, \beta)}
        \big)^\vee \big)}
      &= (1-q^{\mb{r}\fr{l}})I_\beta(q^{\mb{r}\fr{l}}),\quad\text{and}\\
            \overline{T}^{-1}\cdot(\check{\ev}_\star)_* \tau_*
      \frac{\ca{O}^{\vir,{\bf E}^{(\bullet)},R,l}_{F_{\beta}}}{\lambda^{\mathbb C^*}_{-1}
        \big(\big(N^{\mathrm{vir}}_{F_{\beta}/MQ^{\epsilon_0}_{0, 1}(X, \beta)}
        \big)^\vee \big)}
      &= (1-q^{\mb{r}\fr{l}})I^{{\bf E}^{(\bullet)},R,l}_\beta(q^{\mb{r}\fr{l}}),
    \end{align*}
    where $\mathbf r$ is the locally constant function on $\bar{I}X$ that takes value
    $r$ on $\bar{I}_rX$.
  %\Ming{I notice that we have defined $\bf r$ multiple times, but I guess that
   % is okay}
\end{corollary}

\subsubsection{$2g - 2 + n + \epsilon_0 d > 0$ case}
\label{subsection-general-case}
\begin{condition}
  \label{cond:condition-on-beta}
  Let $\ub = (\beta^\prime,\{ \beta_1 , \ldots, \beta_k \})$ be a tuple
  satisfying the following conditions
  \begin{enumerate}
    % \item
    %   $\beta^\prime, \beta_1 , \ldots, \beta_k$ are effective curve classes;
   % \marginpar{\footnotesize{ \Yang{ I find these conditions a little bit
    %      annoying. Maybe we can just keep (2) and (3) and claim the
     %     corresponding term is zero if the moduli is empty. }}}
  \item $\{\beta_1,\dots,\beta_k\}$ is a (unordered) multiset;
  \item $\beta = \beta^\prime + \beta_1 + \cdots + \beta_k$;
  \item $\deg(\beta_i) = d_0$ for $i = 1 , \ldots, k$;
    % \item
    %   $2g-2 + n + k + \epsilon \deg(\beta')>0$ for any $\epsilon> \epsilon_0$.
  \end{enumerate}
\end{condition}

For each $\ub = (\beta^\prime,\{ \beta_1 , \ldots, \beta_k \})$ satisfying
Condition~ \ref{cond:condition-on-beta}, we define a (possibly empty) substack
\begin{align*}
  F_{\ub} = \{\xi \mid & \xi \text{ has exactly $k$ entangled tails, }
  \\
                       & \text{which are all fixed tails,   of degrees } \beta_1, \ldots , \beta_k
                         \}
\end{align*}
of $MQ^{\epsilon_0}_{g, n}(X, \beta)$. According to
\cite[\textsection{6.5}]{Zhou2}, $ F_{\ub}$ is closed (if nonempty), and $F_ + ,
F_-$ and $F_{\ub}$ are all the fixed-point components of the $\bb{C}^*$-action
on the $\bb{M}$aster space $MQ^{\epsilon_0}_{g, n}(X, \beta)$.

We now recall the structure of $F_{\ub}$ described in
\cite[\textsection{6.5}]{Zhou2}. Recall that $\mathfrak Z_{(k)} \subset
\mathfrak U_{k}$ is the proper transform of the locus $\mathfrak Z_k \subset
\mathfrak M^{\mathrm{wt}, \mathrm{ss}}_{g, n, d}$ where there are at least $k$
rational tails of degree $d_0$. Recall from Section
\ref{sec:weighted-twisted-curves} that
\[
  \tilde{\mathrm{gl}}_k : \tilde{\mathfrak M}_{g, n + k, d-kd_0} \times' {\left(
      \mathfrak M_{0, 1, d_0}^{\mathrm{wt}, \mathrm{ss}} \right)}^k
  \rightarrow \mathfrak Z_{(k)}
\]
is the morphism that is induced by gluing the universal curve of ${\left(
    \mathfrak M_{0, 1, d_0}^{\mathrm{wt}, \mathrm{ss}} \right)}^k$ to the last
$k$ markings of the universal curve $\tilde{\mathfrak M}_{g, n + k, d-kd_0}$ as
degree-$d_0$ rational tails. According to \cite[Lemma 6.5.3]{Zhou2}, there is a
forgetful morphism $F_{\ub} \rightarrow \mathfrak Z_{(k)}$. We form the fibered
diagram
\[
  \begin{tikzcd}
    {\tilde{\mathrm{gl}}_k^*}F_{\ub} \arrow[r] \arrow[d] & \tilde{\mathfrak
      M}_{g, n + k, d-kd_0} \times' {\left( \mathfrak
        M_{0, 1, d_0}^{\mathrm{wt}, \mathrm{ss}} \right)}^k \arrow[d] \\
    F_{\ub} \arrow[r] & \mathfrak Z_{(k)}
  \end{tikzcd}.
\]
Note that the $k$ entangled tails are ordered in $
{\tilde{\mathrm{gl}}_k^*}F_{\ub}$ and hence so are the curve classes $\beta_1,
\dots \beta_k$. There is a natural $S_k$-action on
${\tilde{\mathrm{gl}}_k^*}F_{\ub}$ which permutes the $k$ entangled tails. It
motivates the following condition:
\begin{condition}
  \label{cond:condition-on-beta-ordered}
  Let $\vec \beta = (\beta^\prime, \beta_1 , \ldots, \beta_k)$ be an ordered
  tuple satisfying the last two conditions in Condition
  \ref{cond:condition-on-beta}.
\end{condition}
We will refer to $\vb$ as an ordered decomposition of the class $\beta$. Set
$\vb_{(i)}=\beta_i,i=1,\dots,k$. We will use this notation when we want to
emphasize that the class $\beta_i$ is the $i$-th component in the decomposition $\vb$.

The stack ${\tilde{\mathrm{gl}}_k^*}F_{\ub}$ has the following decomposition
\[
  {\tilde{\mathrm{gl}}_k^*}F_{\ub} = \coprod_{\vb\mapsto
    \ub}{\tilde{\mathrm{gl}}_k^*}F_{\vec \beta},
\]
where the disjoint union is over all ordered tuples satisfying Condition
\ref{cond:condition-on-beta-ordered} and having $\ub$ as the underlying multiset, and the notation ${\tilde{\mathrm{gl}}_k^*}F_{\vec \beta}$ denotes the
substack parametrizing quasimaps whose curve classes of ordered entangled tails
are given by $\vb$. Note that for a permutation $\sigma \in S_k$, it maps the
component labeled by $\vb$ to that labeled by $ \sigma(\vb)$ where
\[
  \sigma(\vb): = (\beta', \beta_{\sigma(1)}, \dots, \beta_{\sigma(k)}).
\]

Let $\ca{C}_{{\tilde{\mathrm{gl}}_k^*}F_{\ub}}$ be
% {the universal curve \color{red}over ${\tilde{\mathrm{gl}}_k^*}F_{\vec
% \beta}$}
the pullback to ${\tilde{\mathrm{gl}}_k^*}F_{\ub}$ of the universal
  curve of $MQ^{\epsilon_0}_{g, n}(X, \beta)$. Let $\ca{C}_{\beta'}$ be the
pullback to ${\tilde{\mathrm{gl}}_k^*}F_{\ub}$ of the universal curve of $
\tilde{\mathfrak M}_{g, n + k, d-kd_0}$ and let $\ca{E}_1, \dots, \ca{E}_k$ be
the pullback to ${\tilde{\mathrm{gl}}_k^*}F_{\ub}$ of the universal curve of ${(
  \mathfrak M_{0, 1, d_0}^{\mathrm{wt}, \mathrm{ss}})}^k$. Note that
$\ca{C}_{{\tilde{\mathrm{gl}}_k^*}F_{\ub}}$ is obtained by gluing $\ca{E}_1,
\dots, \ca{E}_k$ to the last $k$ markings of $\ca{C}_{\beta'}$. We call
  $\ca{C}_{\beta'}$ the main component of
  $\ca{C}_{{\tilde{\mathrm{gl}}_k^*}F_{\ub}}$. Let $T_{p_i} \ca{C}_{\beta'}$
(resp. $T_{p_i} \ca{E}_i$) be the orbifold tangent line bundles of the main
component $\ca{C}_{\beta'}$ (resp. the rational tail $\ca{E}_i$) at the orbifold
node $p_i$. Denote $T_{p_1} \ca{C}_{\beta'} \otimes T_{p_1} \ca{E}_1$ by
$\Theta$. According to Lemma \cite[Lemma 2.5.5]{Zhou2} and
\cite[\textsection{6.5}]{Zhou2}, there are canonical isomorphisms
\[
  \Theta \cong T_{p_i} \ca{C}_{\beta'} \otimes T_{p_i} \ca{E}_i, \quad i = 2,
  \dots k,
\]
and
\begin{equation} \label{eq:theta-k-th-root} \Theta^{\otimes k} \cong
  \bb{M}^\vee_{\beta'}
\end{equation}
on ${\tilde{\mathrm{gl}}_k^*}F_{\ub}$. Here, $\bb{M}^\vee_{\beta'}$ is the
(pullback of the) calibration bundle on $ \tilde{\mathfrak M}_{g, n + k,
  d-kd_0}$. Consider the stack
\[
  Y \rightarrow {\tilde Q}^{\epsilon_ + }_{g, n + k}(X, \beta^\prime)
\]
of $k$-th roots of the pullback to ${\tilde Q}^{\epsilon_ + }_{g, n + k}(X,
\beta^\prime)$ of $\bb{M}^\vee_{\beta'}$. Then (\ref{eq:theta-k-th-root})
induces a morphism
\begin{equation} \label{eq:map-to-Y} {\tilde{\mathrm{gl}}_k^*}F_{\ub}
  \rightarrow Y.
\end{equation}

Now we focus on each component $ {\tilde{\mathrm{gl}}_k^*}F_{\vec \beta}$ of $
{\tilde{\mathrm{gl}}_k^*}F_{\ub}$. As explained in
\cite[\textsection{6.5}]{Zhou2}, the restrictions of quasimaps to $\ca{E}_i$
give rise to morphisms
\begin{equation} \label{eq:map-to-Fbeta} {\tilde{\mathrm{gl}}_k^*}F_{\vec \beta}
  \rightarrow F_{\star, \beta_i}, \quad i = 1, \dots, k.
\end{equation}
Let
\[
  \text{ev}_Y:Y \rightarrow (I X)^k
\]
denote the evaluation maps at the last $k$-markings and for
each $i$, let
\[
  \check{\text{ev}}_{\star, \beta_i} : F_{\star, \beta_i} \rightarrow IX
\]
denote the evaluation map at the unique marking $x_{\star}$ composed with the
involution on $IX$. Consider the fiber product of $\text{ev}_Y$ and
$\check{\text{ev}}_{\star, \beta_i}$ over $(IX)^k$:
\[
  Y \times_{(IX)^k} \prod_{i = 1}^k F_{\star, \beta_i}
\]

By \cite[Lemma 6.5.5]{Zhou2}, the morphism
\[
  \varphi: {\tilde{\mathrm{gl}}_k^*}F_{\vec \beta} \overset{}{\rightarrow} Y
  \times_{(IX)^k} \prod_{i = 1}^k F_{\star, \beta_i},
\]
induced by (\ref{eq:map-to-Y}) and (\ref{eq:map-to-Fbeta}) is a representable,
finite, and \'etale of degree $\prod_{i = 1}^k \mathbf r_i$. Here $\mb{r}_i$ is
the locally constant function whose value is the order of the automorphism group
of the $i$-th node $p_i$. In fact, according to the proof of \cite[Lemma
6.5.5]{Zhou2}, the morphism $\varphi$ is a fiber product of cyclic \'etale
covers obtained by taking the $\mb{r}_i$-th root of $(T_{p_i}
\ca{E}_i)^{\mb{r}_i}$.

Let $\ovir_Y \in K_\circ(Y)$ be the flat pullback of $\ovir_{\tilde Q^{\epsilon_
    + }_{g, n + k}(X, \beta^\prime)}$, and let
$\ovir_{\tilde{\mathrm{gl}}_k^*F_{\vec \beta}}$ be the flat pullback of
$\ovir_{F_{\vec \beta}}$, the virtual structure sheaf defined by the fixed part
of the absolute perfect obstruction theory. We denote by $\tilde{L}(\mathcal
E_i)$ the orbifold cotangent line bundle of the rational tail $\mathcal E_i$ at
the orbifold node and by $\tilde L_{n + i}$ the orbifold cotangent line bundle
at the $(n + i)$-th marking of ${\tilde Q}^{\epsilon_ + }_{g, n + k}(X,
\beta^\prime)$. Let $L(\mathcal E_i)$ and $L_{n + i}$ denote the cotangent line
bundles on the coarse curves. Recall from
Section~\ref{inflated-bundle-boundary-divisor} that $\mathfrak D_{i} \subset \tilde{\mathfrak{M}}_{g, n, d}$ is the closure of the
locally closed reduced locus where there are exactly $(i+1)$ entangled tails.

We want to express the correction terms in terms of Euler characteristic of
sheaves over $Y \times_{(IX)^k} \prod_{i = 1}^k F_{\star, \beta_i}$. One
complication is that the pullback of the universal curve on $Y \times_{(IX)^k}
\prod_{i = 1}^k F_{\star, \beta_i}$ is not isomorphic to the universal curve
over ${\tilde{\mathrm{gl}}_k^*}F_{\vec \beta}$. Nevertheless, we have the
following comparison results between virtual structure sheaves and virtual
normal bundles.

\begin{lemma}
  \label{lem:obstruction-theory-F-beta}
  We have
  \[
    \varphi_* \big( \ca{O}_{\tilde{\mathrm{gl}}_k^*F_{\vec \beta}} \big) =
    \prod_{i = 1}^k \theta^{\mb{r}_i} \big(\tilde{L} (\ca{E}_i) \big), \]
  \[
    \ovir_{\tilde{\mathrm{gl}}_k^*F_{\vec \beta}} = \varphi^* \big(\ovir_Y
    \boxtimes \textstyle \prod_{i = 1}^k \ovir_{F_{\star, \beta_i}} \big),
  \]
  and
  \begin{equation*}
    \begin{aligned}
      \frac{ 1 } { \lambda^{\mathbb C^*}_{-1}
        \big((N^{\mathrm{vir}}_{F_{\ub}/MQ^{\epsilon_0}_{g, n}(X,
          \beta)}|_{{\tilde{\mathrm{gl}}_k^*}F_{\vec \beta}})^\vee \big) } = &
      \varphi^*\bigg(\frac {\prod_{i = 1}^k(1-q^{\frac{\mathbf
            r_i\fr{l}}{k}}L(\mathcal E_i)^\vee)} {1-q^{ - \frac{\fr{l}}{k}}
        \tilde L(\mathcal E_1) \cdot \tilde L_{n + 1} \cdot \ca{O}(\sum_{i =
          k}^{m-1} \mathfrak
        D_i))} \cdot \\
      & \mathcal I_{\beta_1}(q^{\frac{\mathbf r_1\fr{l}}{k}}L(\mathcal
      E_1)^\vee) \boxtimes \cdots \boxtimes \mathcal
      I_{\beta_k}(q^{\frac{\mathbf r_k \fr{l}}{k}}L(\mathcal E_k)^\vee) \bigg),
    \end{aligned}
  \end{equation*}
  where $\theta^{\mb{r}_i} \big(\tilde{L} (\ca{E}_i) \big)$ is the
  $\mb{r}_i$-the Bott's cannibalistic class and $m:=\lfloor d/d_0 \rfloor$.
\end{lemma}
% \begin{remark}
%   \label{rmk:entanglement-implies-equation-of-psi}
%   Note that since $\mathcal E_1 , \ldots, \mathcal E_k$ are the entangled
%   tails, we have
%   {\color{red}
%   \[
%     \tilde L(\mathcal E_1) \otimes \tilde L_{n + 1} = \cdots =
%     \tilde L(\mathcal E_k) \otimes \tilde L_{n + k}
%   \]
% }
%\marginpar{\footnotesize{ \Yang{ This is already said before
 %     \eqref{eq:theta-k-th-root}. } \Ming{I commented out the remark.}}}
% on ${\tilde{\mathrm{gl}}_k^*} F_{\vec \beta}$.
% \end{remark}
\begin{proof} The first statement
  follows from the fact that $\varphi$ is a fiber product of cyclic \'etale
  covers obtained by taking the $\mb{r}_i$-th root of $(T_{p_i}
  \ca{E}_i)^{\mb{r}_i}$ (c.f., for example,
  \cite[\textsection{3.5}]{Esnault-Viehweg}).

  To prove the last two statements, we recall the comparison between the perfect
  obstruction theories of $\tilde{\mathrm{gl}}_k^*F_{\vec \beta}$ and $Y
  \times_{(IX)^k} \prod_{i = 1}^k F_{\star, \beta_i}$ given in the proof of
  \cite[Lemma 6.5.6]{Zhou2}.

  Let $\pi_1: \mathcal C_1 \to {\tilde{\mathrm{gl}}_k^*}F_{\vec \beta}$ be the
  universal curve and let $u_1: \mathcal C_1 \to [W/G]$ be the universal map.
  Recall that the complex $\mathbb E_1 = \big(R \pi_{1*}u_1^* \mathbb
  T_{[W/G]})^{\vee}$ defines a perfect obstruction theory relative to $M
  \tilde{\mathfrak M}_{g, n, d}$. Let $\mathbb E_{MQ}$ denote the absolute
  perfect obstruction theory on $MQ^{\epsilon_0}_{g, n}(X, \beta)$. We have a
  distinguished triangle
  \[
    \mathbb L_{M \tilde{\mathfrak M}_{g, n,
        d}}|_{{\tilde{\mathrm{gl}}_k^*}F_{\vec \beta}} \rightarrow \mathbb
    E_{MQ}|_{{\tilde{\mathrm{gl}}_k^*}F_{\vec \beta}} \rightarrow \mathbb
    E_1 \overset{ + 1}{\rightarrow}.
  \]

  Let $\pi_2: \mathcal C_2 \to Y{\times}_{(IX)^{k}} \prod_{i = 1}^k F_{\star,
    \beta_i}$ be the universal curve obtained from gluing the unique marking of
  $F_{\star, \beta_i}$ to the $(n + i)$-th marking of $Y$. Let $u_2: \mathcal
  C_2 \to [W/G]$ be the universal map. Then the $\bb{C}^*$-fixed part of the
  complex $\mathbb E_2 = (R \pi_{2*}u_2^* \mathbb T_{[W/G]})^{\vee}$ defines a
  perfect obstruction theory of $Y {\times}_{(IX)^{k}} \prod_{i = 1}^k F_{\star,
    \beta_i}$ relative to $\tilde{\mathfrak M}_{g, n + k, d-d_0k}$. By using a
  splitting-node argument similar to that of Lemma \ref{lem:int-with-D-k-1}, one
  can prove that $\ovir_Y \boxtimes \textstyle \prod_{i = 1}^k \ovir_{F_{\star,
      \beta_i}}$ is equal to the virtual pullback of $\ca{O}_{\tilde{\mathfrak
      M}_{g, n + k, d-d_0k}}$ defined by $\bb{E}_2^{\text{f}}$.

  In \cite{Zhou2}, the second author constructed a principal
  $(\bb{C}^*)^k$-bundle $Y' \rightarrow Y$ such that the base change of the two
  universal families to $Y'$ are isomorphic. More precisely, we have a fibered
  diagram
  \[
    \begin{tikzcd}
      Y' \times_Y {\tilde{\mathrm{gl}}_k^*}F_{\vec \beta} \arrow[r, " \varphi'"]
      \arrow[d, swap, "p_1"] & Y' \times_{(I X)^k} \prod_{i = 1}^k F_{\star,
        \beta_i}
      \arrow[d, "p_2"] \\
      {\tilde{\mathrm{gl}}_k^*}F_{\vec \beta} \arrow[r, " \varphi"] & Y
      \times_{(I X)^k} \prod_{i = 1}^k F_{\star, \beta_i}
    \end{tikzcd}
  \]
  which is equipped with an isomorphism of the universal curves
  \[
    \tilde{\varphi}:p_1^* \ca{C}_1 \rightarrow p_2^* \ca{C}_2
  \]
  that commutes with the universal maps to $[W/G]$. Hence it induces an
  isomorphism
  \begin{equation} \label{eq:iso-obstruction-theory} p_1^* \bb{E}_1 \cong p_1^*
    \varphi^* \bb{E}_2.
  \end{equation}

  According to the proof of \cite[Lemma 6.5.6]{Zhou2}, there is a natural
  isomorphism between the fixed parts
  \[
    \bb{E}_1^\text{f} \cong \varphi^* \bb{E}_2^\text{f}.
  \]
  Since the flat pullback of $\ca{O}_{\tilde{\mathfrak M}_{g, n + k, d-d_0k}}$
  equals $\ca{O}_{M \tilde{\mathfrak M}_{g, n + k, d-d_0k}}$, the above
  isomorphism implies the second statement of the lemma.

  By the analysis in the proof of \cite[Lemma 6.5.6]{Zhou2}, the pullback of the
  $K$-theoretic Euler class of the virtual normal bundle of $F_{\ub}$ in
  $MQ^{\epsilon_0}_{g, n}(X, \beta)$ comes from three contributions:
  \begin{enumerate}
  \item the moving part $\bb{E}_1^{\vee, \text{mv}}$;
  \item $\oplus_{i = 1}^kL(\ca{E}_i)[-1]$ where the $i$-th factor has
    $\bb{C}^*$-weight $- \mb{r}_i\fr{l}/k$;
  \item $\Theta_1 \otimes \mathcal O (- \sum_{i = k}^{\infty} \mathfrak D_{i})$
    with $\bb{C}^*$-weight $\fr{l}/k$.
  \end{enumerate}
  The contributions (2) and (3) give rise to the equality
  \begin{equation} \label{eq:first-part-contribution}
    \frac{1}{\lambda_{-1}^{\bb{C}^*} \big((\mathbb L_{M \tilde{\mathfrak M}_{g,
          n, d}}|_{{\tilde{\mathrm{gl}}_k^*}F_{\vec \beta}})^{\text{mv}} \big)}
    = \frac {\prod_{i = 1}^k(1-q^{\frac{\mathbf r_i\fr{l}}{k}}L(\mathcal
      E_i)^\vee)} {1-q^{ - \frac{\fr{l}}{k}} \tilde L(\mathcal E_1) \cdot \tilde
      L_{n + 1} \cdot \ca{O}(\sum_{i = k}^{m-1} \mathfrak D_i)}.
  \end{equation}

  To compute the contribution (1), we consider group actions on $p_i^* \ca{C}_i,
  i = 1, 2$. There is a $\bb C^* \times(\bb{C}^*)^k$ action on $p_1^* \ca{C}_1$
  where the first factor acts by (\ref{eq:C*action-master-space}) and the second
  factor acts on $Y'$. There is a $(\bb C^*)^k \times(\bb{C}^*)^k$ action on
  $p_2^* \ca{C}_2$ where the first factor acts by the product of the
  $\bb{C}^*$-actions (\ref{eq:C*action}) on the graph spaces $QG_{0, 1}(X,
  \beta_i)$ and the second factor acts on $Y'$. Let
  \[
    (q_1, \dots, q_k, \lambda_1, \dots, \lambda_k)
  \]
  be the equivariant parameters of $(\bb C^*)^k \times(\bb{C}^*)^k$ and let $q$
  be the equivariant parameter of $\bb{C}^*$ as before. By definition, we have
  \[
    \frac{1}{\lambda_{-1}^{(\bb{C}^*)^k \times(\bb{C}^*)^k} \left(p_2^*
        \bb{E}_2^{\text{mv}} \right)} = \ca{I}_{\beta_1}(q_1) \boxtimes \cdots
    \boxtimes \ca{I}_{\beta_k}(q_k).
  \]
  Then it follows from \cite[Lemma 6.5.4]{Zhou2} and
  (\ref{eq:iso-obstruction-theory}) that
  \[
    \frac{1}{\lambda_{-1}^{\bb{C}^* \times(\bb{C}^*)^k} \left(p_1^*
        \bb{E}_1^{\text{mv}} \right)} = \ca{I}_{\beta_1}(q^{\mb{r}_1\fr{l}/k}
    \lambda_1^{-1}) \boxtimes \cdots \boxtimes
    \ca{I}_{\beta_k}(q^{\mb{r}_k\fr{l}/k} \lambda_k^{-1}).
  \]
  By the construction of $Y'$ (c.f. \cite[Section 6.5]{Zhou2}), under
    the natural equivalence between
    \[
      K_{\mathbb C^* \times (\mathbb C^*)^k}^{\circ}(Y' \times_Y
      {\tilde{\mathrm{gl}}_k^*}F_{\vec \beta}) \quad \text{and} \quad K_{\mathbb
        C^*}^{\circ}(\tilde{\mathrm{gl}}_k^*F_{\vec \beta}),
    \]
    % the $\mathbb C^* \time (\mathbb
    % C^*)^k$-equivariant $K$-theory of
  the $(\bb{C}^*)^k$-weights $\lambda_i$ over the principal bundle $Y'
  \times_Y {\tilde{\mathrm{gl}}_k^*}F_{\vec \beta} \rightarrow
  {\tilde{\mathrm{gl}}_k^*}F_{\vec \beta}$ correspond to the pullbacks of the
  coarse cotangent line bundles $L(\ca{E}_i)$. Therefore the above equality
  descends to
  \[
    \frac{1}{\lambda_{-1}^{\bb{C}^*} \left(\bb{E}_1^{\text{mv}} \right)} =
    \ca{I}_{\beta_1}(q^{\mb{r}_1\fr{l}/k}L(\ca{E}_1)^\vee) \boxtimes \cdots
    \boxtimes \ca{I}_{\beta_k}(q^{\mb{r}_k\fr{l}/k}L(\ca{E}_k)^\vee).
  \]
  Combining with (\ref{eq:first-part-contribution}), we conclude the proof of
  the third statement.
\end{proof}
\begin{remark}\label{rem:compare-twisting}
For the various moduli spaces $\ca{M}$ in the previous lemma, we denote by
  $\ca{T}_{\ca{M}}$ the twisting class~\eqref{eq:combined-twisting-class} over
  $\ca{M}$. Let $\mathrm{ev}^k$ denote the evaluation map from $ Y' \times_{(I X)^k} \prod_{i = 1}^k F_{\star,
        \beta_i}$ (or $ Y \times_{(I X)^k} \prod_{i = 1}^k F_{\star,
        \beta_i}$) to $(IX)^k$. It follows from the proof of Lemma
      \ref{lem:int-with-D-k-1} that
      \[
\ca{T}_{Y' \times_{(I X)^k} \prod_{i = 1}^k F_{\star,
    \beta_i}}=\big(\ca{T}_{Y'}\boxtimes
\prod_{i=1}^k\ca{T}_{F_{\star,
    \beta_i}}(q_i)
\big)
\cdot
 (\mathrm{ev}^k)^*(T^{-1})^{\boxtimes k}
      \]
By using similar arguments as in the proof of Lemma
\ref{lem:obstruction-theory-F-beta}, we can show that
 \[
 \ca{T}_{{\tilde{\mathrm{gl}}_k^*}F_{\vec \beta}}
 =\varphi^*\big(\ca{T}_Y\boxtimes\prod_{i=1}^k\ca{T}_{F_{\star,\beta_i}}(q^{\mb{r}_i\fr{l}/k}L(\ca{E}_i)^\vee)\cdot
 (\mathrm{ev}^k)^*(T^{-1})^{\boxtimes k}\big).
\]
\end{remark}

\iffalse \bibliography{ref} \fi

%%% Local Variables:
%%% mode: latex
%%% TeX-master: "main"
%%% End:

%\input{wallcrossing-formula}
\section{$K$-theoretic wall-crossing formula} \label{section-wall-crossing}

\iffalse
\Ming{The pairing on the state space is defined by
  \[
    \chi(\alpha, \beta) = \chi(I \ca{X}, \alpha \otimes \iota^* \beta),
  \]
  for $\alpha, \beta \in K^\circ(I \ca{X})$. We ASSUME this pairing is
  non-degenerate. I am unable to find the reference on when it holds, so it is
  necessary to make this assumption.}

\Yang{
  If we are in the algebraic setting, this is almost never finite dimensional.
  We could either stay in the algebraic setting and write everything in terms
  of
  pushforward, pullback and Gysin pullback along diagonal; or we could pass to
  topological K-theory and using the pairing there. The pairing on algebraic
  K-theory you are defining here passes through the topological K-theory, of
  course.
}

\Ming{Here I refer to topological K-theory. I do need to check the "Kunneth"
  formula in K-theory and see when decomposing the diagonal, whether $K^1(X)$
  shows up. The Chern character gives an isomorphism between $K^\circ$ and the
  EVEN
  cohomology group. }
\Ming{double checked. $K^1(X)$ does show up.}
\fi

In this section, we first recall the virtual $K$-theoretic localization formula
in \cite{Kiem-Savvas} and review some elementary properties of the residue
operation. Then we prove the $K$-theoretic wall-crossing formulas by applying
the localization formula to the master space.
\subsection{Virtual $K$-theoretic localization formula}
Let $\ca{X}$ be an admissible Deligne-Mumford stack of finite type with a $\bb{C}^*$-action and a
$\bb{C}^*$-equivariant almost perfect obstruction theory
%%\marginpar{\footnotesize{
  %  \Yang{
   %   According to Michael Savvas, a perfect obstruction is not automatically
    %  an
     % almost perfect obstruction theory, since in the definition of almost
      %perfect obstruction theory we need an equivariant \'etale atlas by a
      %scheme. Some remarks are needed here.
   % }
 % \Ming{In their newest version, they clarified the relation: a POT is
  %  automatically an almost-POT. They allow atlas to be DM stacks}}}
in the sense of Definition 2.10 and Definition 5.11 in
\cite{Kiem-Savvas}. Denote the $\bb{C}^*$-equivariant weight
by
$q$. We have the following virtual $K$-theoretic localization formula.
\begin{theorem}[\cite{Kiem-Savvas}] \label{K-theoretic-localization}
  Let $F$ denote the $\bb{C}^*$-fixed locus of $\ca{X}$ and let $\iota:F \rightarrow
  \ca{X}$ be the embedding. Under the assumption that the virtual normal bundle
  $N^\vir$ of $F$ in $\ca{X}$ has a global resolution $[N_0 \rightarrow N_1]$ of
  locally free sheaves on $F$, we have
  \begin{equation}
    \label{eq:K-theoretic-localization}
    [\ovir_{\ca{X}}] = \iota_*
    \frac{[\ovir_F]}{\lambda_{-1}^{\C^*}((N^\vir)^\vee)} \in
    K_\circ^{\bb{C}^*}(\ca{X}) \otimes_{\bb{Q}[q, q^{-1}]} \bb{Q}(q).
  \end{equation}
\end{theorem}

According to~\cite[Remark 5.2]{Kiem-Savvas}, a $\C^*$-equivariant perfect obstruction
theory is a special case of a $\C^*$-equivariant almost perfect obstruction theory. By Proposition 5.13 and Remark 5.14 in~\cite{Kiem-Savvas},
 A Deligne-Mumford stack with a $\bb{C}^{*}$-action is always admissible after possibly reparameterizing the action of $\bb{C}^{*}$.
It is clear from Lemma~\ref{lem:obstruction-theory-F-beta} and the discussions
in Section~\ref{epsilon + } and Section~\ref{epsilon-} that the virtual normal
bundle of each $\C^*$-fixed-point component in $MQ^{\epsilon_0}_{g, n}(X, \beta)$ has a global two-term locally-free
resolution. Hence the virtual $K$-theoretic localization
formula~\eqref{eq:K-theoretic-localization} is valid for the $\C^*$-action on the
{$M$aster} space.

\subsection{Residue operation and its properties}
%\Ming{I changed the base field from $\C$ to $\Q$. If you spot any $\C$, please
 % change it to $\Q$.}
Let $\ca{X}$ be a Deligne--Mumford stack. Consider the vector space
$K_\circ(\ca{X}) \otimes \Q(q)$ of rational function in $q$ with
coefficients in the $K$-group $K_\circ(\ca{X})_{\Q}:=K_\circ(\ca{X})\otimes\Q$. We define a
residue operation
\[
  \text{Res}(-): K_\circ(\ca{X}) \otimes \bb{Q}(q) \rightarrow
  K_\circ(\ca{X})_{\Q}
\]
by
\[
  \text{Res} \big(f(q) \big): = \big[\text{Res}_{q = 0} + \text{Res}_{q = \infty}
  \big] \big(f(q) \big) \, \frac{dq}{q}.
\]
Here $\text{Res}_{q = 0}(-)dq/q$ and $\text{Res}_{q = \infty}(-)dq/q$ are the
obvious extensions of the corresponding residue operations on $\bb{Q}(q)$. From
now on, we simply refer to $\text{Res} \big(f(q) \big)$ as the residue of
$f(q)$.

We first summarize two elementary properties of the residue operation in the
following lemma .
\begin{lemma}
  \phantomsection
  \label{elementary-properties-of-residues}
  \begin{enumerate}
  \item For any Laurent polynomial $g(q) \in K_\circ(\ca{X}) \otimes \bb{Q}[q,
    q^{-1}]$, we have
    \[
      \emph{Res}(g(q)) = 0.
    \]
  \item Let $r$ be an integer and $L$ be a line bundle on $\ca{X}$. For any $f(q)
    \in K_\circ(\ca{X}) \otimes \bb{Q}(q)$, the change of variable $q \mapsto q^rL$
    does not change the residue, i.e., we have
    \[
      \emph{Res}(f(q)) = \emph{Res}(f(q^rL)).
    \]
  \end{enumerate}
\end{lemma}
%\begin{proof}
 % The first statement follows from Cauchy's residue theorem and the fact that
 % Laurent polynomials do not have poles other than 0 and $\infty$. The second
 % statement follows from the fact that the residue {\color{red}of a rational
  %  function $f(q)$}
 % \marginpar{\footnotesize{
  %    \Yang{
   %     Should be residue of $f(q) \frac{dq}{q}$, right?
    %  }
    %\Ming{I commented out the proof since it is standard.}}}
% at 0 or $\infty$ equals the constant term in the Laurent expansion of $f(q)$ at
 % 0 or $\infty$. The change of variable $q \mapsto q^rL$ preserves the constant
  %terms.
%\end{proof}

  \begin{proof}
    By linearity it suffices to assume $g(q) \in \mathbb Q[q, q^{-1}]$. Thus
    the
    results are standard.
  \end{proof}

  % = = = = = = = = = = = = = = = = = = =
  % Remember to comment the ``addtocounter'' before releasing the paper.
  % = = = = = = = = = = = = = = = = = = =

 % \Ming{What does this addtocounter thing do? You can add this corollary directly.}
%  \addtocounter{dummy}{-1}
  \begin{corollary}
    \label{cor:master-space-technique}
    In the situation of Theorem~ \ref{K-theoretic-localization}, for any proper
    morphism $q:\ca{X}\rightarrow \ca{Y}$ between Deligne-Mumford stacks such that  $q$ is $\C^*$-equivariant
    with respect to the trivial $\C^*$-action on $\ca{Y}$, we have
    \[
      \emph{Res} \Big(q_*\big(F,
      \frac{[\ovir_F]\cdot\ca{E}}{\lambda_{-1}^{\C^*}((N^\vir)^\vee)}\big) \Big) = 0,
    \]
    for any $\ca{E}\in K_{\bb{C}^{*}}^\circ(\ca{X})_{\Q}$.
  \end{corollary}
  \begin{proof}
   Consider the trivial $\C^*$-action on $\ca{Y}$. Then $q$ is $\C^*$-equivariant. We apply $q_*$ to both
    sides of \eqref{eq:K-theoretic-localization}. Then the corollary follows from
    the observation that $q_*(\mathcal X, [\ovir_{\ca{X}}]\cdot\ca{E}) \in K^{\C^*}_\circ(\ca{Y})\otimes\mathbb Q[q,
    q^{-1}]$.
    % Then the residue of the left
    % hand side
    % \[
    %   \chi(\mathcal X, )
    % \]
  \end{proof}

An important example for us is the residue of the rational function of the form
\[
  f(q) = \frac{g(q)}{1-q^{-1}L},
\]
where $g(q) \in K_\circ(\ca{X}) \otimes \bb{Q}(q)$ and $L$ is a line bundle on
$\ca{X}$. 
%Note that in the above expression of $f(q)$, the denominator contains
%$K$-theory classes. Nevertheless, we can view it as an element in
%$K_\circ(\ca{X}) \otimes \bb{Q}(q)$ as follows.\footnote{The authors learned
  %this
  %trick from \cite{Givental-Tonita}.} Suppose $P(q)$ is the minimal polynomial of $L$
  %in $K^\circ(\ca{X}) \otimes \bb{Q}$.
%\marginpar{
 % \footnotesize{
  %  \Yang{Why does $P(q)$ exist?}
  %}}
%Note that the constant term in $P(q)$ is
%nonzero and hence $P(q) \neq0$. Then we have
%\[
 % P(q) = P(q)-P(L) = (q-L)T (q, L) \quad \text{in} \ K_\circ(\ca{X}) \otimes
  %\bb{Q}(q).
%\]
%Here $T(q, L)$ is a polynomial in $q$ and $L$ whose degree is smaller than that
%of $P(q)$. By modifying the above equality slightly, we get
%\begin{equation} \label{eq:relation-to-invert}
  %q^{-1}P(q) = (1-q^{-1}L)T (q, L).
%\end{equation}
%Therefore, we can define $1/(1-q^{-1}L)$ by
%\begin{equation} \label{eq:honest-k-class}
  %\frac{qT(q, L)}{P(q)}.
%\end{equation}
%
A convenient way to compute the residue
\[
  \text{Res} \bigg(\frac{g(q)}{1-q^{-1}L} \bigg)
\]
is by using the Laurent expansions of $g(q)/(1-q^{-1}L)$ at $0$ and $\infty$.
More precisely, we have
\[
  \frac{g(q)}{1-q^{-1}L} = - \frac{g(q)qL^{-1}}{1-qL^{-1}}
\]
in $K_\circ(\ca{X}) \otimes \bb{Q}(q)$. The right hand side of the above
equation has the following formal Laurent series expasion around 0:
\begin{equation} \label{eq:geometric-series-expansion}
  -g(q) \sum_{i = 1}^\infty q^iL^{-i} \in K_\circ(\ca{X})((q)).
\end{equation}
%By viewing (\ref{eq:relation-to-invert}) as a relation in
%$K_\circ(\ca{X})((q))$, we can easily check that the expansion
%(\ref{eq:geometric-series-expansion}) equals (\ref{eq:honest-k-class}) as
%formal Laurent series. 
It follows that
\[
  \text{Res}_{q = 0} \bigg(\frac{g(q)}{1-q^{-1}L} \bigg) \, \frac{dq}{q} = -
  \sum_{i = 1}^\infty[g(q)]_{-i}L^{-i},
\]
where $[g(q)]_{-i}$ denotes the coefficient of $q^{-i}$ in the formal
%\marginpar{\footnotesize{
 %   \Yang{
  %    As you have pointed out, it is not only a formal Laurent series. It is
   %   indeed a Laurent series. Namely there are only finitely many negative
    %  power terms.
    %}
  %\Ming{I thought in the definition of formal Laurent series, only finitely many
  %negative power terms are allowed. The word ``Formal'' only concerns about the positive
  %terms, as in formal power series.}}}
Laurent
series expansion of $g(q)$ at $q = 0$. Note that $[g(q)]_{-i} \neq 0$ only for finitely many $i>0$.
%\ChYang{
 % Note that $[g(q)]_{-i} \neq 0$ only for finitely many $i>0$. Hence the
  %infinite
  %sum is actually finite.
  % Note that since $g(q) \in K_{\circ}(\mathcal X) \otimes \mathbb C(q)$, its
  % expansions are Laurent series. Hence the  has
  % finitely many negative terms in the
  % infinite sum is actually finite.
%}
To compute the residue at $\infty$, we make
the change of variable $w = 1/q$. By using similar arguments, we can show that
\[
  \text{Res}_{q = \infty} \bigg( \frac{g(q)}{1-q^{-1}L} \bigg) \,
  \frac{dq}{q} = - \text{Res}_{w = 0} \bigg(\frac{g(1/w)}{1-wL} \bigg) \,
  \frac{dw}{w} =
  - \sum_{i = 0}^\infty[g(1/q)]_{-i}L^{i}
\]
% \marginpar{\footnotesize{
% \Yang{
% In what sense this finite sum lies in
% $K_\circ(\ca{X}) \otimes \bb{C}$?
% }}}

We summarize the above discussion in the following lemma.
\begin{lemma}Let $g(q) \in K_\circ(\ca{X}) \otimes \bb{Q}(q)$ and let $L$ be a
  line bundle on $\ca{X}$. Then
  \[\emph{Res} \bigg(\frac{g(q)}{1-q^{-1}L} \bigg) = - \sum_{i =
      1}^\infty[g(q)]_{-i}L^{-i}- \sum_{i = 0}^\infty[g(1/q)]_{-i}L^{i}. \]
\end{lemma}

\begin{corollary} \label{residue-constant-g}
  Let $g \in K_\circ(\ca{X})_{\Q}$ and let $L$ be a line bundle on $\ca{X}$. We have
  \[
    \emph{Res} \bigg(\frac{g}{1-q^{-1}L} \bigg) = -g
  \]
  and
  \[
    \emph{Res} \bigg(\frac{g}{1-qL} \bigg) = g.
  \]
\end{corollary}

For later computations, we need one more elementary identity of the residue
operation. Consider a rational function $h(q) \in K_\circ(\ca{X}) \otimes
\bb{Q}(q)$. We write $h(q)$ as a unique sum of a Laurent polynomial $[h(q)]_ +
$ and a proper rational function $[h(q)]_-$. The notation is consistent with
those introduced in Section \ref{K-theoretic-quasimap-invariants}.

\begin{lemma} \label{negative-projection}
  For any $h(q) \in K_\circ(\ca{X}) \otimes \bb{Q}(q)$, we have
  \[
    \sum_{i = 0}^\infty q^i \emph{Res}(q^{-i}h(q)) = [h(q)]_-.
  \]
\end{lemma}
\begin{proof}
  By the linearity of the residue operation, we have
  \[
    \sum_{i = 0}^\infty q^i \text{Res}(q^{-i}h(q)) = \sum_{i = 0}^\infty q^i
    \text{Res}(q^{-i}[h(q)]_ +) + \sum_{i = 0}^\infty q^i
    \text{Res}(q^{-i}[h(q)]_-).
  \]
  The first term of the right hand side is zero because
  $q^{-i}[h(q)]_ + $ is a Laurent polynomial. Therefore, we only need to prove
  the
  second term on the right hand side is equal to $[h(q)]_-$. We first analyze the
  residues at $\infty$. Note that $q^{-i}[h(q)]_-$ vanishes at $\infty$ for
  $i \geq0$. Therefore we have
  \[
    \text{Res}_{q = \infty}(q^{-i}[h(q)]_-) \frac{dq}{q} = 0.
  \]

  Now we compute the residues at 0. By definition, $[h(q)]_-$ can be expanded
  into a formal power series around 0
  \[
    [h(q)]_- = \sum_{j = 0}^\infty a_jq^{j}, \quad \text{where} \ a_j \in
    K_\circ(\ca{X})_{\Q}.
  \]
  The lemma follows from the following identity:
  \[
    \text{Res}_{q = 0} \big(q^{-i}(\sum_{j = 0}^\infty a_jq^{j}) \big) \frac{dq}{q}
    = a_i, \quad i \geq 0.
  \]
\end{proof}

\subsection{The $g = 0, n = 1$ and $d = d_0$ case}
The statements of the results and their proofs in the rest of the section apply
to both the untwisted and the twisted theories. To simplify the exposition, we
will drop $ \mb{E}^{(\bullet)}, R$ and $l$ from the notation of the twisted theory.

Recall that $\mb{r}$ denotes the locally constant function on $\bar{I}X$ that
takes value $r$ on the component $\bar{I}_rX$. Set $\mb{r}_1 =
\ev_1^*(\mb{r})$. Recall that $\tilde{L}_1$ is the orbifold cotangent
line bundle at the unique marking and $\check{\ev}_{1}:Q^{\epsilon_
  + }_{0, 1}(X, \beta)\rightarrow  \bar{I}X$ is the rigidified evaluation map at the unique marking $x_{1}$ composed with the
involution on
$\bar{I}X$.

\begin{lemma} \label{special-evaluation}
  \[
     (\check{\ev}_{1})_*\Big(\ovir_{Q^{\epsilon_ + }_{0,
        1}(X, \beta)}
    \cdot L_1^\ell \Big)
    = 
    \emph{Res}
    \big(
    q^{-\ell}(1-q) I_{\beta}(q) \big) ,
    \quad \ell \in \mathbb Z.
  \]
  The same identity holds if we replace $\ovir_{Q^{\epsilon_ + }_{0,
        1}(X, \beta)}$ and $I_{\beta}(q)$ by their twisted counterparts.
\end{lemma}
\begin{proof}
    Let $\check{\widetilde{\ev}}_{1}:MQ^{\epsilon_0}_{0, 1}(X,
    \beta)\rightarrow \bar{I}X$ be the rigidified evaluation map at the unique marking $x_{1}$ composed with the
involution on
$\bar{I}X$. We apply Corollary~ \ref{cor:master-space-technique} to
    the master space $MQ^{\epsilon_0}_{0, 1}(X, \beta)$ with the
    $\bb{C}^*$-action (\ref{eq:C*action-master-space}) and $\check{\widetilde{\ev}}_{1}$. The fixed locus $F$ is
    the disjoint union of $F_ + $ and $F_\beta$, which are studied in
    Section~ \ref{epsilon + } and Corollary \ref{normal-bundle-special-case},
    respectively.
    Thus we obtain
    \begin{align*}
      & \mathrm{Res}\biggl( (\check{\ev}_{1})_*\Big(
        \frac{\ovir_{Q^{\epsilon_ + }_{0, 1}(X, \beta)} \cdot
        L_1^\ell}
        {1- q^{-1} \mathbb M_ + ^\vee}
        \Big) \biggr)
        +
        \mathrm{Res} \biggl( q^{-l\mathbf{r}}(1-q^{\mathbf r})
        I_{\beta}(q^{\mathbf r})  \biggr) = 0.
    \end{align*}
    We conclude the proof by applying Corollary~ \ref{residue-constant-g} to the first term and Lemma~\ref{elementary-properties-of-residues}, part (2) to the second term.
 % {\color{red}
  %  We apply the virtual $K$-theoretic localization formula to the master space
   % $MQ^{\epsilon_0}_{0, 1}(X, \beta)$ with respect to the $\bb{C}^*$-action
    %(\ref{eq:C*action-master-space}). By the results mentioned in Section
    %\ref{epsilon + }, Corollary \ref{normal-bundle-special-case} and the projection
    %formula, we obtain
    %\begin{align*}
     % \chi \big(MQ^{\epsilon_0}_{0, 1}(X, \beta),
      %\ovir_{MQ^{\epsilon_0}_{0, 1}(X, \beta)} \otimes \tilde{L}_1^{\ell} \big)
      %= &
       %   \chi \bigg(Q^{\epsilon_ + }_{0, 1}(X, \beta),
        %  \frac{\ovir_{Q^{\epsilon_ + }_{0, 1}(X, \beta)} \otimes
         % \tilde{L}_1^\ell}{1- q^{-1}M^\vee} \bigg) \\
     % + &
      %    \chi \big(I_\mu X, q^{-l}(1-q^{\mathbf r}) I_{\beta}(q^{\mathbf r}) \big).
    %\end{align*}}
  %{\color{red}
   % Note that the left hand side of the above equation is a Laurent polynomial in
    %$q$ and hence by Lemma \ref{elementary-properties-of-residues}, its residue
 %   vanishes. We conclude the proof by taking the residues $\text{Res}(-)$ of
  %  both
   % sides of the above equation and using Corollary \ref{residue-constant-g}.
  %}
\end{proof}

Let $L_1$ denote the coarse cotangent line bundle at the unique marking of
$Q^{\epsilon}_{0, 1}(X, \beta)$. We have the relation $L_1 =
\tilde{L}_1^{\mb{r}_1}$.

\begin{corollary} \label{evaluation-last-wall}
  % {\color{red}For
  % $\epsilon \in(\frac{1}{\emph{deg}(\beta)}, \frac{1}{\emph{deg}(\beta)-1})$,
  % w}
  We have
  \[
    (\check{\ev}_{1})_*\bigg(Q^{\epsilon_{ + }}_{0, 1}(X, \beta),
    \frac{\ovir_{Q^{\epsilon_{ + }}_{0, 1}(X, \beta)}}{1-qL_1} \bigg)
    =
      [(1-q)I_{\beta}(q)]_- .
  \]
    The same identity holds if we replace $\ovir_{Q^{\epsilon_ + }_{0,
        1}(X, \beta)}$ and $I_{\beta}(q)$ by their twisted counterparts.
\end{corollary}
\begin{proof}
  Consider the formal geometric series expansion:
  \[
    \frac{1}{1-qL_1} = \sum_{i = 0}^\infty q^iL_1^i .
  \]
  By Lemma \ref{special-evaluation}, Lemma
  \ref{elementary-properties-of-residues}, and Lemma \ref{negative-projection},
  we compute
  \begin{align*}
     (\check{\ev}_{1})_* \bigg(Q^{\epsilon_{ + }}_{0, 1}(X, \beta), \frac{\ovir_{Q^{\epsilon_{ +
    }}_{0, 1}(X, \beta)}}{1-qL_1} \bigg) & = \sum_{i = 0}^\infty q^i  (\check{\ev}_{1})_*
                                           \bigg(Q^{\epsilon_{ + }}_{0, 1}(X, \beta), \ovir_{Q^{\epsilon_{ + }}_{0, 1}(X,
                                           \beta)} \cdot L^{ \mb{r}_1}_1 \bigg) \\
                                         & = \sum_{i = 0}^\infty  q^i \cdot \text{Res}
                                           \big(
                                           q^{-i}(1-q) I_{\beta}(q) \big) \\
                                         & = [(1-q)I_{\beta}(q)]_-.
  \end{align*}

\end{proof}

\subsection{The main case}
We study the case $2g-2 + n + d \epsilon_0>0$. As in the previous subsection, we
use $\ovir$ to denote both the untwisted and twisted virtual structure sheaves. We will
mention the needed modifications and facts for the twisted case in the proofs.

By the $K$-theoretic virtual
localization formula~\eqref{eq:K-theoretic-localization}, we have
\begin{equation}\label{eq:localization-main-case}
  \begin{aligned}
    \ovir_{MQ^{\epsilon_0}_{g, n}(X, \beta)}=&
    (\iota_{F_+})_*
    \bigg(
    \frac{\ovir_{\tilde Q^{\epsilon_ + }_{g, n}(X, \beta)}}{\lambda^{\C^*}_{-1}\big(N^{\mathrm{vir},\vee}_{F_+/MQ^{\epsilon_0}_{g, n}(X,
  \beta)}\big)} \bigg)
    +
    (\iota_{F_-})_*
    \bigg(
    \frac{\ovir_{
        Q^{\epsilon_-}_{g, n}(X, \beta)}}{\lambda^{\C^*}_{-1}\big(N^{\mathrm{vir},\vee}_{F_-/MQ^{\epsilon_0}_{g,
    n}(X, \beta)}\big)}
    \bigg)
    \\
   &+   \sum_{i = 1}^{m}\sum_{\ub}
   (\iota_{F_{\ub}})_*
   \bigg(
   \frac{
      \ovir_{F_{\ub}} } { \lambda^{\mathbb C^*}_{-1}
      \big(N^{\mathrm{vir},\vee}_{F_{\ub}/MQ^{\epsilon_0}_{g, n}(X, \beta)}
      \big)}
    \bigg),
    \end{aligned}
  \end{equation}
where $m:=\lfloor d/d_0 \rfloor$, the sum is over all $\ub$ satisfying Condition
\ref{cond:condition-on-beta}, and $\iota_{F_+},\iota_{F_-}$ and $\iota_{F_{\ub}}$ are the embeddings of the
corresponding fixed-point components into $MQ^{\epsilon_0}_{g, n}(X, \beta)$.
 
We define the morphism
\begin{equation*}
\tau_0:MQ^{\epsilon_0}_{g, n}(X, \beta)\rightarrow Q^{\epsilon_-}_{g,n}(\bb{P}^N,d)
\end{equation*}
by
\begin{itemize}
\item composing the quasimaps with~\eqref{eq:equation-GIT-embedding},
\item taking the coarse moduli of the domain curves,
  \item taking the $\epsilon_-$-stabilization of the obtained quasimaps to $\bb{P}^N$.
  \end{itemize}
Let $\ev:MQ^{\epsilon_0}_{g, n}(X, \beta)\rightarrow (\bar{I}X)^n$ be the
  product of the rigidified evaluation maps at the $n$ markings. Consider the trivial
  $\C^*$-action on $Q^{\epsilon_-}_{g,n}(\bb{P}^N,d)\times (\bar{I}X)^n$. Then
  $\tau_0\times\ev$ is a $\C^*$-equivariant proper morphism. Define the map
  \begin{equation}\label{eq:embedding-contraction}
    \tau:\tilde Q^{\epsilon_ + }_{g, n}(X, \beta)\rightarrow Q^{\epsilon_-}_{g,n}(\bb{P}^N,d)
  \end{equation}
  as the
  restriction of $\tau_0$ to $ F_ + \cong \tilde Q^{\epsilon_ + }_{g, n}(X,
  \beta)$. Note that the restriction of $\tau_0$ to $ F_- \cong  Q^{\epsilon_-}_{g, n}(X,
  \beta)$ is the map $\iota$ defined in the introduction. In general, suppose $M$ is a quasimap space with a $\C^*$-equivariant proper morphism $\tau:M\rightarrow
  Q^{\epsilon_-}_{g,n}(\bb{P}^N,d)$ and a product of evaluation morphisms at certain
  markings $\ev:M\rightarrow (\bar{I}X)^n$ that are both obvious from the context. To
  simplify the notation, we write
  \[
\chi(M,E\cdot F):=(\tau\times\ev)_*(E\cdot F),\quad\text{for}\ E\in K_{\C^*}^\circ(M)_{\Q},F\in K_\circ^{\C^*}(M)_{\Q}.
  \]

As mentioned in Section \ref{epsilon + } and Section
\ref{epsilon-}, the virtual normal bundles of $F_+$ and $F_-$ in the $M$aster
space are given by $\bb{M}_+$ with weight $\fr{l}$ and $\bb{M}^{\vee}_-$ with
weight $-\fr{l}$, respectively. Hence, we have
\[
\lambda^{\C^*}_{-1}\big(N^{\mathrm{vir},\vee}_{F_+/MQ^{\epsilon_0}_{g, n}(X,
  \beta)}\big)
=1-q^{-\fr{l}}\bb{M}^\vee_+
\quad
\text{and}\quad\lambda^{\C^*}_{-1}\big(N^{\mathrm{vir},\vee}_{F_-/MQ^{\epsilon_0}_{g,
    n}(X, \beta)}\big)
=1-q^{\fr{l}}\bb{M}_-.
\]
%\marginpar{\footnotesize{
 %   \Yang{
  %    Too vague.
   % }
    %\Ming{changed it.}}}
Consider the proper pushforward of the
  relation~\eqref{eq:localization-main-case} along $\tilde{\tau}$.
By Corollary \ref{cor:master-space-technique}, we have
\begin{equation*}   \begin{aligned}
    0  = & \operatorname{Res}\bigg(\chi \Big(\tilde Q^{\epsilon_ + }_{g, n}(X, \beta),
    \frac{\ovir_{\tilde Q^{\epsilon_ + }_{g, n}(X, \beta)}}{1-q^{-\fr{l}} \bb{M}_ +
      ^\vee} \Big)\bigg) 
     +  \operatorname{Res}\bigg(\chi \Big( Q^{\epsilon_-}_{g, n}(X, \beta), \frac{\ovir_{
        Q^{\epsilon_-}_{g, n}(X, \beta)}}{1-q^{\fr{l}} \bb{M}_-}\Big) \bigg) \\
   &+   \sum_{i = 1}^{m}\sum_{\ub}
     \operatorname{Res}\bigg(\chi \Big( F_{\ub},  \frac{
      \ovir_{F_{\ub}} } { \lambda^{\mathbb C^*}_{-1}
      \big(N^{\mathrm{vir},\vee}_{F_{\ub}/MQ^{\epsilon_0}_{g, n}(X, \beta)}
      \big)}
    \Big)\bigg).
  \end{aligned}
\end{equation*}
Note that the residue operation commutes with proper pushforward. Using
Corollary \ref{residue-constant-g} and Corollary
\ref{elementary-properties-of-residues} (2), we obtain

\begin{equation} \label{eq:residue-localization-main-case}
  \begin{aligned}
    0  = & - \chi \bigg(\tilde Q^{\epsilon_ + }_{g, n}(X, \beta), \ovir_{\tilde
      Q^{\epsilon_ + }_{g, n}(X, \beta)} \bigg) + \chi \bigg( Q^{\epsilon_-}_{g,
      n}(X, \beta), \ovir_{\tilde Q^{\epsilon_-}_{g, n}(X, \beta)} \bigg) \\
    & + \sum_{i = 1}^{m} \sum_{\ub} \chi \bigg( F_{\ub}, \text{Res}
    \bigg( \frac{ \ovir_{F_{\ub}} } { \lambda^{\mathbb C^*}_{-1}
      \big(N^{\mathrm{vir},\vee}_{F_{\ub}/MQ^{\epsilon_0}_{g, n}(X, \beta)}
      \big)} \bigg) \bigg)
  \end{aligned}
\end{equation}
By Lemma \ref{lem:vir-str-sheaf-comp-entangled-tails}, the first term on the right
side is equal to
$- \chi ( Q^{\epsilon_ + }_{g, n}(X, \beta), \ovir_{
      Q^{\epsilon_ + }_{g, n}(X, \beta)} )
$.
    
Now let us compute the residues over $F_{\ub}$.
By the construction of the gluing morphism (\ref{eq:gluing-morphism}), it
descends to an \'etale morphism of degree $1/ \prod_{i = n + 1}^{n + k} \mathbf
r_i$
\[
  \left[\left(\tilde{\mathfrak M}_{g, n + k, d-kd_0} \times' {( \mathfrak
        M_{0, 1, d_0}^{\mathrm{wt}, \mathrm{ss}})}^k \right)/S_k \right]
  \rightarrow  \mathfrak Z_{(k)}, \]
where the symmetric group $S_k$ permutes the last $k$ factors. In fact, this morphism is the fiber product of the universal gerbes at the 
$k$ nodes of the universal curve of $ \mathfrak Z_{(k)}$. Recall from Section
\ref{subsection-general-case} that we have the following fibered diagram
\begin{equation*} \label{eq:glue-F-beta}
  \begin{tikzcd}
  \left[ {\tilde{\mathrm{gl}}_k^*}F_{\ub}/S_k
    \right] \arrow[r] \arrow[d,"\tilde{\mathrm{gl}}_{\ub}"] & \left[\left( \tilde{\mathfrak M}_{g, n + k,
          d-kd_0} \times' {\left( \mathfrak
            M_{0, 1, d_0}^{\mathrm{wt}, \mathrm{ss}} \right)}^k \right)/S_k
    \right] \arrow[d] \\
        F_{\ub} \arrow[r] & \mathfrak Z_{(k)}
  \end{tikzcd},
\end{equation*}
where $S_k$ acts on ${\tilde{\mathrm{gl}}_k^*}F_{\ub}$ by permuting the $k$
ordered entangled tails. It follows that the morphism 
$
  \tilde{\mathrm{gl}}_{\ub}
  $
  is the fiber products of the universal gerbes at the $k$ nodes of the
  universal curve of $F_{\ub}$.
  In particular, we have
\[
  \big(\tilde{\mathrm{gl}}_{\ub} \big)_*
  \ca{O}_{[{\tilde{\mathrm{gl}}_k^*}F_{\ub}/S_{k}]} =
  \ca{O}_{F_{\ub}}
\]
and
\[
  \big(\tilde{\mathrm{gl}}_{\ub} \big)^* \ovir_{F_{\ub}} =
  \ovir_{[{\tilde{\mathrm{gl}}_k^*}F_{\ub}/S_{k}]}.
\]
The same pullback relation holds for twisted virtual structure sheaves with
level structure.
Hence by the projection formula, we obtain 
\begin{equation} \label{eq:unglue-k-tails}
  \chi \bigg( F_{\ub}, \frac{ \ovir_{F_{\ub}} } {
    \lambda^{\mathbb
      C^*}_{-1} \big(N^{\mathrm{vir},\vee}_{F_{\ub}/MQ^{\epsilon_0}_{g,
        n}(X, \beta)} \big)}
  \bigg) = \chi \bigg( [{\tilde{\mathrm{gl}}_k^*}F_{\ub}/S_{k}],
  \frac{ \ovir_{[{\tilde{\mathrm{gl}}_k^*}F_{\ub}/S_{k}]} } {
    \lambda^{\mathbb
      C^*}_{-1} \big(N^{\mathrm{vir},\vee}_{F_{\ub}/MQ^{\epsilon_0}_{g,
        n}(X, \beta)}\big |_{{\tilde{\mathrm{gl}}_k^*}F_{\ub}} \big)}
  \bigg).
\end{equation}

Let $\vb$ be an ordered tuple whose underlying multiset is $\ub$. Note that
\[
  \left[ {\tilde{\mathrm{gl}}_k^*}F_{\ub}/S_k \right] = \left[
    {\tilde{\mathrm{gl}}_k^*}F_{\vec \beta}/S_{\vb} \right],
\]
where $S_{\vb}$ is the stabilizer subgroup of $S_k$ that fixes
$\vb$ under the action (\ref{eq:S_k-permutating-index}). By Lemma \ref{lem:obstruction-theory-F-beta}, the residue of the right
hand side of (\ref{eq:unglue-k-tails}) equals
\begin{equation}
  \label{eq:residue-gl-F-beta}
  \chi \bigg([\tilde{\mathrm{gl}}^*_k F_{\vb}/S_{\vb}], 
 \ \varphi^*\,\mathrm{Res} \bigg( \frac
  {\prod_{i = 1}^k(1-q^{\frac{\mathbf{r}_i\fr{l}}{k}}L(\mathcal E_i)^\vee)} {1-q^{
      - \frac{\fr{l}}{k}} \tilde L(\mathcal E_1) \cdot \tilde L_{n + 1} \cdot
    \ca{O}(\sum_{i = k}^{m-1} \mathfrak
    D_i))} \cdot
  \prod_{i = 1}^k \mathcal
  I_{\beta_i}(q^{\frac{\mathbf{r}_i\fr{l}}{k}}L(\mathcal E_i)^\vee)
  \bigg)
    \cdot \ovir_{[\tilde{\mathrm{gl}}^*_k F_{\vb}/S_{\vb}]}
  \bigg).
\end{equation}
 Recall that we have a morphism
\[
  \varphi:[\tilde{\mathrm{gl}}^*_k F_{\vb}/S_{\vb}]
\rightarrow
  \Big[\Big(Y  \times_{(IX)^{k}} 
\prod_{i = 1}^k F_{\star, \beta_i} \Big)/S_{\vb} \Big]
\]
which is \'etale of degree $\prod_i \mb{r}_i$.
Using the projection formula and Lemma \ref{lem:obstruction-theory-F-beta} again, we
rewrite (\ref{eq:residue-gl-F-beta}) as
\begin{equation} \label{eq:residue-Y-F-beta}
  \begin{aligned}
  \chi \bigg(
  \Big[\Big(Y \times_{(IX)^{k}}
  \prod_{i = 1}^k F_{\star, \beta_i}\Big)
  /S_{\vb} \Big],
  \ \mathrm{Res}
  &\bigg(
  \frac
  {\prod_{i = 1}^k\theta^{\mb{r}_i} \big(\tilde{L} (\ca{E}_i) \big)
    \cdot(1-q^{\frac{\mathbf{r}_i\fr{l}}{k}} L(\mathcal E_i)^\vee)
  }
  {1 - q^{- \frac{\fr{l}}{k}} \tilde L(\mathcal E_1) \cdot \tilde
    L_{n + 1} \cdot \ca{O}(\sum_{i = 0}^{m-k-1} \mathfrak D_i'))
  }
  \\
  &\cdot
  \prod_{i = 1}^k \mathcal I_{\beta_i}(q^{\frac{\mathbf{r}_i\fr{l}}{k}}L(\mathcal
  E_i)^\vee)
  \bigg)
  \cdot
  \big(\ovir_Y
  \boxtimes \textstyle \prod_{i = 1}^k \ovir_{F_{\star, \beta_i}} \big)
  \bigg).
  \end{aligned}
\end{equation}
Here $\chi$ denotes the proper pushforward along the morphism
$(b_k\circ\tau\circ p)\times\ev$, where
\begin{itemize}
\item $\tau:[ \tilde{Q}^{\epsilon_ + }_{g, n + k}(X,
  \beta^\prime)/S_{\vb}]
  \rightarrow
 [Q^{\epsilon_-}_{g,n+k}(\bb{P}^N, d-kd_0)/S_k]$ is defined similarly
  to~\eqref{eq:embedding-contraction},
  \item $b_k:[Q^{\epsilon_-}_{g,n+k}(\bb{P}^N, d-kd_0)/S_k]\rightarrow Q^{\epsilon_-}_{g,n}(\bb{P}^N, d)$
    is the map that replaces the last $k$ markings by base points of length $d_0$,
    \item $\ev: [(Y \times_{(IX)^{k}}
  \prod_{i = 1}^k F_{\star, \beta_i})
  /S_{\vb} ]\rightarrow (\bar{I}X)^n$ is the
  product of the rigidified evaluation maps at the first $n$ markings
\end{itemize}
By Remark~\ref{rem:compare-twisting}, the formula~\eqref{eq:residue-Y-F-beta} is
also valid in the twisted case if we replace $\ca{I}_\beta(q)$, $\ovir_Y$ and
$\ovir_{F_{\star,\beta_i}}$ by their twisted counterparts.

According to (2) of Lemma \ref{elementary-properties-of-residues}, after the change of
variable
\[
  q \mapsto
  q^{k/\fr{l}} \tilde L(\mathcal E_1)^{k} \otimes \tilde
  L_{n + 1}^{k} = \cdots = q^{k/\fr{l}} \tilde L(\mathcal E_k)^{k} \otimes \tilde L_{n +
    k}^{k},
\]
\eqref{eq:residue-Y-F-beta} becomes
% \Mingg{NEW: the correct change of variable should be
% \[
%   q \rightarrow q^k \ca{O}_{\mathfrak E_{k-1}^*}(k).
% \]
% Forgetting the $S_{\vec{\beta}}$-action, we have the isomorphism:
% \[
%   \ca{O}_{\mathfrak E_{k-1}^*}(1) \cong L(\mathcal E_1) \otimes \tilde L_{n +
%   1} \cong \cdots \cong \tilde L(\mathcal E_k) \otimes \tilde L_{n + k},
% \]
% }

%   \Ming{I am still confused about the torus action here. Is the residue in (14)
%   well-defined or is it defined using the above change of variable.}
%   \Yang{What formula is correct is clear but the reason is subtle. It is due to
%   this
%   epidemic abuse of the notation of fixed substack in the field. Technically
%   speaking, when fractional weights appears on the normal bundle, the fixed
%   substack is not what you thought it should be. Please see my note
%   \textit{example-fractional-weight.pdf} for the concrete worked out example
%   $\mathbb P(1, 2)$.}

%   \Ming{I see. Great example. I addressed your question in that file.}

\begin{equation} \label{eq:residue-Y-F-beta-2}
  \begin{aligned}
    \chi \bigg(
 \Big[\Big(Y \times_{(IX)^{k}}
  \prod_{i = 1}^k F_{\star, \beta_i}\Big)
  /S_{\vb} \Big],
  \ \mathrm{Res}
  &\bigg(  \frac
  {\prod_{i = 1}^k\theta^{\mb{r}_i} \big(\tilde{L} (\ca{E}_i) \big)
    \cdot(1-q^{\mathbf{r}_i}L_{n + i})} {1-q^{
      -1} \ca{O}(\sum_{i = 0}^{m-k-1} \mathfrak
    D_i'))}
  \\
  &\cdot
  \prod_{i = 1}^k
  \mathcal I_{\beta_i}(q^{\mathbf{r}_i}L_{n + i})\bigg)
  \cdot
\big(\ovir_Y
  \boxtimes \textstyle \prod_{i = 1}^k \ovir_{F_{\star, \beta_i}} \big)
  \bigg) \bigg).
\end{aligned}
\end{equation}
Consider the natural morphism
\[
p_k:\Big[\Big(Y \times_{(IX)^{k}}
  \prod_{i = 1}^k F_{\star, \beta_i}\Big)
  /S_{\vb} \Big]
  \rightarrow
  \Big[\Big(Y \times_{(\bar{I}X)^{k}}
  \prod_{i = 1}^k F_{\star, \beta_i}\Big)
  /S_{\vb} \Big],
\]
where the target is the fiber product of $\ev_Y$ and $\check{\ev}_{\star,
  \beta_i}$ over $(\bar{I}X)^k$. Note that $p_k$ is a gerbe whose elements in
the automorphism groups act on $\tilde{L} (\ca{E}_i)$ as $\mb{r}_i$-th roots
of unity. It follows that $(p_k)_*(\prod_{i=1}^k\tilde{L}^{a_i} (\ca{E}_i))=0$
unless $a_i$ is divisible by $\mb{r}_i$ for all $i$. Hence
\[
(p_k)_*\big(\prod_{i = 1}^k\theta^{\mb{r}_i} \big(\tilde{L} (\ca{E}_i) \big)\big)=1.
\]
By applying the projection formula to $p_k$, we rewrite~\eqref{eq:residue-Y-F-beta-2}
as
\begin{equation*} 
    \chi \bigg(
 \Big[\Big(Y \times_{(\bar{I}X)^{k}}
  \prod_{i = 1}^k F_{\star, \beta_i}\Big)
  /S_{\vb} \Big],
  \ \mathrm{Res}
  \bigg(  \frac
  {\prod_{i=1}^k(1-q^{\mathbf{r}_i}L_{n + i})\mathcal I_{\beta_i}(q^{\mathbf{r}_i}L_{n + i})} {1-q^{
      -1} \ca{O}(\sum_{i = 0}^{m-k-1} \mathfrak
    D_i'))}
  \bigg)
  \cdot
\big(\ovir_Y
  \boxtimes \textstyle \prod_{i = 1}^k \ovir_{F_{\star, \beta_i}} \big)
  \bigg).
\end{equation*}
Recall that we have morphisms
\[ \Big[\Big(Y \times_{(\bar{I}X)^{k}}
  \prod_{i = 1}^k F_{\star, \beta_i}\Big)
  /S_{\vb} \Big]
  \overset{\mathrm{pr}_Y}{\rightarrow}
  [Y/S_{\vb}]
  \overset{p}{\rightarrow}
  [ \tilde{Q}^{\epsilon_ + }_{g, n + k}(X, \beta^\prime)/S_{\vb}].
\]
By taking the pushforward along $\text{pr}_Y$, we obtain
\begin{equation*}
  \begin{aligned}
     & \chi \bigg([Y/S_{\vb}], \ \mathrm{Res} \bigg( \frac
  {\prod_{i = 1}^k \ev_{n + i}^* ((1-q^{\mathbf{r}}L_{n + i})
    I_{\beta_i}(q^{\mb{r}}L_{n + i}))} {1-q^{-1} \ca{O}(\sum_{i = 0}^{m-k-1}
    \mathfrak D_i'))} 
  \bigg)
    \cdot \ovir_{
      [Y/S_{\vb}]
    }
    \bigg)\\
    =&
  \chi \bigg([Y/S_k], \ \sum_{\vb \mapsto \ub} \mathrm{Res} \bigg( \frac
  {\prod_{i = 1}^k \ev_{n + i}^* ((1-q^{\mathbf{r}}L_{n + i})
    I_{\beta_i}(q^{\mb{r}}L_{n + i}))} {1-q^{-1} \ca{O}(\sum_{i = 0}^{m-k-1}
    \mathfrak D_i'))} 
  \bigg)
    \cdot \ovir_{
      [Y/S_{k}]
    }
    \bigg).
    \end{aligned}
\end{equation*}
  Here we have extended the definition of $\mathrm{ev}_{n + i}^*$ by linearity:
  we
  first compute $(1-q^{\mathbf{r}}L_{n + i}) I_\beta(q^{\mb{r}}L_{n + i})$ by
  viewing $L_{n + i}$ as a formal variable, and then we agree that
  \[
    \mathrm{ev}_{n + i} (q^{a}L_{n + i}^b \alpha) : = q^{a}L_{n + i}^b
    \mathrm{ev}_{n + i} (\alpha), \quad a, b \in \mathbb Z, \alpha \in
    K(\bar{I}X)\otimes\Lambda.
  \]
  The above equality follows from the elementary fact that given two finite
  groups $H\subset G$ and a $H$-module $V$, we have
  $(\operatorname{Ind}^G_HV)^G=V^H$. Here $\operatorname{Ind}^G_HV$ denotes the
  induced $G$-module. 

Note that $Y \rightarrow \tilde{Q}^{\epsilon_ + }_{g, n + k}(X, \beta^\prime)$
is a gerbe
banded by $\mu_k$. Hence we have
\[
  p_* \ca{O}_{[Y/S_{k}]} = \ca{O}_{[ \tilde{Q}^{\epsilon_ + }_{g, n +
      k}(X, \beta^\prime)/S_{k}]}.
\]
By applying the projection formula again, we conclude that the correction terms
in the wall-crossing formula (\ref{eq:residue-localization-main-case}) equal
\begin{equation}
  \label{eq:contr-Q-3}
  \sum_{k = 1}^m \sum_{\vb} \chi \bigg([
  \tilde{Q}^{\epsilon_ + }_{g, n + k}(X, \beta^\prime)/S_k], \
  \mathrm{Res} \bigg( \frac
  {\prod_{i = 1}^k \ev_{n + i}^* \big(
  (1-q^{\mathbf{r}}L_{n +
    i})I_{\beta_i}(q^{\mb{r}}L_{n + i})
    \big)
  }
  {1-q^{-1} \ca{O}(\sum_{i = 0}^{m-k-1} \mathfrak D_i'))}
  \bigg)
  \cdot \ovir_{
      [
  \tilde{Q}^{\epsilon_ + }_{g, n + k}(X, \beta^\prime)/S_k]
    }
    \bigg).
\end{equation}
Here we abuse the notation and write the summation sign $\sum_{\vb}$ outside
of the proper pushforward. By our convention, $\chi$ denote the proper
pushforward along the map $(b_k\circ\tau)\times\ev$. The
formula~\eqref{eq:contr-Q-3} is also valid in the twisted case if we replace
$I_{\beta_i}$ and $\ovir_{
      [
  \tilde{Q}^{\epsilon_ + }_{g, n + k}(X, \beta^\prime)/S_k]}$ by their twisted counterparts.

In the rest of the section, we will simplify~\eqref{eq:contr-Q-3} and prove the
following:
\begin{proposition} \label{prop-general-case-simplified}
  The wall-crossing contribution~\eqref{eq:contr-Q-3} is equal to 
  \begin{equation*}
    \sum_{k = 1}^m \sum_{\vec{\beta}} \chi\bigg(
    [Q^{\epsilon_+}_{g,n+k}(X,\beta')/S_k]
    ,
    \prod_{i=1}^k\ev_{n+i}^*\mu_{\beta_i}(L_{n+i})
    \cdot
    \ovir_{[Q^{\epsilon_+}_{g,n+k}(X,\beta')/S_k]}
  \bigg)
\end{equation*}
in $K_\circ\big([(Q^{\epsilon_-}_{g,n}(\bb{P}^N, d)\times (\bar{I}X)^n\big)/S_n]\big)_{\Q}$.
\end{proposition}

To simplify the exposition, we introduce some notation here. For any rational
function $f(q)$ in $q$ with coefficients in some $K$-group, we denote by $(f)_0$
the formal Laurent series
expansion of $f(q)$ at
$q = 0$ and by $(f)_\infty$ the formal Laurent series expansion of $f(1/q)$ at
$q = 0$. Recall that the notation $[g(q)]_{s}$ denotes the coefficient of $q^s$
of a
formal Laurent series $g(q)$ at $q = 0$. Consider the permutation action of
$S_k$ of $(\bar{I}X)^k$. Let $S$ be a subgroup of $S_k$. For $\ca{G}\in K_\circ([\tilde{Q}^{\epsilon_ + }_{g, n + k}(X, \beta^\prime)
    /S])$, we denote
\[
 \left \langle \ca{G} \right \rangle^{S}_k:
   = \chi \left([\tilde{Q}^{\epsilon_ + }_{g, n + k}(X, \beta^\prime)
    /S],  \ca{G} \cdot \ovir_{[\tilde{Q}^{\epsilon_ + }_{g, n + k}(X, \beta^\prime)
    /S]}\right).
\]
When using the bracket notation, we omit the moduli space and the pullback via
evaluation maps. Recall that for a partition $\vb=(\beta^\prime, \beta_1 ,
\ldots, \beta_k)$, we also denote $\beta_i$ by $\vb_{(i)}$. This notation is
introduced to avoid confusion when evaluation maps are omitted in the bracket
notation. In the discussion below, a typical example of $\ca{F}$ is of
the form $\sum_{\vb}\sum \prod_i C_{\vb_{(i)}}$, where $C_{\vb_{(i)}}$ are certain classes
depending on $I_{\beta_i}$. We denote 
\[
 \prod_{i=1}^k\fbi:= \big (\prod_{i=1}^k\ev_{n+i}\big)^*\Big((1-q^{\mathbf{r}}L_{n +
    i})I_{\beta_i}(q^{\mb{r}}L_{n + i})\Big).
\]

 As the first step in simplifying~\eqref{eq:contr-Q-3}, we evaluate the residues within and prove the following: 
\begin{proposition}
  \label{general-case}
  We have
  \begin{equation}
    \label{eq:general-case}
    \begin{aligned}
      & \sum_{k = 1}^m \sum_{\vec{\beta}}
      \bigg\langle 
      \mathrm{Res} \bigg(\frac {\prod_{i = 1}^k   f_{\vb_{(i)}}
  }
{1 - q^{-1}
        \ca{O} (\sum_{i = 0}^{m-k-1} \mathfrak D_i')} \bigg)
      \bigg\rangle^{S_k}_k \\
      = & - \sum_{k = 1}^m \sum_{\vec{\beta}} \sum_{s \geq 1} \Big
      \langle[ \prod_{i = 1}^k (\fbi)_0 ]_{-s}
      \Big \rangle^{S_k}_k - \sum_{k = 1}^m \sum_{\vec{\beta}}
      \sum_{t \geq 0}
      \Big \langle [\prod_{i = 1}^k  (\fbi)_\infty]_{-t} \Big
      \rangle^{S_k}_k \\
      & - \sum_{k = 2}^{m} \sum_{r = 1}^{k-1} \sum_{\vec{\beta},\vb'}
      \bigg(
      \sum_{\substack{t \geq0 \\ s \geq1}}
      \sum_{\substack{ j_1 + \cdots + j_r = t \\ j_u \geq0}}
      \Big \langle
      [ \prod_{i = 1}^{k-r} (\fbi)_0]_{-t-s}\cdot
      \prod_{u = 1}^r
      \Big(- \mathrm{Res}
      \big(
      q^{-j_u} f_{\vb'_{(u)}}
      \big) \Big)
      \Big \rangle^{S_{k-r}\times S_k }_k \\
      & \hspace{1.65cm} + \sum_{\substack{t \geq 0 \\ s \geq 1}}
      \sum_{\substack{j_1 + \cdots + j_r = t + 1 \\ j_u>0}}
      \Big \langle
      [\prod_{i = 1}^{k-r}  (\fbi)_\infty]_{-t-s}
      \cdot
      \prod_{u= 1}^r \mathrm{Res}
      \big(q^{j_u} f_{\vb'_{( u)}} \big)
      \Big \rangle^{S_{k-r}\times S_r }_k \bigg).
    \end{aligned}
  \end{equation}
  %\Yang{
   % The first line of right hand side can be simplified as
  %  \[
   %   - \sum_{k = 1}^m \sum_{\vec{\beta}_\order}
    %  \Big \langle
     % - \mathrm{Res} \big(\frac{\prod_{i = 1}^k f_{\beta_i}}{1 - q^{-1}} \big)
      % [ \prod_{i = 1}^k (f_{\beta_i})_0 ]_{-s}
    %  \Big \rangle^{S_k}.
   % \]
    %I don't know if this is helpful.
 % }
%  \Yang{
 %   The indices could be simplified. For example, by introducing $a = t + s -
  %  1$, we
   % could write the third line as
    %\[
     % - \sum_{k = 2}^{m} \sum_{r = 1}^{k-1} \sum_{\vec{\beta}}
      %\bigg(
     % \sum_{a = 0}^\infty
      %\sum_{| \vec j| \leq a }
     % \Big \langle
      %[ \prod_{i = 1}^{k-r} (f_{\beta_i})_0]_{-a - 1},
%      \prod_{i = 1}^r
 %     \Big(- \emph{Res}
  %    \big( q^{-j_i} f_{\beta_{k-r + i}}
   %   \big) \Big)
    %  \Big \rangle^{S_{\vec{\beta}} } ,
 %   \]
  %  where $\vec{j} = (j_1 , \ldots, j_r)$, $j_{i} \geq 0$, $| \vec{j}| =
   % \sum_{i = 1}^r j_i$.
 %   Is $\vec{j}$ here ordered or unordered?
  %}
%
 % \Ming{The indices in the proof are really hard to keep track of. It is very
  %  likely that there are mistakes in the indices.}
\end{proposition}

Suppose $    \vec \beta' = (\beta'', \beta'_1, \ldots, \beta'_r)
$
is an ordered decomposition of the class $\beta'$ satisfying Condition
\ref{cond:condition-on-beta-ordered}. Set $d':=\mathrm{deg}(\beta')$ and
$\vb'_{(i)}:=\beta'_i$ for $i=1,\dots,r$. Consider the diagram
\[
  \begin{tikzcd}[column sep = 5em]
    & \tilde{Q}^{\epsilon_+}_{g,n+k}(X,\beta^\prime) \ar[d, "{\tau \times \ev}"] \\
    \tilde{Q}^{\epsilon_+}_{g, n + k + r}(X, \beta^{\prime \prime}) \ar[r, "(b\,
    \circ\, \tau) \times \ev"] &
    {Q}^{\epsilon -}_{g, n + k}(\mathbb P^{N}, d^{\,\prime}) \times
     (\bar{I}X)^{n + k}.
  \end{tikzcd}
\]
As before, $\tau$ is defined similarly to~\eqref{eq:embedding-contraction}, $b$ replaces
the last $r$ markings by base points of length $d_0$, and $\ev$ is the
product of the rigidified evaluation maps at the first $n+k$ markings. %Recall from Section \ref{split-tail} that for $r = 1, \dots, m-k$, we have
%\begin{equation}
%  \label{eq:further-split-tails}
%  \tilde Q^{ \epsilon_ + }_{g, n + k}(X, \beta') |_{{ \tilde{
 %       \mathrm{gl}}_r^*} \mathfrak D'_{r-1}}
%  \xrightarrow{p}
%  \coprod_{\vec \beta_\order'} \tilde
%  Q^{\epsilon_ + }_{g, n + k + r}(X, \beta'') \times_{(IX)^r} \prod_{i = 1}^r
%  Q^{\epsilon_ + }_{0, 1}(X, \beta_{i}') \xrightarrow{p_1} \tilde
 % Q^{\epsilon_ + }_{g, n + k + r}(X, \beta'')
%\end{equation}
%where $p$ is the inflated projective bundle $\tilde{\mathbb P}(\Theta_1 \oplus
%\cdots \oplus \Theta_r)$, the disjoint union is over all ordered decompositions
%for $\beta'$ satisfying Condition \ref{cond:condition-on-beta-ordered}, and
%$p_1$ is the projection.
To prove Proposition \ref{general-case},
we need the following lemma.
%\Ming{Do we need to reminder the reader that $b$ is different from $b_r$ (it
 % appeared as $b_k$ before)?}
\begin{lemma} \label{special-lemma}
  For $a = 1, 2,\dots, m-k$ and $s \geq 1$, we have the following identities between $S_{n + k}$-equivariant $K_{\circ}$-classes on
  ${Q}^{\epsilon -}_{g, n + k}(\mathbb P^{N-1}, d^{\,\prime}) \times
     (\bar{I}X)^{n + k}$:
 \begin{equation}\label{eq:simplication-1}
    \begin{aligned}
     &  (\tau \times \ev)_*\Big( \mathcal {O}(- \sum_{i = 0}^{a-1} \mathfrak
      D_i') \cdot \mathcal
      O^{\mathrm{vir}}_{\tilde{Q}^{\epsilon_+}_{g,n+k}(X,\beta^\prime)} \Big) \\
      = &
      \sum_{r=0}^{a} \biggl[ \sum_{\vec\beta^\prime} ((b \circ \tau) \times
      \ev)_{*} \Big( \prod_{u=1}^{r}
      \ev_{n+k+u}^{*}( -
      \mathrm{Res}(f_{\vb'_{(u)}})) \cdot \mathcal
      O^{\mathrm{vir}}_{\tilde{Q}^{\epsilon_+}_{g, n + k + r}(X, \beta^{\prime
          \prime})} \Big) \biggr]^{S_r},
    \end{aligned}
  \end{equation}
 \begin{equation}\label{eq:simplication-2}
    \begin{aligned}
      &  (\tau \times \ev)_*\Big(
      \mathcal {O}(-(s+1) \sum_{i = 0}^{m-k-1} \mathfrak
      D_i') \cdot \mathcal
      O^{\mathrm{vir}}_{\tilde{Q}^{\epsilon_+}_{g,n+k}(X,\beta^\prime)} \Big)
      -(\tau \times \ev)_*\Big(
      \mathcal {O}(-s \sum_{i = 0}^{m-k-1} \mathfrak
      D_i') \cdot \mathcal
      O^{\mathrm{vir}}_{\tilde{Q}^{\epsilon_+}_{g,n+k}(X,\beta^\prime)} \Big)
      \\
      = &
      \sum_{r=1}^{m-k} \biggl[
      \sum_{\vec\beta^\prime}
      \sum_{
      \substack{j_1 + \cdots + j_r  = s \\j_u \geq0}}
      ((b \circ \tau) \times
      \ev)_{*} \Big( \prod_{u=1}^{r} \ev_{n+k+u}^{*}( -
      q^{-j_u}\mathrm{Res}(f_{\vb'_{(u)}})) \cdot \mathcal
      O^{\mathrm{vir}}_{\tilde{Q}^{\epsilon_+}_{g, n + k + r}(X, \beta^{\prime
          \prime})} \Big) \biggr]^{S_r},
    \end{aligned}
  \end{equation}
  and
   \begin{equation}\label{eq:simplication-3}
    \begin{aligned}
      &  (\tau \times \ev)_*\Big(
      \mathcal {O}(s \sum_{i = 0}^{m-k-1} \mathfrak
      D_i') \cdot \mathcal
      O^{\mathrm{vir}}_{\tilde{Q}^{\epsilon_+}_{g,n+k}(X,\beta^\prime)} \Big)
      -(\tau \times \ev)_*\Big(
      \mathcal {O}((s-1) \sum_{i = 0}^{m-k-1} \mathfrak
      D_i') \cdot \mathcal
      O^{\mathrm{vir}}_{\tilde{Q}^{\epsilon_+}_{g,n+k}(X,\beta^\prime)} \Big)
      \\
      = &
      \sum_{r=1}^{m-k} \biggl[
      \sum_{\vec\beta^\prime}
      \sum_{
      \substack{j_1 + \cdots + j_r  = s \\j_u >0}}
      ((b \circ \tau) \times
      \ev)_{*} \Big( \prod_{u=1}^{r}
      \ev_{n+k+u}^{*}( 
      q^{j_u}\mathrm{Res}(f_{\vb'_{(u)}})) \cdot \mathcal
      O^{\mathrm{vir}}_{\tilde{Q}^{\epsilon_+}_{g, n + k + r}(X, \beta^{\prime
          \prime})} \Big) \biggr]^{S_r},
    \end{aligned}
  \end{equation}
   where $[\cdot]^{S_r}$ denotes taking the $S_r$-invariant part. The same
   formulas hold in the twisted setting if we replace the virtual structure
   sheaves and $f_{\vb'_{(u)}}$ by their twisted counterparts.

\end{lemma}
\begin{proof}
 For $r = 0,1,2 ,\ldots, $
  we define the operator
  \[
    (\cdot)|_{\mathfrak D^\prime_{r-1}} : K_{\circ}(\tilde{Q}^{\epsilon +}_{g,n+k}(X,
    \beta^\prime)) \rightarrow K_{\circ}(\tilde{Q}^{\epsilon +}_{g,n+k}(X, \beta^\prime))
  \]
  as
  \begin{equation}
    \label{eq:restriction}
    \mathcal F|_{\mathfrak D_{r-1}^\prime} := (\tilde\iota_{\mathfrak D_{r-1}^\prime})_*
    \iota_{\mathfrak D_{r-1}^{\prime}}^{!} \mathcal F,
  \end{equation}
  where $\tilde\iota_{\mathfrak D_{r-1}^\prime}$ and $\iota_{\mathfrak
    D_{r-1}^{\prime}}$ are the obvious arrows in the following fibered diagram
  \[
    \begin{tikzcd}
      \tilde{Q}^{\epsilon +}_{g,n+k}(X, \beta^\prime)|_{\mathfrak D_{r-1}^\prime}
      \ar[d] \ar[r, "\tilde\iota_{\mathfrak D_{r-1}^\prime}"] & \tilde{Q}^{\epsilon
        +}_{g,n+k}(X, \beta^\prime) \ar[d]\\
      \mathfrak D_{r-1}^\prime \ar[r, "\iota_{\mathfrak D_{r-1}^{\prime}}"]
      & \tilde {\mathfrak M}_{g, n + k, d^\prime}
    \end{tikzcd}
  \]
  Note that for any $\alpha \in K^{\circ}(\tilde{Q}^{\epsilon
        +}_{g,n+k}(X, \beta^\prime))$, we have
  \[
    (\alpha \cdot \mathcal F)|_{\mathfrak D_{r-1}^\prime} =
    \alpha \cdot (\mathcal F|_{\mathfrak D_{r-1}^\prime}).
  \]
  Indeed, we have
  \[
    (\alpha \cdot \mathcal F)|_{\mathfrak D_{r-1}^\prime}
    = (\tilde\iota_{\mathfrak D_{r-1}^\prime})_* \iota_{\mathfrak D_{r-1}^{\prime}}^{!}
    (\alpha\cdot \mathcal F)
    = (\tilde\iota_{\mathfrak D_{r-1}^\prime})_* (\tilde\iota_{\mathfrak
      D_{r-1}^\prime}^*\alpha \cdot \iota_{\mathfrak D_{r-1}^{\prime}}^{!}
    \mathcal F) = \alpha \cdot (\mathcal F|_{\mathfrak D_{r-1}^\prime}). \quad
     \]
  For any $\mathcal F\in K_\circ(\tilde{Q}^{\epsilon +}_{g,n+k}(X,
  \beta^\prime))$, we have
  \[
    \mathcal O(-\mathfrak D_i^\prime)\cdot \mathcal F = \mathcal F - \mathcal
    F|_{\mathfrak D_i^\prime}.
  \]
  This follows from the definition of the refined Gysin map (see, for example,
  \cite[\textsection{2.1}]{Lee}).
  By repeatedly using this relation, we have for any $\mathcal F\in
  K_\circ(\tilde{Q}^{\epsilon +}_{g,n+k}(X, \beta^\prime))$,
  \begin{equation}
    \label{eq:divisor-sequence-recursive-1}
    \mathcal O(- \sum_{i=0}^{a-1} \mathfrak D_i^\prime) \cdot \mathcal F =
    \mathcal F - \sum_{r=1}^{a} \mathcal O(- \sum_{i=r}^{a-1} \mathfrak
    D_{i}^\prime) \cdot \mathcal F|_{\mathfrak  D_{r-1}^\prime}.
  \end{equation}
We further consider the fibered diagram
  \[
    \begin{tikzcd}
      {[\tilde{Q}^{\epsilon +}_{g,n+k}(X,
      \beta^\prime)|_{\tilde{\mathrm{gl}}^*_{r} \mathfrak D_{r}^\prime \ar[r]}/ S_r]}
      \ar[d]
      & \tilde{Q}^{\epsilon +}_{g,n+k}(X, \beta^\prime)|_{\mathfrak
        D_{r-1}^\prime} \ar[d, "p_{\mathfrak D_{r-1}^\prime}"'] \\
      {[\tilde{\mathrm{gl}}^*_{r} \mathfrak D_{r-1}^\prime / S_r]} \ar[r] &  \mathfrak D_{r-1}^\prime
    \end{tikzcd}.
  \]
  Note that the horizontal arrows are \'etale gerbes. Hence replacing $\mathfrak
  D_{r-1}^\prime$ by $[\tilde{\mathrm{gl}}^*_{r} \mathfrak D_{r-1}^\prime /S_r]$ in
  \eqref{eq:restriction} does not change the definition of $\mathcal
  F|_{\mathfrak D_{r-1}^\prime}$. In particular
  $\mathcal O^{\mathrm{vir}}_{\tilde{Q}^{\epsilon +}_{g,n+k}(X, \beta^\prime)}
  |_{\mathfrak D_{r-1}^\prime}$ is equal to the pushforward of
  $\mathcal O^{\mathrm{vir}}_{[\tilde{Q}^{\epsilon +}_{g,n+k}(X,
    \beta^\prime)|_{\tilde{\mathrm{gl}}^*_{r} \mathfrak D_{r-1}^\prime} / S_r]}$.
  Recall that
  \[
    p : \tilde{Q}^{\epsilon +}_{g,n+k}(X, \beta^\prime)|_{\tilde{\mathrm{gl}}^*_{r}
      \mathfrak D_{r-1}^\prime}
    \rightarrow
    \coprod_{\vec\beta^\prime}\tilde{Q}^{\epsilon_+}_{g, n+k+r}(X, \beta^{\prime
      \prime}) \times_{(IX)^r} \prod_{u=1}^r Q^{\epsilon_+}_{0,1}(X, \beta^\prime_{u})
  \]
  is an inflated projective bundle.
  By Lemma~\ref{lem:divisors1} (2), the restriction of $\mathcal O(- \sum_{i=r}^{a-1} \mathfrak
  D_{i}^\prime)$ to $\tilde{Q}^{\epsilon +}_{g,n+k}(X, \beta^\prime)|_{\tilde{\mathrm{gl}}^*_{r}
    \mathfrak D_{r - 1}^\prime}$ is equal to the pullback of $\mathcal
  O(-\sum_{i= 0}^{a - 1 - r}\mathfrak D_{i}^{\prime\prime})$ as $S_{n+k}\times S_{r}$-equivariant
  sheaves, where $\mathfrak D_i^{\prime\prime}$ are the boundary divisors on
  $\tilde{\mathfrak M}_{g, n + k + r, d^{\prime\prime}}$.
  Note that we have a commutative diagram
  \[
    \begin{tikzcd}
      {\tilde{Q}^{\epsilon +}_{g,n+k}(X,
        \beta^\prime)|_{\tilde{\mathrm{gl}}^*_{r} \mathfrak D_{r - 1}^\prime \ar[r]}}
      \ar[r] \ar[d, "p"']
      &
      \tilde{Q}^{\epsilon +}_{g,n+k}(X, \beta^\prime)  \ar[r, "\tau \times
      \ev"]
      &
      {Q}^{\epsilon_-}_{g, n + k}(\mathbb P^{N-1}, d^\prime) \times (\bar{I}X)^{n+k}
      \\
      {\coprod_{\vec\beta^\prime}\tilde{Q}^{\epsilon_+}_{g, n+k+r}(X, \beta^{\prime
          \prime}) \times_{(IX)^r} \prod_{u=1}^r Q^{\epsilon_+}_{0,1}(X,
        \beta^\prime_{u})} \ar[rr, "\mathrm{pr}_1"]
      & & \coprod_{\vec\beta^\prime}\tilde{Q}^{\epsilon_+}_{g, n+k+r}(X, \beta^{\prime
       \prime})
     \ar[u, "(b\circ c) \times \ev"]
    \end{tikzcd}
  \]
  % where the lower horizontal arrows is the composition of
  % \[
  %  {\coprod_{\vec\beta^\prime}\tilde{Q}^{\epsilon_+}_{g, n+k+r}(X, \beta^{\prime
  %         \prime}) \times_{(IX)^r} \prod_{i=1}^r Q^{\epsilon_+}_{0,1}(X, \beta^\prime_{i})}
  %    \overset{i}{\longrightarrow}
  %    \coprod_{\vec\beta^\prime}\tilde{Q}^{\epsilon_+}_{g, n+k+r}(X, \beta^{\prime
  %      \prime})
  %    \overset{(b\circ i) \times \mathrm{ev}}{\longrightarrow}
  %    {Q}^{\epsilon_-}_{g, n + k}(\mathbb P^{N-1}, d^\prime) \times (IX)^{n+k}
  % \]
  where the maps are $S_{n + k}$-equivariant and $S_i$-invariant.
  % and
  % the
  % lower horizontal
  Hence as $S_{n+k}\times S_r$-equivariant $K_{\circ}$-classes on
  ${Q}^{\epsilon_-}_{g, n + k}(\mathbb P^{N-1}, d^\prime) \times (\bar{I}X)^{n+k}$,
  where $S_r$ acts trivially, we have
  \begin{equation*}
    \begin{aligned}
      & (\tau \times \ev)_* \big(
      \mathcal O(- \sum_{i=r}^{a-1} \mathfrak D_{i}^\prime) \cdot
      \mathcal O^{\mathrm{vir}}_{\tilde{Q}^{\epsilon +}_{g,n+k}(X,
        \beta^\prime)}|_{\mathfrak D_{r  - 1}^\prime}
      \big)\\
      = &
      \biggl[
      ((b\circ \tau) \times \ev)_* \biggl(
      (\mathrm{pr}_1)_*p_*\big( \mathcal
     O^{\mathrm{vir}}_{\tilde{Q}^{\epsilon +}_{g,n+k}(X,
        \beta^\prime)|_{\tilde{\mathrm{gl}}^*_{r} \mathfrak D_{r-1}^\prime }}
      \big) \cdot
      \mathcal
      O(-\sum_{i= 0}^{a - 1 - r}\mathfrak D_{i}^{\prime\prime})
      \biggr)
      \biggr]^{S_r}.
    \end{aligned}
  \end{equation*}
 We denote by $\mathcal O^{\mathrm{vir}}_{\tilde{Q}^{\epsilon_+}_{g, n+k+r}(X,
      \beta^{\prime \prime})} \otimes_{(IX)^k} \prod_{u=1}^r
    \mathcal O^{\mathrm{vir}}_{Q^{\epsilon_+}_{0,1}(X, \beta^\prime_{u})}$ the
    structure sheaf of the fiber product. By Lemma~\ref{lem:int-with-D-k-1},
  \begin{equation*}
    \begin{aligned}
      &((b\circ \tau) \times \ev)_* (\mathrm{pr}_1)_*p_*\big( \mathcal
      O^{\mathrm{vir}}_{\tilde{Q}^{\epsilon +}_{g,n+k}(X,
        \beta^\prime)|_{\tilde{\mathrm{gl}}^*_{r} \mathfrak D_{r-1}^\prime }}
      \big) \\
      = &
     ((b\circ \tau) \times \ev)_*  \sum_{\vec\beta^\prime}(\mathrm{pr}_1)_*\big(
      p_*( \mathcal
     O_{\tilde{Q}^{\epsilon +}_{g,n+k}(X,
        \beta^\prime)|_{\tilde{\mathrm{gl}}^*_{r} \mathfrak D_{r-1}^\prime }})
      \cdot
      \mathcal O^{\mathrm{vir}}_{\tilde{Q}^{\epsilon_+}_{g, n+k+r}(X,
      \beta^{\prime \prime})} \otimes_{(IX)^k} \prod_{u=1}^r
    \mathcal O^{\mathrm{vir}}_{Q^{\epsilon_+}_{0,1}(X, \beta^\prime_{u})}
    \big)\\
    &\hspace{-0.85cm} \overset{\mathrm{Appendix}}{=} ((b\circ \tau) \times \ev)_*
     \sum_{\vec\beta^\prime} (\mathrm{pr}_1)_*\big(
      \mathcal O^{\mathrm{vir}}_{\tilde{Q}^{\epsilon_+}_{g, n+k+r}(X,
      \beta^{\prime \prime})} \otimes_{(IX)^k} \prod_{u=1}^r
    \mathcal O^{\mathrm{vir}}_{Q^{\epsilon_+}_{0,1}(X, \beta^\prime_{u})}
    \big)\\
    =&((b\circ \tau) \times \ev)_*
     \sum_{\vec\beta^\prime} (\underline{\mathrm{pr}}_1)_*\big(
      \mathcal O^{\mathrm{vir}}_{\tilde{Q}^{\epsilon_+}_{g, n+k+r}(X,
      \beta^{\prime \prime})} \otimes_{(\bar{I}X)^k} \prod_{u=1}^r
    \mathcal O^{\mathrm{vir}}_{Q^{\epsilon_+}_{0,1}(X, \beta^\prime_{u})}
    \big)\\
    = &((b\circ \tau) \times \ev)_*
    \sum_{\vec\beta^\prime}
    \big(\prod_{u=1}^{r}\ev_{n+k+u}^* (\check{\ev}_{1})_*(\mathcal
    O^{\mathrm{vir}}_{Q^{\epsilon_+}_{0,1}(X, \beta^\prime_{u})})
    \big)
    \cdot
    \mathcal O^{\mathrm{vir}}_{\tilde{Q}^{\epsilon_+}_{g, n+k+r}(X, \beta^{\prime \prime})}.
    \end{aligned}
  \end{equation*}
  Here $\underline{\mathrm{pr}}_1$ denotes the projection of ${\coprod_{\vec\beta^\prime}\tilde{Q}^{\epsilon_+}_{g, n+k+r}(X, \beta^{\prime
          \prime}) \times_{(\bar{I}X)^r} \prod_{u=1}^r Q^{\epsilon_+}_{0,1}(X,
        \beta^\prime_{u})}$ onto its first factors.

  By Lemma~\ref{special-evaluation}, we have
  \[
   (\check{\ev}_{1})_*(\mathcal
   O^{\mathrm{vir}}_{Q^{\epsilon_+}_{0,1}(X, \beta^\prime_{u})}) =
   \mathrm{Res}\big(
   (1
   - q^{\mathbf r}) I_{\beta^{\prime}_u}(q^{\mathbf r})
   \big).
 \]
 According to Lemma~\ref{elementary-properties-of-residues}, part (2), the above class also equals
 $\mathrm{Res}(f_{\vb'_{(u)}})$ after the change of variable $q\mapsto
 q\tilde{L}_{n+k+u}$.
 %\Ming{I just found out that we need to make the change of
  % variable here. Please double check.}
  Summarizing, we have
  \begin{equation*}
    \begin{aligned}
      & (\tau \times \ev)_* (\mathcal O(- \sum_{i=r}^{a-1} \mathfrak
      D_{i}^\prime) \cdot
      \mathcal O^{\mathrm{vir}}_{\tilde{Q}^{\epsilon +}_{g,n+k}(X, \beta^\prime)}|_{\mathfrak  D_{r-1}^\prime})\\
      = &
      \biggl[
      \sum_{\vec\beta^\prime}
      ((b\circ \tau) \times \ev)_*
      \biggl(
      \big(\prod_{u=1}^{r}\ev_{n+k+u}^*
      \mathrm{Res}\big(
     f_{\vb'_{(u)}}
      \big)
      \big)
      \cdot
      \mathcal O^{\mathrm{vir}}_{\tilde{Q}^{\epsilon_+}_{g, n+k+r}(X, \beta^{\prime \prime})}
      \cdot
      \mathcal
      O(-\sum_{i= 0}^{a - 1 - r}\mathfrak D_{i}^{\prime\prime})
      \biggr) \biggr]^{S_r}.
    \end{aligned}
  \end{equation*}
  Substituting these into \eqref{eq:divisor-sequence-recursive-1}, we have
  \begin{equation}
    \label{eq:overall-first-simplification}
    \begin{aligned}
      & (\tau \times \ev)_*
      \biggl(
      \mathcal O(- \sum_{i=0}^{a-1} \mathfrak D_i^\prime) \cdot
      \mathcal O^{\mathrm{vir}}_{\tilde{Q}^{\epsilon +}_{g,n+k}(X, \beta^\prime)}
      \biggr) \\
      = & (\tau \times \ev)_*
      \biggl(
      \mathcal O^{\mathrm{vir}}_{\tilde{Q}^{\epsilon +}_{g,n+k}(X, \beta^\prime)}
      \biggr) \\
      & -
      \sum_{r=1}^{a}
      \biggl[
      \sum_{\vec\beta^\prime}
      ((b\circ \tau) \times \ev)_*
      \biggl(
     \prod_{u=1}^{r}\ev_{n+k+u}^*
      \mathrm{Res}\big(
      f_{\vb^{\prime}_{(u)}}
      \big)
      \cdot
      \mathcal O^{\mathrm{vir}}_{\tilde{Q}^{\epsilon_+}_{g, n+k+r}(X, \beta^{\prime \prime})}
      \cdot
      \mathcal
      O(-\sum_{i= 0}^{a - 1 - r}\mathfrak D_{i}^{\prime\prime})
      \biggr)
      \biggr]^{S_r}.
    \end{aligned}
  \end{equation}
  When $a=1$, this is already the desired formula. In general we use induction.
  For $a > 1$, suppose that~\eqref{eq:simplication-1} is true with $a$ replaced
  by $a-r$,
  $\beta^{\prime}$ replaced by $\beta^{\prime\prime}$ and $n+k$ replaced by
  $n+k+r$, with $r\geq 1$, then we have
  \begin{equation*}
    \begin{aligned}
      & (\tau \times \ev)_*\big(
      \mathcal O^{\mathrm{vir}}_{\tilde{Q}^{\epsilon_+}_{g, n+k+r}(X, \beta^{\prime \prime})}
      \cdot
      \mathcal
      O(-\sum_{i = 0}^{a - 1 - r}\mathfrak D_{i}^{\prime\prime})
      \big) \\
      = &
      \sum_{i=0}^{a - r}
      \biggl[
      \sum_{\vec\beta^{\prime\prime}}
      ((b \circ \tau) \times \ev)_{*}
      \big(
      \prod_{v=1}^{i}
      \ev_{n+k+r+v}^{*}( - \mathrm{Res}(f_{\vb^{\prime\prime}_{(v)}})) \cdot
      \mathcal O^{\mathrm{vir}}_{\tilde{Q}^{\epsilon_+}_{g, n + k + r + i}(X, \beta^{\prime\prime\prime})}
      \big)
      \biggr]^{S_i}
    \end{aligned}
  \end{equation*}
  % when pushed forward to ${Q}^{\epsilon_-}_{g, n + k + r}(\mathbb P^{N-1},
  % d^{\prime\prime}) \times (IX)^{n + k + r}$,
  as $S_{n+k+r}$-equivariant
  $K_{\circ}$-classes on ${Q}^{\epsilon_-}_{g, n + k + r}(\mathbb P^{N-1},
  d^{\prime\prime}) \times (\bar{I}X)^{n + k + r} $. In particular, the equation holds as $S_{n+k} \times
  S_{r}$-equivariant $K_{\circ}$-classes. Note that the inner sum is over
  all ordered decompositions $\vb''=(\beta''',\beta_1'',\dots,\beta_i'')$ with $\vb''_{(v)}:=\beta_v''$ and the maps have new meaning:
  e.g.\ $\mathrm{ev}$ is the evaluation at the first $n+k+r$ markings.

  Tensoring the relation with $\prod_{u=1}^{r}\ev_{n+k+u}^*
  \mathrm{Res}\big(
  f_{\vb^{\prime}_{(u)}}
  \big)$ and taking the pushforward along
  \[
   {Q}^{\epsilon_-}_{g, n + k + r}(\mathbb P^{N-1},
   d^{\prime\prime}) \times (\bar{I}X)^{n + k + r}
   \rightarrow
   {Q}^{\epsilon_-}_{g, n + k}(\mathbb P^{N-1},
   d^{\prime\prime}) \times (\bar{I}X)^{n + k},
  \]
  which replaces the last $r$-markings by length-$d_0$ base points and forgets
  the last $r$ copies of $IX$, we obtain
  \begin{equation*}
    \begin{aligned}
         % \biggl[
      & \sum_{\vec\beta^\prime}
      ((b\circ \tau) \times \ev)_*
      \biggl(
     \prod_{u=1}^{r}\ev_{n+k+u}^*
      \mathrm{Res}\big(
     f_{\vb^{\prime}_{(u)}}
      \big)
      \cdot
      \mathcal O^{\mathrm{vir}}_{\tilde{Q}^{\epsilon_+}_{g, n+k+r}(X, \beta^{\prime \prime})}
      \cdot
      \mathcal
      O(-\sum_{i= 0}^{a - 1 - r}\mathfrak D_{i}^{\prime\prime})
      \biggr) \\
      = &
    \sum_{r=1}^{a}  \sum_{i=0}^{a - r}
      \biggl[
      \sum_{\vec\beta^\prime, \vec\beta^{\prime\prime}}
      ((b \circ \tau) \times \ev)_{*}
      \Big(
      \prod_{u=1}^{r}\ev_{n+k+u}^*
      \mathrm{Res}\big(
       f_{\vb^{\prime}_{(u)}}
      \big)
      \prod_{v=1}^{i}
      \ev_{n+k+r+v}^{*}( - \mathrm{Res}(f_{\vb^{\prime\prime}_{(v)}})) \cdot
      \mathcal O^{\mathrm{vir}}_{\tilde{Q}^{\epsilon_+}_{g, n + k + r + i}(X, \beta^{\prime\prime\prime})}
      \Big)
      \biggr]^{S_r\times S_i}
      % \biggr]^{S_r}.
    \end{aligned}
  \end{equation*}
  This is the second term in \eqref{eq:overall-first-simplification}. Hence
  we have
  \begin{equation*}
    \begin{aligned}
         & (\tau \times \ev)_*
      \biggl(
      \mathcal O(- \sum_{i=0}^{a-1} \mathfrak D_i^\prime) \cdot
      \mathcal O^{\mathrm{vir}}_{\tilde{Q}^{\epsilon +}_{g,n+k}(X, \beta^\prime)}
      \biggr) \\
      = &
(\tau \times \ev)_*
      \biggl(
      \mathcal O^{\mathrm{vir}}_{\tilde{Q}^{\epsilon +}_{g,n+k}(X, \beta^\prime)}
      \biggr) \\
       - & \sum_{r=1}^{a}  \sum_{i=0}^{a - r}
      \biggl[
      \sum_{\vec\beta^\prime, \vec\beta^{\prime\prime}}
      ((b \circ \tau) \times \ev)_{*}
      \Big(
      \prod_{u=1}^{r}\ev_{n+k+u}^*
      \mathrm{Res}\big(
      f_{\vb_{(u)}^\prime}
      \big)
      \prod_{v=1}^{i}
      \ev_{n+k+r+v}^{*}( - \mathrm{Res}(f_{\vb^{\prime\prime}_{(v)}})) \cdot
      \mathcal O^{\mathrm{vir}}_{\tilde{Q}^{\epsilon_+}_{g, n + k + r + i}(X, \beta^{\prime\prime\prime})}
      \Big)
      \biggr]^{S_r\times S_i} \\
      =  &
      \sum_{i=0}^{a} \biggl[ \sum_{\vec\beta^\prime} ((b \circ \tau) \times
      \ev)_{*} \Big( \prod_{v=1}^{i} \ev_{n+k+v}^{*}( -
      \mathrm{Res}(f_{\vb^\prime_{(v)}})) \cdot \mathcal
      O^{\mathrm{vir}}_{\tilde{Q}^{\epsilon_+}_{g, n + k + i}(X, \beta^{\prime
          \prime})} \Big) \biggr]^{S_i}.
    \end{aligned}
  \end{equation*}
  The last equality can be seen by expanding
  \[
    0 = \sum_{i=0}^a
    \biggl[
    \sum_{\vec\beta^\prime}
    ((b \circ \tau) \times \ev)_{*}
    \biggl(
    \prod_{v=1}^i(
    \ev_{n+k+v}^{*}(\mathrm{Res}(f_{\vb^\prime_{(v)}}) +
    \ev_{n+k+v}^{*}( - \mathrm{Res}(f_{\vb^\prime_{(v)}})
    )
    )
    \cdot \mathcal
    O^{\mathrm{vir}}_{\tilde{Q}^{\epsilon_+}_{g, n + k + i}(X, \beta^{\prime
        \prime})}
    \biggr)
    \biggr]^{S_i}.
  \]
  This finishes the proof of~\eqref{eq:simplication-1}.

 To prove~\eqref{eq:simplication-2}, we make similar simplifications as follows:
  \begin{equation*}
\begin{aligned}
      &  (\tau \times \ev)_*\Big(
      \mathcal {O}(-(s+1) \sum_{i = 0}^{m-k-1} \mathfrak
      D_i') \cdot \mathcal
      O^{\mathrm{vir}}_{\tilde{Q}^{\epsilon_+}_{g,n+k}(X,\beta^\prime)} \Big)
      -(\tau \times \ev)_*\Big(
      \mathcal {O}(-s \sum_{i = 0}^{m-k-1} \mathfrak
      D_i') \cdot \mathcal
      O^{\mathrm{vir}}_{\tilde{Q}^{\epsilon_+}_{g,n+k}(X,\beta^\prime)} \Big)
      \\
      = &(\tau \times \ev)_*\Big(
\mathcal {O}(-s \sum_{i = 0}^{m-k-1} \mathfrak
D_i')\cdot
\big(\mathcal {O}(- \sum_{i = 0}^{m-k-1} \mathfrak
D_i')-\ca{O}
\big)
\cdot \mathcal
      O^{\mathrm{vir}}_{\tilde{Q}^{\epsilon_+}_{g,n+k}(X,\beta^\prime)}
      \Big)\\
      = &-\sum_{r=1}^a(\tau \times \ev)_*\Big(
\big(\mathcal {O}(-s \sum_{i = 0}^{m-k-1} \mathfrak
D_i')
\cdot
\mathcal {O}(- \sum_{i = r}^{m-k-1} \mathfrak
D_i')
\cdot \mathcal
O^{\mathrm{vir}}_{\tilde{Q}^{\epsilon_+}_{g,n+k}(X,\beta^\prime)}
\big)|_{\mathfrak D_{r  - 1}^\prime}
\Big)
\end{aligned}
    \end{equation*}
By Lemma~\ref{pullback-sum-of-all-divisors}, we have
    \begin{equation*}
       \begin{aligned}
        & (\tau \times \ev)_*\Big(
\big(\mathcal {O}(-s \sum_{i = 0}^{m-k-1} \mathfrak
D_i')
\cdot
\mathcal {O}(- \sum_{i = r}^{m-k-1} \mathfrak
D_i')
\cdot \mathcal
O^{\mathrm{vir}}_{\tilde{Q}^{\epsilon_+}_{g,n+k}(X,\beta^\prime)}
\big)|_{\mathfrak D_{r  - 1}^\prime}
\Big)\\
      = &
      \biggl[
      ((b\circ \tau) \times \ev)_* \biggl(
      (\mathrm{pr}_1)_*p_*\Big( \mathcal
     O^{\mathrm{vir}}_{\tilde{Q}^{\epsilon +}_{g,n+k}(X,
        \beta^\prime)|_{\tilde{\mathrm{gl}}^*_{r} \mathfrak D_{r-1}^\prime }}
      \cdot
      \mathcal O_{{\tilde{\mathrm{gl}}_r^*} \mathfrak D_{r-1}'}
      \big(
      -s\sum_{j=0}^{r-2}D_j
      \big)
      \cdot \mathcal O_{ \tilde{\mathbb P}}(s)
      \Big)
      \cdot
      \mathcal
      O(-\sum_{i= 0}^{a - 1 - r}\mathfrak D_{i}^{\prime\prime})
      \biggr)
      \biggr]^{S_r},
    \end{aligned}
  \end{equation*}
  where the divisors $D_j,j=0,\dots,r-2$ are the tautological divisors of the
  inflated projective bundle $p$. It follows from
  Lemma~\ref{lem:int-with-D-k-1}, Lemma~\ref{lem:inflated-projective-bundle} and Lemma~\ref{special-evaluation} that
  \begin{align*}
 &(\mathrm{pr}_1)_*p_*\Big( \mathcal
     O^{\mathrm{vir}}_{\tilde{Q}^{\epsilon +}_{g,n+k}(X,
        \beta^\prime)|_{\tilde{\mathrm{gl}}^*_{r} \mathfrak D_{r-1}^\prime }}
      \cdot
      \mathcal O_{{\tilde{\mathrm{gl}}_r^*} \mathfrak D_{r-1}'}
      \big(
      -s\sum_{j=0}^{r-2}D_j
      \big)
      \cdot \mathcal O_{ \tilde{\mathbb P}}(s)
    \Big)\\
    =&
       -\sum_{\vec\beta^\prime}(\mathrm{pr}_1)_*\big(
    \sum_{
       \substack{j_1 + \cdots + j_r  = s \\j_u >0}}
    \prod_{u=1}^r(-\Theta_u^{-j_u})
      \cdot
      \mathcal O^{\mathrm{vir}}_{\tilde{Q}^{\epsilon_+}_{g, n+k+r}(X,
      \beta^{\prime \prime})} \otimes_{(IX)^k} \prod_{u=1}^r
    \mathcal O^{\mathrm{vir}}_{Q^{\epsilon_+}_{0,1}(X, \beta^\prime_{u})}
    \big)\\
    =&
       -\sum_{\vec\beta^\prime}
      \sum_{
       \substack{j_1 + \cdots + j_r  = s \\j_u >0}}
    \prod_{u=1}^r
    \ev_{n+k+u}^*(-q^{-j_u}f_{\vb'_{(u)}})
    \cdot
     \mathcal O^{\mathrm{vir}}_{\tilde{Q}^{\epsilon_+}_{g, n+k+r}(X,
      \beta^{\prime \prime})}.
  \end{align*}
  To summarize, we have  
\begin{align*}
 &  (\tau \times \ev)_*\Big(
      \mathcal {O}(-(s+1) \sum_{i = 0}^{m-k-1} \mathfrak
      D_i') \cdot \mathcal
      O^{\mathrm{vir}}_{\tilde{Q}^{\epsilon_+}_{g,n+k}(X,\beta^\prime)} \Big)
      -(\tau \times \ev)_*\Big(
      \mathcal {O}(-s \sum_{i = 0}^{m-k-1} \mathfrak
      D_i') \cdot \mathcal
   O^{\mathrm{vir}}_{\tilde{Q}^{\epsilon_+}_{g,n+k}(X,\beta^\prime)} \Big)\\
  =&\sum_{r=1}^{m-k} \biggl[
     \sum_{\vec\beta^\prime}
      \sum_{
       \substack{j_1 + \cdots + j_r  = s \\j_u>0}}
  ((b\circ \tau) \times \ev)_*
  \Big(
 \prod_{u=1}^r
  \ev_{n+k+u}^*(-q^{-j_u}f_{\vb'_{(u)}})
  \cdot
\mathcal O^{\mathrm{vir}}_{\tilde{Q}^{\epsilon_+}_{g, n+k+r}(X, \beta^{\prime \prime})}
      \cdot
      \mathcal
      O(-\sum_{i= 0}^{a - 1 - r}\mathfrak D_{i}^{\prime\prime})
  \Big)
  \biggr]^{S_r}\\
   =&\sum_{r=1}^{m-k}\sum_{i=0}^{m-k-r} \biggl[
     \sum_{\vec\beta^\prime,\vb''}
      \sum_{
       \substack{j_1 + \cdots + j_r  = s \\j_u >0}}
  ((b\circ \tau) \times \ev)_*
  \Big(
 \prod_{u=1}^r
  \ev_{n+k+u}^*(-q^{-j_u}f_{\vb'_{(u)}})
  \\
  &\hspace{7cm}\cdot
\prod_{v=1}^{i}
      \ev_{n+k+r+v}^{*}( - \mathrm{Res}(f_{\vb^{\prime\prime}_{(v)}}))
  \cdot
\mathcal O^{\mathrm{vir}}_{\tilde{Q}^{\epsilon_+}_{g, n+k+r}(X, \beta^{\prime \prime\prime})}
  \Big)
  \biggr]^{S_r \times S_i}\\
   =&\sum_{r=1}^{m-k} \biggl[
     \sum_{\vec\beta^\prime}
      \sum_{
       \substack{j_1 + \cdots + j_r  = s \\j_u \geq0}}
  ((b\circ \tau) \times \ev)_*
  \Big(
 \prod_{u=1}^r
  \ev_{n+k+u}^*(-q^{-j_u}f_{\vb'_{(u)}})
  \cdot
\mathcal O^{\mathrm{vir}}_{\tilde{Q}^{\epsilon_+}_{g, n+k+r}(X, \beta^{\prime \prime})}
  \Big)
  \biggr]^{S_r}
\end{align*}

The identity~\eqref{eq:simplication-3} can be proved similarly by using the relation
\[
 \mathcal O( \sum_{i=0}^{m-k-1} \mathfrak D_i^\prime) \cdot \mathcal F =
 \mathcal F
 +\sum_{r=0}^{m-k-1} \mathcal O( \sum_{i=0}^{r} \mathfrak
    D_{i}^\prime) \cdot \mathcal F|_{\mathfrak  D_{r}^\prime}
  \]
for any $\mathcal F\in
  K_\circ(\tilde{Q}^{\epsilon +}_{g,n+k}(X, \beta^\prime))$. We omit the details. 

\end{proof}

\begin{proof}[Proof of Proposition \ref{general-case}]
  We first evaluate the residues at 0.
  Consider the following formal expansion at $q = 0$:
  \[\frac{1}{1-q^{-1} \ca{O}(\sum_{i = 0}^{m-k-1} \mathfrak
      D_i')} = -q \ca{O} \big(- \sum_{i = 0}^{m-k-1} \mathfrak
    D_i' \big)- \sum_{s \geq1}q^{s + 1} \ca{O} \big(-(s + 1) \sum_{i = 0}^{m-k-1}
    \mathfrak
    D_i') \big) \big).
  \]
  Then by using~\eqref{eq:simplication-1} and~\eqref{eq:simplication-2}, we have
  \begin{align*}
    & \sum_{k = 1}^m \sum_{\vec{\beta}}
      \bigg\langle 
      \mathrm{Res}_{q=0} \bigg(\frac {\prod_{i = 1}^k   f_{\vb_{(i)}}
  }
{1 - q^{-1}
        \ca{O} (\sum_{i = 0}^{m-k-1} \mathfrak D_i')} \bigg)\frac{dq}{q}
      \bigg\rangle^{S_k}_k \\
    = & - \sum_{k = 1}^m \sum_{\vb} \Big \langle[\prod_{i = 1}^k
        (\fbi)_0]_{-1} \cdot \ca{O}(- \sum_{i = 0}^{m-k-1} \mathfrak
        D_i') \Big \rangle^{S_k}_k \\
    & - \sum_{k = 1}^m \sum_{\vb} \sum_{s \geq 1} \Big
      \langle[\prod_{i = 1}^k
      (\fbi)_0]_{-s-1} \cdot \ca{O}(-(s + 1) \sum_{i = 0}^{m-k-1}
      \mathfrak
      D_i')
      \Big \rangle^{S_k}_k \\
    = & - \sum_{k = 1}^m \sum_{\vb} \Big \langle[\prod_{i = 1}^k
        (\fbi)_0]_{-1} \Big \rangle^{S_k}_k \\
    & - \sum_{k = 2}^m \sum_{\vb,\vb'} \sum_{t \geq
      0} \sum_{r = 1}^{k-1} \sum_{\substack{j_1 + \cdots + j_r
      = t \\j_u \geq0}} \Big \langle[\prod_{i = 1}^{k-r}
    (\fbi)_0]_{-t-1} \cdot\prod_{u = 1}^r \Big(- \text{Res}
    \big(q^{-j_u}f_{\vb'_{(u)}} \big) \Big) \Big \rangle^{S_{k-r} \times
    S_r}_k \\
    & - \sum_{k = 1}^m \sum_{\vb} \sum_{s \geq 1} \Big
      \langle[\prod_{i = 1}^k
      (\fbi)_0]_{-s-1} \cdot \ca{O}(-s \sum_{i = 0}^{m-k-1}
      \mathfrak
      D_i')
      \Big \rangle^{S_k}_k.
  \end{align*}
Hence we obtain a recursive formula:
  \begin{align*}
    & - \sum_{k = 1}^m \sum_{\vb} \sum_{s \geq 1} \Big \langle[\prod_{i =
      1}^k
      (\fbi)_0]_{-s} \cdot\ca{O}(-s \sum_{i = 0}^{m-k-1} \mathfrak D_i')
      \Big \rangle^{S_k}_k\\
      &- \bigg(- \sum_{k = 1}^m \sum_{\vb} \sum_{s \geq
      1} \Big \langle[\prod_{i = 1}^k
      (\fbi)_0]_{-s-1} \cdot \ca{O}(-s \sum_{i = 0}^{m-k-1} \mathfrak
      D_i')
      \Big \rangle^{S_k}_k \bigg) \\
    =& - \sum_{k = 1}^m \sum_{\vb} \Big \langle[\prod_{i = 1}^k
        (\fbi)_0]_{-1} \Big \rangle^{S_k}_k \\
     & - \sum_{k = 2}^m \sum_{\vb,\vb'} \sum_{t \geq
      0} \sum_{r = 1}^{k-1} \sum_{\substack{j_1 + \cdots + j_r
      = t \\j_u \geq0}} \Big \langle[\prod_{i = 1}^{k-r}
    (\fbi)_0]_{-t-1} \cdot\prod_{u = 1}^r \Big(- \text{Res}
    \big(q^{-j_u}f_{\vb'_{(u)}} \big) \Big) \Big \rangle^{S_{k-r} \times
    S_r}_k.
  \end{align*}
  By induction, we have
  \begin{align*}
      & \sum_{k = 1}^m \sum_{\vec{\beta}}
      \bigg\langle 
      \mathrm{Res}_{q=0} \bigg(\frac {\prod_{i = 1}^k   f_{\vb_{(i)}}
  }
{1 - q^{-1}
        \ca{O} (\sum_{i = 0}^{m-k-1} \mathfrak D_i')} \bigg)\frac{dq}{q}
      \bigg\rangle^{S_k}_k \\
    = & - \sum_{k = 1}^m \sum_{\vb} \sum_{s \geq 1} \Big
        \langle[\prod_{i = 1}^k
        (\fbi)_0]_{-s} \cdot \ca{O}(-s \sum_{i = 0}^{m-k-1} \mathfrak
        D_i')
        \Big \rangle^{S_k}_k \\
    = & - \sum_{k = 1}^m \sum_{\vb} \sum_{s \geq1} \Big
        \langle[\prod_{i = 1}^k
        (\fbi)_0]_{-s} \Big \rangle^{S_k}_k \\
    & - \sum_{k = 2}^m \sum_{r = 1}^{m-k} \sum_{\vb,\vb'}
      \sum_{\substack{t \geq0 \\
    s \geq1}} \sum_{\substack{j_1 + \cdots + j_r
    = t \\j_u \geq0}} \Big \langle[\prod_{i = 1}^{k-r}
    (\fbi)_0]_{-t-s}\cdot \prod_{u = 1}^r \Big(- \text{Res}
    \big(q^{-j_u}f_{\vb'_{(u)}} \big) \Big) \Big \rangle^{S_{k-r}\times S_r}_k.
  \end{align*}

  Now we compute the residue at $q = \infty$. Set $w = 1/q$ and consider the
  following formal expansion at $w = 0$:
  \[
    \frac{1}{1-w \, \ca{O}(\sum_{i = 0}^{m-k-1} \mathfrak D_i')} = 1 + \sum_{t
      \geq0}w^{t + 1} \ca{O} \big(t \sum_{i = 0}^{m-k-1} \mathfrak D_i' \big)
    \cdot
    \ca{O} \big(\sum_{i = 0}^{m-k-1} \mathfrak D_i' \big).
  \]
  Then by using~\eqref{eq:simplication-3}, we have
  \begin{align*}
& \sum_{k = 1}^m \sum_{\vec{\beta}}
      \bigg\langle 
      \mathrm{Res}_{q=\infty} \bigg(\frac {\prod_{i = 1}^k   f_{\vb_{(i)}}
  }
{1 - q^{-1}
        \ca{O} (\sum_{i = 0}^{m-k-1} \mathfrak D_i')} \bigg)\frac{dq}{q}
      \bigg\rangle^{S_k}_k \\
    = -&\sum_{k = 1}^m \sum_{\vec{\beta}}
      \bigg\langle 
      \mathrm{Res}_{w=0} \bigg(\frac {\prod_{i = 1}^k   f_{\vb_{(i)}}
  }
{1 - q^{-1}
        \ca{O} (\sum_{i = 0}^{m-k-1} \mathfrak D_i')} \bigg)\frac{dw}{w}
      \bigg\rangle^{S_k}_k \\
    = - & \sum_{k = 1}^m \sum_{\vb} \Big \langle[\prod_{i = 1}^k
          (\fbi)_{\infty}]_{0} \Big \rangle^{S_k}_k \\
    -&  \sum_{k = 1}^m \sum_{\vb} \sum_{t \geq 0} \Big
      \langle[\prod_{i = 1}^k (\fbi)_{\infty}]_{-t-1} \cdot \ca{O}(\sum_{i =
      0}^{m-k-1} \mathfrak D_i') \cdot \ca{O}(t \sum_{i = 0}^{m-k-1} \mathfrak
      D_i')
      \Big \rangle^{S_k}_k \\
    = - & \sum_{k = 1}^m \sum_{\vb} \Big \langle[\prod_{i = 1}^k
          (\fbi)_{\infty}]_{0} \Big \rangle^{S_k}_k \\
    - & \sum_{k = 2}^m \sum_{t \geq 0} \sum_{r = 1}^{k-1} \sum_{\vb,\vb'}
        \sum_{\substack{j_1 + \cdots + j_r = t + 1 \\j_u>0}} \Big \langle[\prod_{i =
    1}^{k-r} (\fbi)_{\infty}]_{-t-1}\cdot \prod_{u = 1}^r \text{Res}
    \big(q^{j_u}f_{\vb'_{(u)}} \big) \Big \rangle^{S_{k-r} \times S_r}_k \\
    - & \sum_{k = 1}^m \sum_{t \geq 0} \sum_{\vb} \Big
        \langle[\prod_{i = 1}^{k} (\fbi)_{\infty}]_{-t-1}
        \cdot
        \ca{O}(t
        \sum_{i = 0}^{m-k-1} \mathfrak D_i')
        \Big \rangle^{S_k}_k.
  \end{align*}
 Similar to
  the previous case, we obtain a recursive formula:
  \begin{align*}
    - & \sum_{k = 1}^m \sum_{\vec{\beta}} \sum_{t \geq 0} \Big
        \langle[\prod_{i = 1}^k (\fbi)_{\infty}]_{-t} \cdot
        \ca{O}(t \sum_{i
        = 0}^{m-k-1} \mathfrak D_i') \Big \rangle^{S_k}_k\\
    &    - \bigg(- \sum_{k = 1}^m
        \sum_{\vb} \sum_{t \geq 0} \Big \langle[\prod_{i = 1}^k
        (f_{\beta_i})_{\infty}]_{-t-1}
        \cdot \ca{O}(t \sum_{i = 0}^{m-k-1} \mathfrak
        D_i') \Big \rangle^{S_k}_k \bigg) \\
    = - & \sum_{k = 1}^m \sum_{\vb} \Big \langle[\prod_{i = 1}^k
          (\fbi)_{\infty}]_{0} \Big \rangle^{S_k}_k \\
    - & \sum_{k = 1}^m \sum_{t \geq 0} \sum_{r = 1}^{k-1} \sum_{\vb,\vb'}
        \sum_{\substack{j_1 + \cdots + j_r = t + 1 \\j_u>0}} \Big \langle[\prod_{i =
    1}^{k-r} (\fbi)_{\infty}]_{-t-1}\cdot \prod_{u = 1}^r \text{Res}
    \big(q^{j_u}f_{\vb'_{(u)}} \big) \Big \rangle^{S_{k-r}\times S_r }_k.
  \end{align*}
  By induction, we have
  \begin{align*}& \sum_{k = 1}^m \sum_{\vec{\beta}}
      \bigg\langle 
      \mathrm{Res}_{q=\infty} \bigg(\frac {\prod_{i = 1}^k   f_{\vb_{(i)}}
  }
{1 - q^{-1}
        \ca{O} (\sum_{i = 0}^{m-k-1} \mathfrak D_i')} \bigg)\frac{dq}{q}
      \bigg\rangle^{S_k}_k \\
    = - & \sum_{k = 1}^m \sum_{\vb} \sum_{t \geq 0} \Big \langle[\prod_{i =
          1}^k (\fbi)_{\infty}]_{-t} \cdot \ca{O}(t \sum_{i = 0}^{m-k-1}
          \mathfrak D_i')
          \Big \rangle^{S_{k}}_k \\
    = - & \sum_{k = 1}^m \sum_{\vb} \sum_{t \geq 0} \Big
          \langle[\prod_{i = 1}^k (\fbi)_{\infty}]_{-t} \Big \rangle^{S_{k}}_k \\
    - & \sum_{k = 2}^m \sum_{r = 1}^{m-k} \sum_{\vb,\vb'} \sum_{\substack{t
        \geq0 \\ s \geq1}} \sum_{\substack{j_1 + \cdots + j_r = t + 1 \\j_u>0}} \Big
    \langle[\prod_{i = 1}^{k-r} (\fbi)_{\infty}]_{-t-s}\cdot \prod_{u = 1}^r
    \text{Res} \big(q^{j_u}f_{\vb'_{(u)}} \big) \Big \rangle^{S_{k-r} \times
    S_r}_k \\
  \end{align*}
  This concludes the proof of Proposition \ref{general-case}
\end{proof}

Next, we want to simplify the right side of (\ref{eq:general-case}) and
show that it equals
\begin{equation} \label{eq:equation-cor}
  \begin{aligned}
    & - \sum_{k = 1}^m \sum_{\vb} \chi \bigg([ \tilde{Q}^{\epsilon_ + }_{g,
      n + k}(X, \beta^\prime)/S_{k}],
    \prod_{i = 1}^k \ev_{n+i}^*\Big(- \mathrm{Res}  \Big(\frac {\fbi}{1-q^{
        -1}} \Big) \Big)\cdot\ovir_{[ \tilde{Q}^{\epsilon_ + }_{g,
      n + k}(X, \beta^\prime)/S_{k}]} \bigg) \\
    = & - \sum_{k = 1}^m \sum_{\vb} \Big \langle \prod_{i = 1}^k
    \text{Lau}(\fbi) \Big \rangle^{S_{k}}_k,
  \end{aligned}
\end{equation}
where $$\text{Lau}(\fbi): = - \mathrm{Res} \big(\fbi/(1-q^{-1})
\big) = \sum_{s \geq1}[(\fbi)_0]_{-s} + \sum_{t
  \geq0}[(\fbi)_\infty]_{-t}$$ is the $K$-theory class obtained by
evaluating the Laurent polynomial part of $\fbi$ at $q = 1$, i.e.,
$\text{Lau}(\fbi) = [\fbi]_ + |_{q = 1}$.

For the fixed curve class $\beta$, we denote by $D_{\beta}$ the set of all
degree-$d_0$ curve classes that appears in the decompositions of $\beta$
satisfying Condition \ref{cond:condition-on-beta}. Recall that $ \fbi=   (1-q^{\mathbf{r}}L_{n +
    i})I_{\beta_i}(q^{\mb{r}}L_{n + i})$. We observe that the dependence of
  both~\eqref{eq:general-case} and~\eqref{eq:equation-cor} on the $I$-function
  is through the coefficients of the formal Laurent series expansions of
  $(1-q)I_{\gamma}(q),\gamma\in D_{\beta}$ at 0 and $\infty$. Hence it suffices to prove the identity by viewing 
  those expansions as independent series
 which can be defined for any tuple
  $(\underline{g_0},\underline{g_\infty})$, where
  $\underline{g_0}:=(g^0_\gamma(q))_{\gamma\in D_\beta}$ and
  $\underline{g_\infty}:=(g^\infty_\gamma(q))_{\gamma\in D_\beta}$ are any Laurent
  series indexed by classes in $D_\gamma$. More
  precisely, given such a tuple, we set $(\fbi)_0=g^0_{\vb_{(i)}}(q^{\bf r}L_{n+i})$ and
  $(\fbi)_\infty=g^\infty_{\vb_{(i)}}(q^{\bf r}L_{n+i})$ and define the right
  sides of~\eqref{eq:general-case} and~\eqref{eq:equation-cor}. Denote
the universal expressions of the right side of (\ref{eq:general-case}) and (\ref{eq:equation-cor}) by
$\text{Loc}(\underline{g_0}, \underline{g_\infty})$ and
$\text{Cor}(\underline{g_0}, \underline{g_\infty})$, respectively.
We prove the following
\begin{lemma} \label{combinatorial-identity}
  We have
  \[
    \mathrm{Loc}(\underline{g_0}, \underline{g_\infty}) = \mathrm{Cor}(\underline{g_0},
    \underline{g_\infty}) \]
  for any tuples $(\underline{g_0}, \underline{g_\infty})$ of formal Laurent
  series indexed by $D_\beta$.
\end{lemma}
\begin{proof}
For any $\gamma\in D_\beta$, we choose an integer $l_\gamma\geq 0$ and a class $\delta_{\gamma}^0 \in K(\bar{I}X) \otimes \Lambda$. Consider the formal directional derivative
  $\nabla_{\sum_{\gamma \in D_\beta}
    \delta_{\gamma}^0q^{l_{\gamma}}}$ of the difference
  $\text{Loc}(\underline{g_0}, \underline{g_\infty})- \text{Cor}(\underline{g_0},
  \underline{g_\infty})$ in the direction
  \begin{equation}\label{eq:increment}
    \underline{g_0} = \big (g_{\gamma}^{0} \big) \mapsto \underline{g_0} +
    \sum_{\gamma \in D_\beta} \delta_{\gamma}^0q^{l_{\gamma}}: =
    \big(g_{\gamma}^{0} + \delta_{\gamma}^0 q^{l_{\gamma}} \big).
  \end{equation}
  More precisely, for any function $F$ depending on $\underline{g_0}$ and
  $\underline{g_\infty}$, we define
  \[
    \nabla_{\sum_{\gamma \in D_\beta}
      \delta_{\gamma}^0q^{l_{\gamma}}}F(\underline{g_0}, \underline{g_\infty}) =
    \lim_{h \to 0}(F(\underline{g_0} + h \sum_{\gamma \in D_\beta}
    \delta_{\gamma}^0q^{l_{\gamma}}, \underline{g_\infty})-
    F(\underline{g_0},
    \underline{g_\infty}))/h.
  \]
 
By definition, for a decomposition $\vb$, we have  $(\fbi)_0=g^0_{\vb_{(i)}}(q^{\mathbf{r}}L_{n+i})$ and
  $(\fbi)_\infty=g^\infty_{\vb_{(i)}}(q^{\bf r}L_{n+i})$. Hence the
  change~\eqref{eq:increment} induces
  \[
(\fbi)_0\mapsto(\fbi)_0+\delta_{\vb_{(i)}}fd^0L_{n+i}^{l_{\vb_{(i)}}}q^{l_{\vb_{(i)}}\mathbf{r}}
\]
and keeps $(\fbi)_\infty$ the same.
  Under the assumption that $l_{\gamma} \geq 0$, it is easy to see that
  $\nabla_{\sum_{\gamma} \delta_{\gamma}^0q^{l_{\gamma}}}
  \text{Lau}(f_{\vb_{(i)}}) = 0$ and
  $\nabla_{\sum_{\gamma} \delta_{\gamma}^0q^{l_{\gamma}}} \text{Res}
  \big(q^{j}(f_{\vb_{(i)}}) \big) = \delta_{\vb_{(i)}}^0L_{n+i}^{l_{\vb_{(i)}}}$ for $j = -l_{\vb_{(i)}}\mathbf{r}$ and
  0 otherwise. Given a class $\gamma\in D_\beta$, we set
  $\beta^\gamma:=\beta-\gamma$ and denote by $\vec{\beta}^\gamma$ an ordered
  decomposition of $\beta^\gamma$ satisfying
  Condition~\ref{cond:condition-on-beta-ordered}. We will use similar notation
  for $\beta'$. By the Lebnitz rule stated in Lemma \ref{lem-derivative}, we obtain
  \[
    \nabla_{\sum_{\gamma} \delta_{\gamma}^0q^{l_{\gamma}}} \,
    \mathrm{Cor}(\underline{g_0}, \underline{g_\infty}) = 0. \]
When we differentiate the right side of~\eqref{eq:general-case}, only the first
and third summations contribute non-zero terms. Using the Lebnitz rule in Lemma
\ref{lem-derivative}, we have 
  \begin{equation} \label{eq:direction-0}
    \begin{aligned}
      & \nabla_{\sum_{\gamma} \delta_{\gamma}^0q^{l_{\gamma}}} \,
      \text{Loc}(\underline{g_0}, \underline{g_\infty}) \\
      = &
      - \sum_{k = 2}^m
      \sum_{\gamma \in D_{\beta}} \sum_{\vec{\beta}^\gamma}
      \sum_{s \geq1}
      \Big \langle
      \ev_1^*(\delta_{\gamma}^0)L^{l_\gamma}_{n+1}
      \cdot
      [ \prod_{i = 1}^{k-1} (f_{\vb^\gamma_{(i)}})_0 ]_{-s-l_{\gamma} \mathbf{r}}
      \Big \rangle^{S_1\times S_{k-1}}_k \\
      & + \sum_{k = 2}^{m}
      \sum_{\gamma \in D_{\beta}} \sum_{\vb^\gamma}
      \sum_{s \geq1}
      \Big \langle
      \ev_1^*(\delta_{\gamma}^0)L^{l_\gamma}_{n+1}
      \cdot
      [ \prod_{i = 1}^{k-1} (f_{\vb^\gamma_{(i)}})_0]_{-s-l_\gamma\mathbf{r}}
      \Big \rangle^{S_1\times S_{k-1}}_k \\
      & - \sum_{k = 3}^{m} \sum_{r = 1}^{k-2}
      \sum_{\gamma \in D_{\beta}} \sum_{\vb^\gamma,\vb'}
      \sum_{\substack{t \geq0 \\ s \geq1}} \sum_{\substack{j_1 + \cdots + j_r =
          t \\j_u \geq0}}
      \Big \langle
      \ev_1^*(\delta_{\gamma}^0)L^{l_\gamma}_{n+1}
      \cdot
      [ \prod_{i = 1}^{k-r-1} (f_{\vb^\gamma_{(i)}})_0]_{-t-s-l_\gamma\mathbf{r}}\\
      &\hspace{8.5cm}
      \cdot
      \prod_{u = 1}^r \Big(- \text{Res} \big(q^{-j_u}f_{\vb'_{( u)}}
      \big) \Big)
      \Big \rangle^{S_1 \times S_{k-r-1}\times S_r}_k \\
      & + \sum_{k = 3}^{m} \sum_{r = 2}^{k-1}
      \sum_{\gamma \in D_{\beta}} \sum_{\vb,\vb'^{\gamma}}
      \sum_{\substack{t \geq0 \\ s \geq1}} \sum_{\substack{j_1 + \cdots +
          \cdots + j_{r-1} = t-l_\gamma \mathbf{r} \\j_u \geq0}}
      \Big \langle
 \ev_1^*(\delta_{\gamma}^0)L^{l_\gamma}_{n+1}
\cdot
[\prod_{i = 1}^{k-r} (f_{\vb_{(i)}})_0]_{-t-s}\\
&\hspace{9cm}
      \cdot
      \prod_{u = 1}^{r-1} \Big(- \text{Res}
      \big(q^{-j_u}f_{\vb'^{\gamma}_{(u)}} \big) \Big)
      \Big \rangle^{S_{1} \times S_{k-r}\times S_{r-1}}_k.
    \end{aligned}
  \end{equation}
   The above class is zero because the first line of (\ref{eq:direction-0}) cancels
  with the second line and the third line cancels
  with the fourth line after replacing $r$ by $r-1$ and $t$ by
  $t-l_\gamma\mathbf{r}$. Therefore, we have
  \[
    \nabla_{\sum_{\gamma} \delta_{\gamma}^0q^{l_{\gamma}}} \,
    \big(\text{Loc}(\underline{g_0}, \underline{g_\infty})-
    \text{Cor}(\underline{g_0}, \underline{g_\infty}) \big) = 0.
  \]

  Let  $\nabla_{\sum_{\gamma \in D_\beta} \delta_{\gamma}^\infty
    q^{l_{\gamma}}}$  denote the formal directional derivative in the direction
  \[
     \underline{g_\infty} = \big (g_{\gamma}^{\infty} \big) \mapsto \underline{g_\infty} +
    \sum_{\gamma \in D_\beta} \delta_{\gamma}^\infty q^{l_{\gamma}'}: =
    \big(g_{\gamma}^{\infty} + \delta_{\gamma}^\infty q^{l_{\gamma}'} \big),
  \]
  where
  $l_{\gamma}'> 0$ and $\delta_{\gamma}^\infty \in K(\bar{I}X) \otimes \Lambda$.
  Using similar computations as in the previous case, we can show that
  \[
    \nabla_{\sum_{\gamma} \delta_{\gamma}^\infty q^{l_{\gamma}'}} \,
    \big(\text{Loc}(\underline{g_0}, \underline{g_\infty})-
    \text{Cor}(\underline{g_0}, \underline{g_\infty}) \big) = 0.
  \]
  We omit the proof.

  The vanishing of formal directional derivatives implies that the difference
  $\text{Loc}(\underline{g_0}, \underline{g_\infty})- \text{Cor}(\underline{g_0},
  \underline{g_\infty})$ does not depend on the coefficients of the nonnegative
  powers of $q$ in $g_\gamma^0$ and the positive powers of $q$ in
  $g_{\gamma}^\infty$. Hence we can set
  \[
    [g_\gamma^0]_{l} = 0, \ l \geq 0 \quad \text{and} \quad
    [g_\gamma^\infty]_{l'} = 0, \ l'> 0.
  \]

  Under this additional assumption, we have
  \[
    - \text{Res}(q^{-j}f_{\vb'_{(u)}}) = [(f_{\vb'_{(u)}})_\infty]_{-j}, \quad \quad
    \text{Res}(q^{j'}f_{\vb'_{(u)}}) = [(f_{\vb'_{(u)}})_0]_{j'}
  \]
  for any $j \geq0, j'>0$. Hence
  \begin{align*}
    \sum_{\substack{j_1 + \cdots + j_r = t \\j_u \geq0}} \,
    \prod_{u = 1}^r \Big(- \text{Res} \big(q^{-j_u}f_{\vb'_{(u)}} \big)
    \Big)
    & = [ \prod_{u = 1}^{r}  (f_{\vb'_{(u)}})_\infty]_{-t}, \\
    \sum_{\substack{j_1 + \cdots + j_r = t + 1 \\j_u>0}} \
    \prod_{u = 1}^r \text{Res} \big(q^{j_u}f_{\vb'_{(u)}} \big)
    & = [ \prod_{u = 1}^{r}  (f_{\vb'_{(u)}})_0]_{t + 1}.
  \end{align*}
  Using the above identities, we can simplify the difference
  $\text{Loc}(\underline{f_0}, \underline{f_\infty})- \text{Cor}(\underline{f_0},
  \underline{f_\infty})$ as follows:
  \begin{equation*}
    \begin{aligned}
      & \text{Loc}(\underline{g_0}, \underline{g_\infty})-
      \text{Cor}(\underline{g_0}, \underline{g_\infty}) \\
      = & - \sum_{k = 1}^m \sum_{\vec{\beta}} \sum_{s \geq 1} \Big
      \langle[ \prod_{i = 1}^k (\fbi)_0 ]_{-s}
      \Big \rangle^{S_k}_k
      - \sum_{k = 1}^m \sum_{\vec{\beta},\vb'}
      \sum_{t
        \geq 0}
      \Big \langle [\prod_{i = 1}^k
      (\fbi)_\infty]_{-t}
      \Big
      \rangle^{S_k}_k \\
      & - \sum_{k = 2}^{m} \sum_{r = 1}^{k-1}
      \sum_{\vec{\beta},\vb'}
      \bigg(
      \sum_{\substack{t \geq0 \\ s \geq1}}
      \Big \langle
      [ \prod_{i = 1}^{k-r} (\fbi)_0]_{-t-s}
      \cdot
      [ \prod_{u = 1}^{r}  (f_{\vb'_{(u)}})_\infty]_{-t}
      \Big \rangle^{S_{k-r}\times S_r}_k \\
      & \hspace{1.65cm} + \sum_{\substack{t \geq 0 \\ s \geq 1}}
      \Big \langle
      [\prod_{i = 1}^{k-r}  (\fbi)_\infty]_{-t-s}
      \cdot
      [ \prod_{u = 1}^{r}  (f_{\vb'_{(u)}})_0]_{t + 1}
      \Big \rangle^{S_{k-r}\times S_r}_k
      \bigg) \\
      & + \sum_{k = 1}^m \sum_{\vb} \Big \langle
      \prod_{i = 1}^k \big(\sum_{s \geq1}[(\fbi)_0]_{-s} + \sum_{t
        \geq0}[(\fbi)_\infty]_{-t} \big)
      \Big \rangle^{S_k}_k \\
      = & - \sum_{k = 1}^m
      \sum_{\vec{\beta}}
      \sum_{s \geq 1}
      \Big \langle
      [ \prod_{i = 1}^k (\fbi)_0 ]_{-s}
      \Big \rangle^{S_{k}}_k
      - \sum_{k = 1}^m \sum_{\vec{\beta}}
      \sum_{t \geq 0}
      \Big \langle [\prod_{i = 1}^k  (\fbi)_\infty]_{-t} \Big
      \rangle^{S_k}_k \\
      & - \sum_{k = 2}^{m} \sum_{r = 1}^{k-1} \sum_{\vec{\beta},\vb'}
      \sum_{\substack{t \geq0 \\ s \geq1}}
      \Big \langle
      [ \prod_{i = 1}^{k-r} (\fbi)_0]_{-s}
      \cdot
      [ \prod_{u = 1}^{r}  (f_{\vb'_{(u)}})_\infty]_{-t}
      \Big \rangle^{S_{k-r}\times S_r }_k \\
      & + \sum_{k = 1}^m \sum_{\vb} \Big \langle
      \prod_{i = 1}^k \big(\sum_{s \geq1}[(\fbi)_0]_{-s}
      + \sum_{t \geq0}[(\fbi)_\infty]_{-t} \big)
      \Big \rangle^{S_{k}}_k \\
      = 0 \end{aligned}
  \end{equation*}
  The last equality follows from the permutation-equivariant multinomial formula.
\end{proof}
\begin{lemma}[Permutation-equivariant Leibniz rule] \label{lem-derivative}
  Let $T_i, 1 \leq i \leq m$ be linear maps from the ring of formal Laurent
  series $[K(\bar{I}X) \otimes \Lambda]((q))$ to its coefficient ring $K(\bar{I}X)
  \otimes \Lambda$. Let $(k_1, \dots, k_m)$ be a partition of a positive integer
  $n$ and let $f_1, \dots f_m$ be formal Laurent series. For $l_i \in \bb{Z},
  \delta_i \in K(\bar{I}X) \otimes \Lambda$, consider the directional derivative
  $\nabla_{\sum \delta_i q^{l_i}}$ along the direction
  \[
    (f_1, \dots, f_m) \mapsto(f_1 + \delta_1q^{l_1}, \dots, f_m + \delta_mq^{l_m}).
  \]
  Then we have
  \begin{align*}
    & \nabla_{\sum \delta_i q^{l_i}}
      \Big \langle
      T_1(f_1^{\otimes k_1}
      \cdot
     T_2(f_2)^{\otimes k_2}
      \cdots
      T_m(f_m)^{\otimes k_m}
      \Big \rangle_{n}^{S_{k_1} \times \dots \times S_{k_m}} \\
    = & \sum_{a = 1}^m
        \Big \langle
        \nabla_{\delta_a q^{l_a}}T_a(f_a)
        \cdot
         T_1(f_1)^{\otimes k_1}
        \cdots
        T_a(f_a)^{\otimes k_a-1}
        \cdots
         T_m(f_m)^{\otimes k_m}
 \Big \rangle_{n}^{S_{k_1} \times \dots \times S_{k_a-1} \times \dots \times S_{k_m}}.
  \end{align*}
Here we adopt the bracket notation from which the pullbacks $\ev^*_i$ are
suppressed. In particular, the equality still holds if we make the change of
variable $q\mapsto q^{\bf r}L$, where $\mathbf{r}\in \Z$ and $L$ denotes the coarse cotangent line bundles at the
corresponding markings, and extend the definition of $T_i$ via
linearity: $T_i (q^{a}L^b \alpha) : = L^bT_i(q^{a}\alpha)$ with $a, b \in \mathbb Z, \alpha \in
    K(\bar{I}X)\otimes\Lambda$.
\end{lemma}
\begin{proof}
  Let $h \in \bb{C}$.
  By the permutation-equivariant binomial formula, we have
  \begin{align*}
    & \big \langle T_1(f_1 + h \delta_1 q^{l_1})^{\otimes k_1}
           \cdots
      T_m(f_m + h \delta_m q^{l_m})^{\otimes k_m}
      \big \rangle_{n}^{S_{k_1} \times \dots \times S_{k_m}} \\
    = & \Big \langle (T_1(f_1) + h \cdot T_1(\delta_1 q^{l_1}))^{\otimes k_1}
        \cdots
        (T_m(f_m) + h \cdot T_m(\delta_m q^{l_m}))^{\otimes k_m}
        \Big \rangle_{n}^{S_{k_1} \times \dots \times S_{k_m}} \\
    = & \Big \langle T_1(f_1)^{\otimes k_1}
\cdots
        T_m(f_m)^{\otimes k_m}
        \Big \rangle_{n}^{S_{k_1} \times \dots \times S_{k_m}} \\
    + & h \sum_{a = 1}^m \big \langle
        T_a(\delta_a q^{l_a})
        \cdot
        T_1(f_1)^{\otimes k_1}
        \cdots
        T_a(f_a)^{\otimes (k_a-1)}
        \cdots
        T_m(f_m)^{\otimes k_m}
        \Big \rangle_{n}^{S_{k_1} \times \dots \times S_{k_a-1} \times \dots
        \times S_{k_m}} + O(h^2).
  \end{align*}
  Hence the lemma follows from the definition of the directional derivative
  $\nabla_{\sum_a \delta_a q^{l_a}}$ and the equality $\nabla_{\delta_a
    q^{l_a}}T_a(f_a) = T_a(\delta_a q^{l_a})$.
\end{proof}

\begin{proof}[Proof of Proposition~\ref{prop-general-case-simplified}]
  Recall that $[(1-q)I_{\beta}(q)]_ + = \mu_{\beta}(q)$ for $\beta>0$. Hence we
have
\[
  \text{Lau} \big((1-q^{\mathbf{r}}L)I_{\beta}(q^{\mb{r}}L) \big) =
  [(1-q^{\mathbf{r}}L)I_{\beta}(q^{\mb{r}}L)]_ + |_{q = 1} = \mu_{\beta}(L).
\]
We conclude the proof by combining Proposition~\ref{general-case}, Lemma~\ref{combinatorial-identity}, Lemma~\ref{lem:vir-str-sheaf-comp-entangled-tails}, and the above equality.
\end{proof}

By substituting the expression in Proposition~\ref{prop-general-case-simplified}
into the $K$-theoretic localization formula
(\ref{eq:residue-localization-main-case}) over the master space, we obtain the
main theorem of this paper:
\begin{theorem} \label{all-genus-single-wall-crossing}
  Assuming that $2g-2 + n + \epsilon_0 \emph{deg}(\beta)>0$, we have
  \begin{align*}
     &(\iota\times \ev)_*\ca{O}^{\emph{vir}}_{Q^{\epsilon_-}_{g,n}(X, \beta)}
   -((c\circ \iota)\times\ev )_*\ca{O}^{\emph{vir}}_{Q^{\epsilon_+}_{g,n}(X, \beta)}\\
   =&\sum_{k\geq 1}\bigg(\sum_{\vec{\beta}}((b_{k}\circ c\circ \iota )\times \ev)_*
   \bigg(
     \prod_{a=1}^k\ev_{n+a}^*\,\mu_{\beta_a}(L_{n+a})
     \cdot
     \ca{O}^{\emph{vir}}_{Q^{\epsilon_+}_{g,n+k}(X, \beta')}
   \bigg)\bigg)
  \end{align*}
   in $K_0\big([(Q^{\epsilon_-}_{g,n}(\bb{P}^N, d)\times (\bar{I}X)^n\big)/S_n]\big)_{\Q}$, where $\vec{\beta}$ runs through
  all ordered decompositions of $\beta$ satisfying Condition
  \ref{cond:condition-on-beta-ordered}. The same formula also holds
  for the
  twisted virtual structure sheaves with level structure.

\end{theorem}

\subsection{The genus-0 case} \label{The-genus-0-case}
In genus zero, the following theorem generalizes the genus-zero toric
mirror theorems in quantum $K$-theory
\cite{Givental-Tonita, givental14, givental15} and quantum $K$-theory with
level
structure \cite{RZ1}, and the $K$-theoretic wall-crossing formula in
\cite{Tseng-You}.
\begin{theorem} \label{thm:genus-0-case}
  For any $\epsilon$, we have
  \[
    J^{\infty}_{S_\infty}(\mb{t}(q) + \mu^{\geq \epsilon}(Q, L), Q) = J^{\epsilon}_{S_\infty}(\mb{t}(q),
    Q).
  \]
The same formula holds for twisted permutation-equivariant $J$-functions with level structure.

\end{theorem}
\begin{proof}
  It follows from Corollary \ref{potential-wall-crossing} that the above equality holds modulo the constant terms in $\mb{t}$. Let
  $\epsilon_-< \epsilon_0 = \frac{1}{d_0}< \epsilon_ + $ and let
  \[
    \mu^{\epsilon_0}(Q, q) = \sum_{\text{deg}(\beta) = 1/ \epsilon_0}
    \mu_\beta(q)Q^\beta.
  \]
  It suffices to prove that
  \begin{equation} \label{eq:genus 0-single wall}
    J_{S_\infty}^{\epsilon_ + }(\mu^{\epsilon_0}(Q, q), Q) = J_{S_\infty}^{\epsilon_-}(0, Q).
  \end{equation}
  The right hand side of (\ref{eq:genus 0-single wall}) is equal to
  \[
    1-q + (1-q)\sum_{0 \leq \text{deg}(\beta) \leq1/ \epsilon_0}I_\beta(q)Q^\beta +
    \sum_{\text{deg}(\beta)>1/ \epsilon_0}Q^\beta
    (\check{\ev}_{1})_{*} \bigg (
    \frac{\ovir_{Q^{\epsilon_-}_{0,1}(X,\beta)}}{1-qL_1}
    \bigg ).
  \]
  According to Theorem \ref{all-genus-single-wall-crossing}, we have
  \begin{align*}
    &\sum_{\text{deg}(\beta)>1/ \epsilon_0}Q^\beta
      (\check{\ev}_{1})_{*}
       \bigg (
    \frac{\ovir_{Q^{\epsilon_-}_{0,1}(X,\beta)}}{1-qL_1}
    \bigg )
     \\
    = & \sum_{\substack{(k = 0, \, \text{deg}(\beta)>1/ \epsilon_0) \,
        \text{or} \\(k \geq1, \beta \geq 0), \, (k, \beta) \neq(1, 0)}}Q^\beta
    (\check{\ev}_{1})_{*}
    \bigg (
    \frac{\ovir_{[Q^{\epsilon_+}_{0,1+k}(X,\beta)/S_k]}}{1-qL_1}\cdot\prod_{i=1}^k\ev_{i+1}^*( \mu^{\epsilon_0}(Q, L_{i+1}))
    \bigg )
      \end{align*}
  By comparing to the left hand side of (\ref{eq:genus 0-single wall}), we see
  that (\ref{eq:genus 0-single wall}) is equivalent to
  \[
    (1-q)\sum_{\text{deg}(\beta) = 1/ \epsilon_0}I_\beta(q)Q^\beta =
    \sum_{\text{deg}(\beta) = 1/ \epsilon_0}Q^\beta
     (\check{\ev}_{1})_{*} \bigg (
    \frac{\ovir_{Q^{\epsilon_+}_{0,1}(X,\beta)}}{1-qL_1}
    \bigg )
    +
    \mu^{\epsilon_0}(Q, q).
  \]
  This follows immediately from Corollary \ref{evaluation-last-wall}.

\end{proof}

\appendix

\section{$K$-theoretic pushforward formula for inflated projection bundles}
  \begin{lemma}
    \label{lem:inflated-projective-bundle}
    Let $X$ be any Deligne--Mumford stack with an $S_k$-action, and
    $\Theta_1, \ldots, \Theta_k$ be an $S_k$-equivariant tuple of line
    bundles on $X$. Let
    $p: \tilde{\mathbb P} \to X$ be the inflated projective bundle associated to
    $\Theta_1 , \ldots, \Theta_k$. For $i = 1, \cdots, k-1$, let
    $D_i$ be the $i$-tautological divisor.
    Then for any $a \in K_\circ(X/S_k)$ and any integer $t \geq 0$, we have
    \begin{equation*}
      \begin{aligned}
        p_* \Big( [\mathcal O(t(D_0 + \cdots + D_{k-2}))
        ] & \cdot [\mathcal O_{\tilde{\mathbb P}}(-t)] \cdot p^*a \Big) \\
        & =    [\mathrm{Sym}^t(\Theta_1 \oplus \cdots \oplus \Theta_k)] \cdot a \\
        & =
        \big(\sum_{\substack{j_1 + \cdots + j_k = t \\j_i \geq0}} [\Theta_1^{j_1}
        \otimes  \cdots \otimes  \Theta_k^{j_k}] \big) \cdot a
      \end{aligned}
    \end{equation*}
    and for any integer $t< 0$, we have
    \begin{equation*}
      \begin{aligned}
        p_* \Big( [\mathcal O(t(D_0 + \cdots + D_{k-2}))] & \cdot [\mathcal
        O_{\tilde{\mathbb P}}(-t)] \cdot p^*a \Big) \\
        & = -   \big(\sum_{\substack{j_1 + \cdots + j_k = -t \\j_i>0}} [(-
        \Theta_1^{-j_1} ) \otimes  \cdots \otimes  (- \Theta_k^{-j_k})] \big) \cdot a
      \end{aligned}
    \end{equation*}
  \end{lemma}

  We will prove this lemma at the end of this section. First, we need some
  preparation. Set
  \[
    L = \mathcal O(D_0 + \cdots + D_{k-2}) \otimes  \mathcal O(-1).
  \]
  Note that we have canonical isomorphism
  \[
    \mathcal O(D_0 + 2D_1 + \cdots + (k-1)D_{k-2}) \cong \mathcal O(k) \otimes
    \Theta_1
    \otimes  \cdots \Theta_k,
  \]
  whose proof is similar to that in the original version of Lemma~2.6.3 of
  [Zhou].
  Thus we have
  \[
    L \cong \mathcal O(k-1) \otimes \mathcal O(-(D_1 + \cdots + (k-2)D_{k-2}))
    \otimes \Theta_1 \otimes  \cdots \otimes  \Theta_k.
  \]
  \begin{lemma}
    \label{lem:minimal-polynomial}
    In $K_0(\tilde{\mathbb P}/S_k)$ we have
    \begin{equation}
      \label{eq:relation-K-theory}
      (L + (- \Theta_1)) \cdots (L + (- \Theta_k)) = 0.
    \end{equation}
  \end{lemma}
  \begin{proof}[Proof of Lemma~ \ref{lem:minimal-polynomial}]
    Let
    \[
      x_i: \mathcal O(-1)  \rightarrow  \Theta_i, \quad i = 1 , \ldots, k
    \]
    be the tautological morphisms over $\bb{P}(\Theta_1 \oplus \dots \oplus
    \Theta_k)$. Then we have
    \[
      x_1 \cdots \hat{x}_i \cdots x_k: \mathcal O(-(k-1)) \rightarrow
      \Theta_1
      \cdots \hat{\Theta}_i \cdots \Theta_k.
    \]
    This is equivalent to
    \[
      \varphi_i : \Theta_i \rightarrow \mathcal O(k-1) \otimes
      \Theta_1 \cdots \Theta_k.
    \]
    Let
    \[
      \textstyle
      H_i = \mathbb P(\bigoplus_{j \neq i} \Theta_i), \quad  i = 1 , \ldots, k
    \]
    be the coordinate hyperplanes.

    Let $b: \tilde{\bb{P}} \rightarrow \bb{P}(\Theta_1 \oplus \dots \oplus
    \Theta_k)$ be the the composition of the sequence of blowups. The vanishing
    locus of $b^* \varphi_i$ defines a Cartier divisor supported on
    $D_0 \cup \cdots \cup D_{k-2}$. Along any irreducible component $D$ of some
    $D_j$,
    the vanishing order of $b^* \varphi_i$ is equal
    \[
      \# \{a \mid b(D) \subset H_a,  1 \leq a \leq k, a \neq i \}.
    \]
    Note that for each irreducible component $D$ of $D_j$, there are precisely
    $(j + 1)$ of $H_1 , \ldots, H_k$ that contain $p(D)$.
    Hence,
    \begin{enumerate}
    \item
      for $i = 1 , \ldots, k$, $b^* \varphi_i$ factors through
      \begin{equation*}
        d_i : \Theta_i \rightarrow  L.
      \end{equation*}
    \item
      at each point $p$ of $\tilde {\mathbb P}$, there exists $i$ such that
      $d_i$ is an isomorphism near $p$.
    \end{enumerate}
    This implies that we have a $S_k$-equivariant Koszul resolution
    \begin{equation*}
      0 \rightarrow \bigwedge^k \big((\Theta_1 \oplus \dots \oplus \Theta_k)
      \otimes L^{-1} \big) \rightarrow \dots \rightarrow \bigwedge^1 \big((\Theta_1
      \oplus \dots \oplus \Theta_k) \otimes L^{-1} \big) \rightarrow \ca{O}
      \rightarrow 0.
    \end{equation*}
    The relation \eqref{eq:relation-K-theory} is obtained by taking the $K$-theory
    class of the above long exact sequence.
  \end{proof}

  \begin{lemma}
    \label{lem:vanising-higher-cohomology}
    For $i \geq 1$, $t \geq 0$, $R^ip_*L^{\otimes t} = 0$.
  \end{lemma}
  \begin{proof}
    It suffices to prove the Lemma in the case when $X$ is a point. We will
    describe $\tilde{\mathbb P}$ as a toric variety and show that $L$ is nef.

    Let $N = \mathbb Z^{k}/ \Delta$, where $\Delta$ is the small diagonal. Then
    $\mathbb P^{k-1}$ is the toric variety associated to the fan $\Sigma_{k-1}$
    whose rays are
    spanned by the image in $N$ of standard basis $e_i$ of $\mathbb Z^k$. We also
    denote them by $e_i$. Thus a typical maximal dimensional cone $\sigma_{k-1}
    \in
    \Sigma_{k-1}$ is spanned by
    \[
      e_1 , \ldots, e_{k-1}.
    \]
    When we say some object is typical, we mean that the other objects are
    obtained via $S_k$-symmetry. Whenever we subdivide some typical cone, we do
    the
    same thing to other cone using $S_k$-symmetry.

    We claim that $\tilde {\mathbb P}$ is the toric variety
    associated to the fan $\Sigma_1$, a typical maximal dimensional cone
    $\sigma_1$ of which
    is spanned by
    \[
      (1^{a}, 0^{k-a}) , \quad a = 1 , \ldots, k-1.
    \]
    Here and below we use $1^a$ or $0^a$ to respectively denote $1$ or $0$
    repeated $a$ times.

    Indeed, blowing up along intersection of the proper transform of $H_1 ,
    \ldots,
    H_{\ell}$ amounts to star-subdividing the fan by adding the ray
    \[
      (1^\ell, 0^{k- \ell}).
    \]
    The first blowup $U_{k-2} \to U_{k-1}$ corresponds to the fan
    $\Sigma_{k-2}$ obtained by star-subdivide
    the $\Sigma_{k-1}$ by adding the ray spanned by
    \[
      (1^{k-1}, 0)
    \]
    and also the rays in it $S_k$-orbits.
    Thus a typical maximal dimensional cone $\sigma_{k-2} \in \Sigma_{k-2}$ is
    spanned by
    \[
      e_1 , \ldots, e_{k-2}, (1^{k-1}, 0).
    \]
    Rays in the $S_k$-orbits of $(1^{k-1}, 0)$ correspond to the exceptional
    divisor.
    Inductively, we define $U_{\ell} = X_{\Sigma_{\ell}}$, a typical maximal
    dimensional
    cone of which is spanned by
    \[
      e_1 , \ldots, e_{\ell}, (1^{\ell + 1}, 0^{k- \ell-1}) , \ldots, (1^{k-1},
      0).
    \]
    The blowup $U_{\ell-1} \to U_{\ell}$ replaces $e_{\ell}$ by
    $(1^{\ell}, 0^{k- \ell})$. This proves the claim.

    We now compute the divisors. For a ray $\rho = (\rho_1, \dots, \rho_k)$ in
    $\Sigma_1$, we write $| \rho| = \sum_{j = 1}^k
    \rho_j$. Then
    \[
      D_i = \sum_{| \rho| = i + 1} D_\rho.
    \]
    And $L$ is linearly equivariant to the $\mathbb Q$-Cartier divisor     \[
      \sum_{i = 0}^{k-2} \frac{k-i-1}{k}D_i = \sum_{\rho} \frac{k-|
        \rho|}{k}D_\rho.
    \]
    We choose the (unique) Cartier data
    \[
      m = (\frac{1-k}{k}, \frac{1}{k} , \ldots, \frac{1}{k}) \in M = N^{\vee}
    \]
    in $\sigma_1 \in \Sigma_1$. The Cartier data in other cones are determined by
    $S_k$-symmetry. Note that all maximal cones that contains $e_1$
    share the same $m$. There are $(k-1)!$ of them. Let $\varphi$ be the support
    function determined by $S_k$-symmetry and the formula
    \[
      \varphi(v) = \langle {m, v} \rangle, \quad v \in \sigma_1.
    \]

    It is easy to see that $\varphi$ is non-positive and $\{v \mid \varphi(v)
    \geq
    -1/k \}$ is the $k-1$-simplex whose vertices are
    \[
      -e_1 , \ldots, -e_k \in N = \mathbb Z^{k}/ \Delta.
    \]
    Hence the support function is convex.\footnote{Note that the definition of
      convexity in \cite{cox2011toric} is different from the usual convention.}
    Hence $L$ is nef. By the Demazure vanishing theorem, $L^{\otimes t}$ has no
    higher cohomology for $t \geq 0$.
  \end{proof}

  Recall that we have constructed
  \[
    d_i : \Theta_i \rightarrow  L.
  \]
  Note that here $\Theta_i$ actually means $p^* \Theta_i$. We have $p_*p^*
  \Theta_i
  = \Theta_i$.
  Using adjunction we have
  \[
    \textstyle
    d: \bigoplus_{i = 1}^k \Theta_i \rightarrow  p_*(L).
  \]
  \begin{lemma}
    \label{lem:basic-pushforward}
    For $t \geq 1$, the composition
    % The map $d$ above is an isomorphism. Moreover,
    % the natural morphism
    \[
      \textstyle
      d^{t} : \mathrm{Sym}^{t}  \bigoplus_{i = 1}^k \Theta_i
      \overset{\mathrm{Sym^t}(d)}{\rightarrow}
      \mathrm{Sym}^{t} \big( p_*(L) \big)
      {\rightarrow}
      p_*(L^{\otimes t})
    \]
    are isomorphisms for $t \geq 1$.
  \end{lemma}
  \begin{proof}
    Again it suffices to prove statements when $X$ is a point. Consider the graded ideal
    \[
      \mathfrak a \subset \mathbb C[x_1 , \ldots, x_k]
    \]
    generated by monomials
    \[
      x_1 \cdots \hat{x_i} \cdots x_k, \quad i = 1 , \ldots, k.
    \]
    Thus the subscheme
    \[
      V_{\mathbb P^{k-1}}(    \mathfrak a) \subset \mathbb P^{k-1}
    \]
    is the codimension-$2$ toric strata.

    Let $[    \mathfrak a^t]_m$ denote the degree $m$ part of $\mathfrak a^t$,
    for any $t, m \geq
    1$. It is easy to see that $\mathfrak a^t$ is generated by monomials
    \[
      x^a : = x_1^{a_1} \cdots x_k^{a_k}, \quad a = (a_1 , \ldots, a_k)
    \]
    such that $\sum_i a_i = (k-1)t$ and $0 \leq a_i \leq t$ for $i = 1 , \ldots,
    k$. We
    rewrite it as
    \[
      \frac{x^{t}}{x^{b}},
    \]
    where $t = (t, \dots, t)$ and $b = (b_1 , \ldots, b_k)$ with $\sum_i b_i = t$
    and $b_i \geq  0, i = 1 , \ldots, t$.
    Thus we have
    \[
      [\mathfrak a^{t}]_{t(k-1)} = \mathrm{Sym}^{t} [\mathfrak a]_{k-1}.
    \]

    Then the map $d^{t}$ is identified with the $d^{t}$ in
    \[
      \begin{tikzcd}
        {[\mathfrak a^t]_{t(k-1)}} \arrow[r, dashed, "d^{t}"] \arrow[d, hook] &
        H^0(\tilde{\mathbb P},
        L^{ \otimes t}) \arrow[d, hook] \\
        H^0(\mathbb P^{k-1}, \mathcal O_{\mathbb P^{k-1}}(t(k-1))) \arrow[r, "
        \cong"] &
        H^0(\tilde{\mathbb P}, \mathcal O_{\mathbb P^{k-1}}(t(k-1)))
      \end{tikzcd}.
    \]
    Let $\tilde{\mathfrak a} \subset \mathcal O_{\mathbb P^{k-1}}$ be the ideal
    sheaf
    associated to $\mathfrak a$. It suffices to show that images of both vertical
    arrows are
    $H^0(\mathbb P^{k-1}, \tilde{\mathfrak a}^{t} \otimes \mathcal O_{\mathbb
      P^{k-1}}(t(k-1)))$.
    For the left vertical arrow, it
    amounts to check that for any homogeneous
    $f \in \mathbb C[x_1 , \ldots, x_k]$ of degree $t(k-1)$, any integer $N>0$
    such that
    \[
      x_i^Nf \in \mathfrak a^t, \quad i = 1 , \ldots, k,
    \]
    we have $f \in \mathfrak a^t$. Since $\mathfrak a^t$ is generated by
    monomials, it suffices
    to prove this for $f = x_1^{a_1} \cdots x_k^{a_k}$, $\sum_{i} a_i =
    t(k-1)$. Then $x_i^Nf \in \mathfrak a^t$ reads
    \[
      x_1^{a_1} \cdots x_i^{a_i + N} \cdots x_k^{a_k} = (x_1^{a^\prime_1} \cdots
      t
      x_i^{a^\prime_i} \cdots x_k^{a^\prime_k}) \cdot (x_1^{c_1} \cdots
      x_i^{c_i} \cdots x_k^{c_k})
    \]
    for some $\sum_j a_j^\prime = t(k-1)$, $0 \leq a^\prime_j \leq t$, $c_j \geq
    0$.
    Since $a_i^\prime \leq t$, $\sum_{j \neq i}a_i^\prime \geq t(k-2)$. Hence
    $\sum_{j \neq i} a_j \geq t(k-2)$. Hence $a_i \leq t$. Now we have shown
    that $a_i \leq t$ for all $i = 1 , \ldots, k$. We conclude that $f \in
    \mathfrak a^{t}$.

    We now consider the right vertical arrow. Suppose $f \in
    H^0(\mathbb P^{k-1}, \tilde{\mathfrak a}^{t} \otimes \mathcal O_{\mathbb
      P^{k-1}}(t(k-1)))$, then
    by the previous paragraph it comes from $[\mathfrak a^t]_{t(k-1)}$. Hence $f$
    lies
    in the image of the right vertical arrow by the commutativity of the diagram.
    Now suppose $f \in H^0(\mathbb P^{k-1}, \mathcal O_{\mathbb
      P^{k-1}}(t(k-1)))$ whose image in $H^0(\tilde{\mathbb P}, \mathcal
    O_{\mathbb P^{k-1}}(t(k-1)))$ vanishes along $t(D_1 + \cdots +
    (k-2)D_{k-2})$. We want to show that $f$ is in $[\mathfrak a^t]_{t(k-1)}$.
    Since the arrow is torus equivariant, without loss of generality we may
    assume
    that $f = x_1^{a_1} \cdots x_k^{a_k}$, $\sum_i a_i = t(k-1)$.
    Since the pullback of $f$ vanishes along $t(k-2)D_{k-2}$, we must have
    \[
      a_1 + \cdots + \hat a_{i} + \cdots + a_k \geq t(k-2)
    \]
    for each $i = 1 , \ldots, k$. Thus we must have $a_i \leq t$. Hence $f \in
    [\mathfrak a^t]_{t(k-1)}$. This completes the proof.
    % \Yang{Should I use something like $\mathfrak a$ instead of $I$, to avoid
    % confusion with the $I$-function?}
    % One can check directly that the ideal $I^\ell$ is
    % saturated, It is easy to see that image of the left vertical arrow is
    % $H^0(\mathbb P^{k-1}, \mathcal I^{\ell})$. Indeed, this follows from the
    % fact
    % that the scheme theoretic closure of
  \end{proof}
  \begin{proof}[Proof of Lemma~ \ref{lem:inflated-projective-bundle}]
    The $t \geq 0$ case follows immediately from Lemma~
    \ref{lem:vanising-higher-cohomology} and
    Lemma~ \ref{lem:basic-pushforward}.
    We focus on the $t<0$ case. Recall that by Lemma~
    \ref{lem:minimal-polynomial}, we have
    \[
      (L + (- \Theta_1)) \cdots (L + (- \Theta_k)) = 0.
    \]
    From the above equality, we can deduce
    \[
      (1 - \lambda^{-1}L)(\sum_{i = 0}^{k-1} \lambda^{-i} \alpha_i) = (1 +
      \lambda^{-1}(- \Theta_1)) \cdots
      (1 + \lambda^{-1}(- \Theta_k)),
    \]
    where $\lambda$ is a formal variable and
    \[
      \alpha_i = \sum_{j = 0}^{i}L^{\otimes j} \otimes \mathrm{Sym}^{i-j} \big(
      (- \Theta_1) \oplus \cdots \oplus (- \Theta_k) \big).
    \]

    Thus we have
    \[
      \frac{1}{1- \lambda^{-1}L} = \frac{\sum_{i = 0}^{k-1} \lambda^{-i} \alpha_i}{(1
        + \lambda^{-1}(- \Theta_1)) \cdots
        (1 + \lambda^{-1}(- \Theta_k))}.
    \]
    The formal expansions of both sides of the above equation around $\lambda = 0$
    gives that the following identity
    \[
      \sum_{j \geq1} \lambda^j(L^{-1})^{\otimes j} = - \big(\sum_{i = 0}^{k-1}
      \lambda^{-i} \alpha_i \big) \cdot \big(\sum_{j \geq1} \lambda^j
      \sum_{\substack{j_1 + \cdots + j_k = j \\j_i>0}} (- \Theta_1^{-j_1} ) \otimes
      \cdots \otimes  (- \Theta_k^{-j_k}) \big)
    \]
    in $K^0(X)[[\lambda]]$.
    By
    Lemma~ \ref{lem:vanising-higher-cohomology} and
    Lemma~ \ref{lem:basic-pushforward}, we have
    \[
      p_*(L^{\otimes j}) = \mathrm{Sym}^j \big(
      \Theta_1 \oplus \cdots \oplus \Theta_k \big)
    \]
    for $j \geq0$.
    Hence $p_*(\alpha_i) = 0$ for $1 \leq i \leq k-1 $ and we have
    \[
      \sum_{j \geq1} \lambda^jp_*(L^{-1})^{\otimes j} = - \sum_{j \geq1} \lambda^j
      \sum_{\substack{j_1 + \cdots + j_k = j \\j_i>0}} (- \Theta_1^{-j_1} ) \otimes
      \cdots \otimes  (- \Theta_k^{-j_k}).
    \]
    We conclude the proof of the lemma by comparing the coefficients of $\lambda$
    in the above equation.
  \end{proof}

\bibliographystyle{amsplain.bst}
\bibliography{ref}
\end{document}